\documentclass{IEEEtran}[journal]%[lettersize,journal]
\usepackage{graphicx}
\usepackage{amsmath, amssymb}
\usepackage{amsfonts}
\usepackage{mathtools}

\usepackage{color}
\usepackage{soul}

\renewcommand{\hl}{}
\newcommand{\hlrow}[0]{\rowcolor{yellow}}
\renewcommand{\hlrow}{}
\usepackage{amsthm}

\theoremstyle{definition}
\newtheorem{theorem}{Theorem}
\newtheorem*{theorem*}{Theorem}
\newtheorem{definition}[theorem]{Definition}
\newtheorem*{definition*}{Definition}

\newtheorem*{corollary*}{Corollary}
\newtheorem{example}{Example}
\newtheorem*{example*}{Example}
\newtheorem{proposition}{Proposition}
\newtheorem*{proposition*}{Proposition}
\newtheorem{lemma}[theorem]{Lemma}
\newtheorem*{lemma*}{Lemma}

\newtheorem*{conjecture*}{Conjecture}
\newtheorem{remark}[theorem]{Remark}
\newtheorem*{remark*}{Remark}

\usepackage{algorithm}%, algorithmicx}
\usepackage[noend]{algpseudocode}
\algrenewcommand\algorithmicrequire{\textbf{Input:}}
\algrenewcommand\algorithmicensure{\textbf{Output:}}

\ifCLASSOPTIONcompsoc
    \usepackage[caption=false, font=normalsize, labelfont=sf, textfont=sf]{subfig}
\else
\usepackage[caption=false, font=footnotesize]{subfig}
\fi

\usepackage{booktabs}
\usepackage{multirow}
\usepackage{lscape}

\usepackage{enumitem}

\usepackage{comment}

\newcommand{\bR}{\mathbb R}

\makeatletter
\def\subsubsection{%
  \@startsection
    {subsubsection}                 % type
    {3}                             % level
    {\parindent}                    % indent
    {3.5ex plus 1.5ex minus 1.5ex}  % beforeskip {0ex plus 0.1ex minus 0.1ex}
    {0.7ex plus .5ex minus 0ex}     % afterskip {0ex}
    {\normalfont\normalsize\itshape}% style
}
\makeatother

\title{\hl{
Linearizing Binary Optimization Problems\\
Using Variable Posets for Ising Machines
}}
\author{
    \IEEEauthorblockN{
        Kentaro~Ohno\IEEEauthorrefmark{1}\IEEEauthorrefmark{2}$^a$, 
        Nozomu~Togawa\IEEEauthorrefmark{2}$^b$,~\IEEEmembership{Member,~IEEE}
    }
    \thanks{%\IEEEauthorblockA{
            \IEEEauthorrefmark{1}NTT, Tokyo 180-8585, Japan %\\\hspace*{.4em}
            \IEEEauthorrefmark{2}Department of Computer Science and Communications Engineering, Waseda University, Tokyo 169-8555, Japan \\%\hspace*{.4em}
            Email: $^a$kentaro.ohno@togawa.cs.waseda.ac.jp,
            $^b$ntogawa@waseda.jp
        }
    }
\begin{document}

\maketitle

\begin{abstract}
    Ising machines are next-generation computers expected to efficiently sample near-optimal solutions of combinatorial optimization problems.
    Combinatorial optimization problems are
    modeled as quadratic unconstrained binary optimization (QUBO) problems
    to apply an Ising machine.
    However, current state-of-the-art Ising machines still often fail to output near-optimal solutions due to the complicated energy landscape of QUBO problems.
    Furthermore, the physical implementation of Ising machines severely restricts the size of QUBO problems to be input as a result of limited hardware graph structures.
    In this study, 
    we take a new approach to these challenges by injecting auxiliary penalties preserving the optimum,
    which reduces quadratic terms in QUBO objective functions.
    The process simultaneously simplifies the energy landscape of QUBO problems, allowing the search for near-optimal solutions, and makes QUBO problems sparser, facilitating encoding into Ising machines with restriction on the hardware graph structure. 
    %We propose partially linearizing QUBO by ordering binary variables.
    %This method has the effect of simplifying the energy landscape of QUBO without changing the optimal solution, allowing the search for a near-optimal solution, and facilitating encoding into an Ising machine with topological constraints by eliminating the quadratic terms in QUBO.
    We propose \hl{linearization of QUBO problems using variable posets} as an outcome of the approach.
    By applying the proposed method to synthetic QUBO instances and to multi-dimensional knapsack problems, 
    we empirically 
    validate %its effectiveness 
    the effects on enhancing minor-embedding of QUBO problems and the performance of Ising machines.
\end{abstract}

\section{Introduction}\label{sec:introduction}
Combinatorial optimization is an important and well-studied research area %for both theory and practice 
with abundant applications in the field of operations research.
For example, the knapsack problem and its variants are %those of the most 
famous combinatorial optimization problems with %plenty of 
applications including production planning, resource-allocation and portfolio selection~\cite{cacchiani2022knapsack}.
%Theoretically, 
Combinatorial optimization problems are often hard to deal with for traditional von Neumann-type computers due to their NP-hardness.
Various heuristics and meta-heuristics have been developed for handling large-scale combinatorial optimization problems.

Ising machines are attracting interests as a next-generation computing paradigm, especially for tackling hard combinatorial optimization problems~\cite{mohseni2022ising}.
Ising machines find %quasi-optimal
heuristic solutions of combinatorial optimization problems in a class called quadratic unconstrained binary optimization (QUBO).
There are several types of Ising machines, including quantum annealing machines~\cite{johnson2011quantum,houck2012chip}, coherent Ising machines~\cite{wang2013coherent,honjo2021100} and specialized-circuit-based digital machines~\cite{aramon2019physics,yamaoka201520k,wang2021solving,yamamoto2020statica,goto2019combinatorial}, depending on the way they are physically implemented.

When utilizing an Ising machine, a combinatorial optimization problem is converted to a QUBO problem~\cite{lucas2014ising}.
Discrete variables are represented via binary variables and constraints are encoded into the objective function as penalties.
The total objective function should be quadratic so that the resulting model is indeed a QUBO problem.
After conversion to the QUBO problem,
binary variables are assigned to physical spins in the Ising machine.
The Ising machine is then executed to sample a solution by minimizing \emph{energy}, that is, the value of the objective function.
Performance of Ising machines solving a QUBO problem typically depends on the energy landscape~\cite{mohseni2022ising,stella2005optimization} of the objective function.
It is also subject to a graph structure associated with the QUBO problem, in which nodes and edges correspond respectively to variables and quadratic terms with non-zero coefficients, from a combinatorial point of view.

There are two major challenges for utilizing Ising machines.
One is that Ising machines suffer from finding near-optimal solutions when the QUBO problem involves a complex energy landscape.
Ising machines often output solutions with a large energy gap (e.g., a more than 10\% optimality gap) to an optimal solution of the problem~\cite{huang2022benchmarking}, on which a simple heuristic might achieve smaller gap. %(e.g., less than several percent optimality gap).
The issue casts doubt on the practical utility of Ising machines %viewed 
as meta-heuristic solvers and the gap needs to be filled. 
The other is that some major Ising machines have physical restriction on the structure of QUBO problems.
For example, a densely connected QUBO problem cannot be input directly to quantum annealing machines, since quantum bits on the machines corresponding to variables in the problem interact only with a limited group of other quantum bits~\cite{Dwaveadvantage,Dwaveadvantage2} in accordance with specific hardware graphs.
This severely limits the potential applicability of Ising machines to combinatorial optimization problems.

For the former challenge, attempts have been made to reduce complexity of the energy landscape. For example, merging several variables into one variable have been proposed to improve the quality of solutions for single-spin-flip-based Ising machines by deforming the energy landscape~\cite{shirai2022multi}. However, the process might change the optimum and thus requires iterations of applying an Ising machine to vary a set of merging variables.
For the latter challenge, %a technique called 
minor-embedding~\cite{choi2008minor,choi2011minor} has been proposed to embed a QUBO problem with an arbitrary graph structure in a hardware graph.
That is, a variable in the QUBO problem is represented by a \emph{chain}, which is a set of possibly multiple connected physical variables.
However, a dense QUBO problem still requires a huge number of auxiliary variables to form chains, which degrades the performance of Ising machines~\cite{Dwaveadvantage,yamamoto2020statica}.
Moreover, such a requirement restricts the size of QUBO problems that can be input to a small one.

In this study, we take a new approach to tackle the issues by considering auxiliary constraint conditions for QUBO.
That is, we extract conditions that the optimal solution must satisfy, consider them as constraint conditions, and add them to the QUBO objective function as penalties.
By taking appropriate penalty terms and their coefficients, the quadratic terms in the objective function are reduced without changing the optimum.
The above process should %is expected to 
have two effects.
First, the auxiliary penalties simplify the energy landscape associated with the QUBO problem, allowing an efficient search for a near-optimal solution via Ising machines. 
Second, the graph associated with the QUBO problem is made sparse by reducing quadratic terms, thereby reducing the number of variables added by minor-embedding.
Reduction of the number of additional variables for minor-embedding should result in alleviating the degradation of Ising machine performance and improving the embeddability of large QUBO problems.
Therefore, establishing such an approach is expected to mitigate the aforementioned issues on Ising machines.

We propose \hl{\emph{linearization using variable posets (partially ordered sets)} for QUBO}, a method realizing the above process by considering ordinal conditions of precedence %$x_i=1\Rightarrow x_j=1$ 
on binary variables as auxiliary constraints.
The proposed method consists of two steps.
Step~1 is \hl{\emph{extracting variable posets}}, that is, to find a partial order of variables from a given optimization problem so that corresponding constraints preserve the optimum.
By considering a sufficient condition for the requirement, we develop a general strategy and fast algorithm for extracting a valid partial order.
Step~2 is \emph{linearization} of the QUBO problem with auxiliary penalties corresponding to the extracted order.
Overall, the proposed method is expected to mitigate the %aforementioned 
issues for Ising machines with feasible computational complexity.

We conduct experiments on synthetic QUBO problems and multi-dimensional knapsack problems (MKPs) to validate the effects of the proposed method.
The results show that  
the proposed method effectively mitigates the defects of minor-embedding of introducing additional variables and substantially improves Ising machine performance on the benchmark MKP instances. %, which validate our proposed method.

Our contribution is summarized as follows:
{
\setlength{\leftmargini}{10pt}
\begin{itemize}
    \item We propose a method of linearization of QUBO problems that improves minor-embedding and Ising machine performance while preserving the optimum of the %original 
    problem.
    \item We provide solid theoretical results and efficient algorithms for the proposed method with practical applications.
    \item We validate the effects of the proposed method on improving 
    minor-embedding and Ising machine performance through comprehensive experiments on synthetic QUBO problems and MKPs.
\end{itemize}
}

The rest of the paper is organized as follows.
Background on QUBO and Ising machines are explained in Section~\ref{sec:preliminaries}.
The notion of ordering variables is introduced and %related 
theoretical results are shown in Section~\ref{sec:theory}.
We explain the proposed method in Section~\ref{sec:proposed_method}.
Application of the method to practical problems are discussed in Section~\ref{sec:application}.
Experimental results are given in Section~\ref{sec:experiments}.
We summarize related work and further discussion in Section~\ref{sec:discussion}.
Section~\ref{sec:conclusion} concludes this paper.

\section{QUBO and Ising Machines}\label{sec:preliminaries}

We %briefly 
review related notions on Ising machines.
Throughout the paper, $n$ is a positive integer denoting the problem size, and $B_n \coloneqq \{0,1\}^n$ denotes a space of binary vectors.

\subsection{Quadratic Unconstrained Binary Optimization}

\emph{Quadratic unconstrained binary optimization (QUBO)} is a class of optimization problems over binary variables defined by a square matrix $Q \in \mathbb R^{n\times n}$ as follows:
\begin{align}
    &{\rm minimize \ } \phi(x) \coloneqq x^{\top} Q x \\
    &{\rm subject \ to \ } x \in B_n.
\end{align}
We also refer to the value $\phi(x)$ as \emph{energy} of $x$.
A QUBO problem is naturally associated with an undirected graph whose nodes and edges correspond to variables and non-zero off-diagonal entries of $Q$, respectively.
Various combinatorial optimization problems can be modeled as QUBO problems~\cite{lucas2014ising}.
When there are constraints on binary variables, \emph{penalty terms} are introduced to represent those constraints in a QUBO form.
If binary variables are constrained to a subset $C\subset B_n$, i.e., $C$ is a set of feasible solutions, then a penalty term corresponding to $C$ is a function $\psi: B_n \to \mathbb R$ satisfying both $\psi(x)=0$ for all $x\in C$ and $\psi(x)>0$ for all $x\in B_n\setminus C$.
The QUBO formulation of the constrained optimization problem
\begin{align}
    &{\rm minimize \ } \phi(x) \\
    &{\rm subject \ to \ } x \in C
\end{align}
with a quadratic objective function $\phi$ of $x$ is obtained as
\begin{align}
    &{\rm minimize \ } \phi(x) + \lambda \psi (x) \\
    &{\rm subject \ to \ } x \in B_n.
\end{align}
The coefficient $\lambda > 0$ of the penalty term is set sufficiently large so that $x$ violating constraints are sufficiently penalized.

\subsection{Ising Machines}

Near-optimal solutions of QUBO problems are expected to be efficiently sampled via computers called \emph{Ising machines}.
Generally, an Ising machine takes a QUBO problem as an input and returns heuristic solutions of the QUBO problem. 
\hl{
However, Ising machines tend to perform poorly when the input QUBO problem involves a complex energy landscape~\mbox{\cite{mohseni2022ising,stella2005optimization}}.
Here, the complexity is typically explained by the behavior of local minima, such as their number and sharpness, the minimum energy gap between the two lowest energy levels and so on.
Since practical combinatorial problems are typically associated with a badly complicated energy landscape when modeled in a QUBO form, 
this is a major issue of Ising machines for practical use.}

Several Ising machines are physically implemented with limited graph structures that we call \emph{hardware graphs}.
For example, D-Wave Advantage~\cite{Dwaveadvantage} and Advantage2~\cite{Dwaveadvantage2} machines are associated with the Pegasus and Zephyr graphs, respectively.
The hardware graph restricts the structure of an input QUBO problem in the sense that the graph associated with the QUBO problem must be a subgraph of the hardware graph.

\subsection{Minor-Embedding}

To embed a QUBO problem with an arbitrary graph structure to Ising machines with sparse hardware graphs, a technique called \emph{minor-embedding}~\cite{choi2008minor} has been developed.
We call the graph structure associated with the QUBO problem the \emph{input graph}.
Minor-embedding represents the input graph $I$ as a minor of the hardware graph $H$.
Specifically, a node in $I$ is represented by a set of connected nodes (called a \emph{chain}) in $H$.
The variables corresponding to nodes in a chain are penalized by their interaction so that they take the same values.
The strength of the penalty is called \emph{chain strength}.

When a QUBO problem is dense, i.e., the input graph has dense edges, minor-embedding requires a large number of hardware nodes to form chains.
Such large chains lead to performance degradation of Ising machines~\cite{yamamoto2020statica}.
Furthermore, the size of a hardware graph severely restricts the size of a dense input graph.
%For example, the 16-Pegasus graph $P_{16}$ of D-Wave Advantage machine~\cite{Dwaveadvantage} admits valid embedding of a complete graph of only 180 nodes and no embedding of larger complete graphs is known,
%while $P_{16}$ consists of 5640 nodes.
%For example, a hardware graph of size $N$ with bounded degree admits valid embedding of a complete graph of only $O(\sqrt{N})$ nodes.
For a complete input graph, minor-embedding is also called \emph{clique embedding}.
Clique embedding (with relatively short chains) of complete graphs of size 180 and 232 has been found for 
16-Pegasus graph $P_{16}$ of D-Wave Advantage~\cite{Dwaveadvantage} with 5640 nodes and 
15-Zephyr graph $Z_{15}$ of D-Wave Advantage2~\cite{Dwaveadvantage2} with 7440 nodes, respectively.
%Although finding optimal minor-embedding is NP-hard in general,
%Clique embedding can be efficiently computed for various hardware graphs.

\section{Order of Variables and Linearization}\label{sec:theory}

\begin{figure}[t]
\centering
\includegraphics[width=0.46\textwidth]{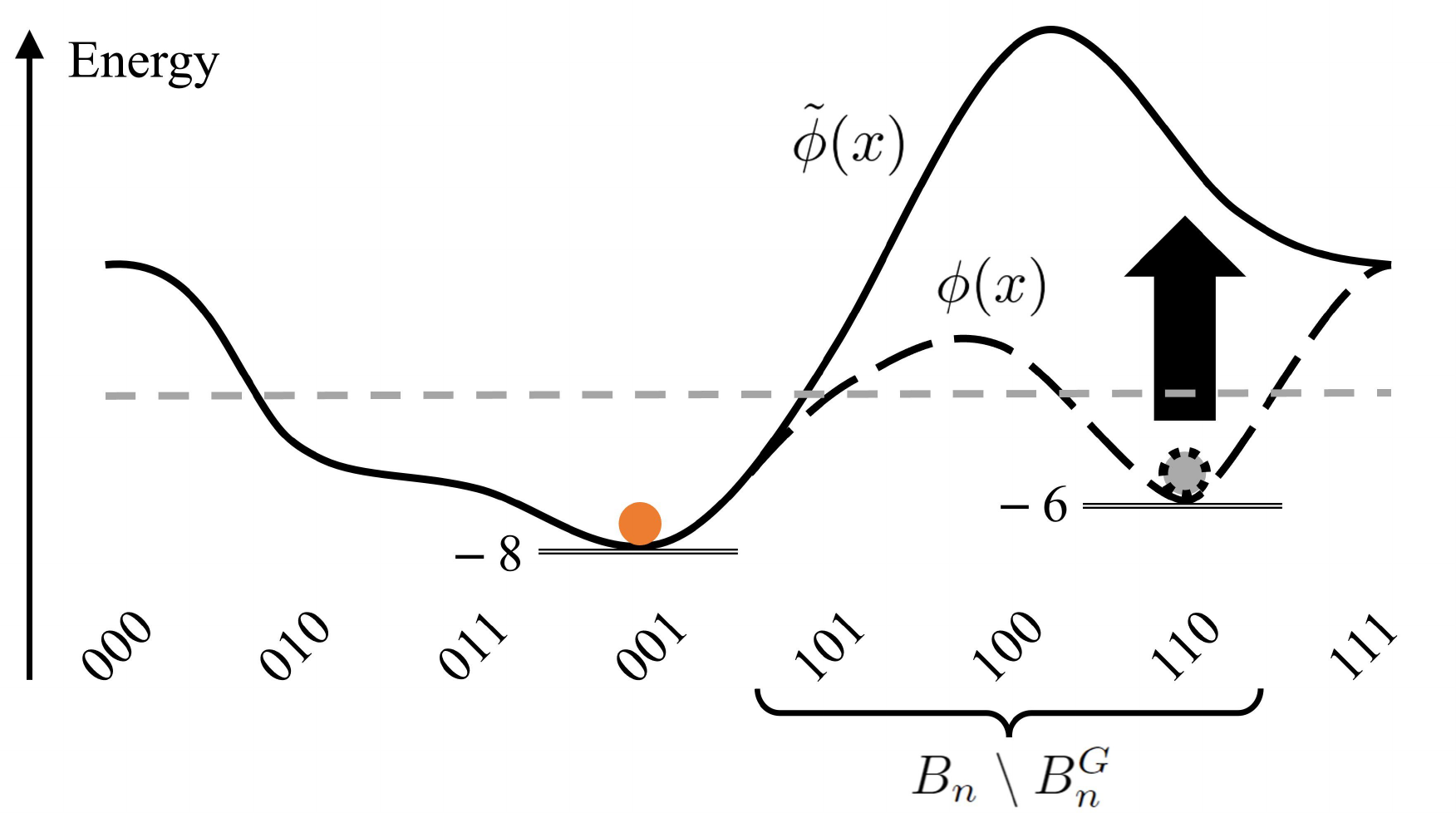}
\caption{
Conceptual figure of effect of auxiliary penalty on energy landscape. %Order of variables $x_1,x_2,x_3$ with edge set $E=\{(x_1,x_2),(x_1,x_3)\}$ is given in this example. 
Dashed and solid lines (black) represent original landscape and modified landscape after adding auxiliary penalty, respectively.
Area crossing %energy 
contour (dashed gray line) is made narrow by modification, restricting search space of Ising machines.
One local minimum (dashed circle) is removed by an auxiliary penalty, and a linearized problem has only one local minimum (red circle).
}
\label{fig:modify-landscape}
\end{figure}

We first give a motivating example of linearization using variable posets. %via ordering variables.
Take a QUBO matrix over three variables
\begin{align}\label{eq:qubo_example_orig}
    Q = \begin{pmatrix}
        -3 &  2 & 7 \\
         0 & -5 & 7 \\
         0 &  0 & -8 \\
    \end{pmatrix}.
\end{align}
The QUBO problem $\min_x \phi(x)$ with $\phi(x) \coloneqq x^\top Q x$ has two local minima: % $(1,1,0)$ and $(0,0,1)$. %, as can be seen by an exhaustive search.
$x=(1,1,0)$ with $\phi(x)=-6$ and $x=(0,0,1)$ with the global minimum $\phi(x)=-8$.
We consider %precedence 
conditions $x_1 = 1 \Rightarrow x_2 = 1$ and $x_1 = 1 \Rightarrow x_3 = 1$ as auxiliary constraints.
Such precedence conditions are obtained on knapsack problems, for example, as we discuss in Section~\ref{sec:application}.
Adding corresponding penalty terms $x_1 - x_1 x_2$ and $x_1 - x_1 x_3$ to %the objective function 
$\phi(x)$ %\coloneqq x^\top Q x$ 
with coefficients equal to 
$Q_{1,2}=2$ and $Q_{1,3}=7$ yields
\begin{align*}
    \tilde \phi(x) \coloneqq &\ %x^\top Q x 
    \phi(x) + 2(x_1 - x_1 x_2) + 7(x_1 - x_1 x_3) \\
    = & -3 x_1 - 5 x_2 - 8 x_3 + 2 x_1 x_2 + 7 x_2 x_3 + 7 x_1 x_3 \\ 
     &+ 2(x_1 - x_1 x_2) + 7(x_1 - x_1 x_3) \\
    = &\ 6 x_1 - 5 x_2 - 8 x_3 + 7 x_1 x_3 = x^\top \tilde Q x
\end{align*}
with a new QUBO matrix
\begin{align}\label{eq:qubo_example_lin}
    \tilde Q \coloneqq \begin{pmatrix}
         6 &  \underline{0} & \underline{0} \\
         0 & -5 & 7 \\
         0 &  0 & -8 \\
    \end{pmatrix},    
\end{align}
%which is a new QUBO objective function with a smaller number of quadratic terms.
%which is a new QUBO matrix with a smaller number of non-zero off-diagonal entries.
which has fewer quadratic terms (the underlined numbers become zero in Eq.~(\ref{eq:qubo_example_lin})).
In fact, the new objective function $\tilde \phi(x)$ has the same minimum as the original function $\phi(x)$, that is, $\tilde \phi(x)=-8$ at $x=(0,0,1)$, which follows from Theorems~\ref{thm:linearization} and \ref{thm:sufficient_condition_valid_order} below.
Energy landscapes of the original and new QUBO problems $\min_x \tilde \phi(x)$ 
are compared in Fig.~\ref{fig:modify-landscape}.
One of the local minima of the original QUBO problem is eliminated by adding the penalties.
Thus, the linearization process above yields an equivalent and sparser QUBO problem with a simplified energy landscape.

In the following sections,
we give a formal description of the generalization of the above example.
We define the notion of ordering variables %in binary optimization 
and explain the associated auxiliary penalties.
We then introduce linearization of QUBO problems and explain related theoretical results.

\subsection{Order of Variables and Associated Auxiliary Penalty}

We %first 
define %several 
notions on order of precedence on variables.

\begin{definition}\label{def:order}
    A \hl{\emph{partial order}} of $n$ binary variables $x=(x_1,\cdots,x_n)$ is a directed acyclic graph $G=(V,E)$ with a vertex set $V=\{x_1,\cdots,x_n\}$\footnote{
    Precisely, the nodes $x_1, \cdots, x_n \in V$ should be distinguished with the variables $(x_1,\cdots,x_n)\in B_n$ in a sense that $x_i$ is just a symbol as a node, not a value of 0 or 1.
    Although it might be consistent to label nodes as $1,\cdots,n$, we use $x_1, \cdots, x_n$ to emphasize that $G$ is defined on the variables.
    }
    and an edge set $E\subset V\times V$.
    %For a directed edge $(x_i,x_j) \in E$, we say $x_j$ \emph{precedes} $x_i$.
    We define a subset $B_n^G \subset B_n$ ordered by a partial order $G$ of variables as 
    \begin{align}
        B_n^G \coloneqq \{ x\in B_n \mid \forall (x_i, x_j)\in E, x_i=1 \Rightarrow x_j=1 \}. %\subset B_n
    \end{align}
    The partial order $G$ is \emph{valid} with respect to minimization of a function $\phi: B_n \to \mathbb R$ 
    if the following equality holds:
    \begin{align}
        \min \left\{\phi(x) \mid x \in B_n \right\} = 
        \min \left\{\phi(x) \mid x \in B_n^G \right\}.
    \end{align}
    %holds.
\end{definition}

By definition, a valid partial order of variables 
with respect to minimization of %a function 
$\phi$ induces auxiliary constraints on the optimization problem $\min_x \phi(x)$ preserving the optimum.
Finding a valid partial order for given $\phi$ is a non-trivial task, which we argue in the later sections.

%\subsection{Auxiliary Penalty}

We discuss injecting penalties in accordance with a given partial order.
Let $\phi: B_n \to \mathbb R$ be an objective function of a minimization problem.
We consider an ordinal condition $x_i=1 \Rightarrow x_j=1$ 
in a partial order $G$ of variables that is valid with respect to minimization of $\phi$.
Since imposing the condition as a constraint %does not change 
preserves the minimum of $\phi$, 
so does adding a penalty term corresponding to the constraint. 
A penalty term representing the 
%precedence 
condition $x_i=1 \Rightarrow x_j=1$ is defined by
\begin{align}%\textstyle
    x_i - x_i x_j.
\end{align}
Indeed, it takes 1 only if $x_i=1$ and $x_j=0$, and 0 otherwise.
Therefore, we have the following result.

\begin{proposition}\label{prop:aux_penalty}
    Let $G$ be a partial order of variables valid with respect to minimization of a function $\phi: B_n \to \mathbb R$.
    Then, for any non-negative function $c: E\times B_n \to \mathbb R$, we have
    \begin{align}\label{eq:aux_penalty}
        \min_{x\in B_n} \phi(x) = \min_{x\in B_n} \left( \phi(x) + \sum_{e\in E} c(e, x) (x_i - x_i x_j) \right).
    \end{align}
    Moreover, if $x^* \in B_n$ attains the minimum of the right hand side, then it attains the minimum of the left hand side.
\end{proposition}

We refer to Appendix~\ref{app:aux_penalty} for a proof, as it is straightforward.
The function $c(e,x)$ in Proposition~\ref{prop:aux_penalty} represents coefficients %corresponding to the strength 
of auxiliary penalties.
Our key insight is that QUBO problems can be simplified by taking appropriate %coefficients 
$c(e,x)$ depending on the objective function 
$\phi$.
%to utilize proposition~\ref{prop:aux_penalty} to simplify QUBO problems by 

\subsection{Linearization of QUBO Problems}

We introduce linearization of a QUBO matrix based on a partial  order of variables.
Then, we prove the equivalence of the linearized and original QUBO problems.

\begin{definition}
    Let $\phi(x) = x^{\top} Q x$ be a quadratic function of binary variables
    defined by an upper-triangle matrix $Q \in \mathbb R^{n\times n}$.
    %Here, $Q$ may be taken as an upper-triangle matrix.
    Let $G=(V,E)$ be a partial order of variables valid with respect to minimization of $\phi$.
    For $i=1,2,\cdots, n$, we define
    \begin{align}
        U_i^+ \coloneqq \left\{j \in \{1,2,\cdots,n\} \mid (x_i,x_{j}) \in E, \ Q_{i,j} > 0 \right\}, \\
        U_i^- \coloneqq \left\{j \in \{1,2,\cdots,n\} \mid (x_i,x_{j}) \in E, \ Q_{j,i} > 0 \right\}.
    \end{align}
    Then, we define a new QUBO matrix $Q^G \in \mathbb R^{n\times n}$ as
    \begin{align}\label{eq:def_linearization}
        Q^G_{i,j} &\coloneqq \notag \\
        &\begin{dcases}
            0 & j \in U_i^+ {\rm \ or \ }i \in U_j^- \\
            %0 & j \in U_i^+ \\
            %0 & i \in U_j^- \\
            Q_{i,i} + \sum_{j'\in U_i^+} Q_{i,j'} + \sum_{j'\in U_i^-} Q_{j',i} & i=j \\
            Q_{i,j} & {\rm otherwise.}
        \end{dcases}
    \end{align}
    The process to obtain $Q^G$ from $Q$ is called \emph{linearization} of $Q$ with respect to $G$.
\end{definition}

Recall that off-diagonal and diagonal entries of a QUBO matrix correspond to quadratic and linear terms of the QUBO objective function, respectively.
The new QUBO matrix $Q^G$ is obtained by converting each quadratic term $Q_{i,j}x_i x_j$ or $Q_{j,i}x_i x_j$ in the objective function with a positive coefficient satisfying $(x_i, x_j) \in E$ to a linear term $Q_{i,i}x_i$, which is the reason we call the process linearization.
%That is, $Q^G$ is obtained by moving each positive off-diagonal entry $Q_{i,j}$ such that there is an edge $(x_i, x_j) \in E$ to the diagonal entry $Q_{i,i}$.
%Note that off-diagonal and diagonal entries of a QUBO matrix corresponds to quadratic and linear terms of the objective function, respectively.
%Therefore, we call the process to obtain $Q^G$ from $Q$ as \emph{linearization} with respect to $G$.
Note that the number of off-diagonal entries of the QUBO matrix is reduced by $\sum_i \left( |U^+_i| + |U^-_i| \right)$ through linearization.

\begin{example}\label{ex:motivating}
    Consider the QUBO matrix $Q$ given in Eq.~(\ref{eq:qubo_example_orig}).
    Two precedence conditions $x_1 = 1 \Rightarrow x_2 = 1$ and $x_1 = 1 \Rightarrow x_3 = 1$ correspond to a partial order $G=(V,E)$ with $V=\{x_1, x_2, x_3\}$ and $E=\{(x_1, x_2), (x_1,x_3)\}$.
    In this setting, $U_1^+ = \{2,3\}$ and all other $U_i^+$ and $U_i^-$ are empty.
    Then, it can be easily checked that the linearized QUBO matrix $Q^G$ obtained by Eq.~(\ref{eq:def_linearization}) is equal to $\tilde Q$ given in Eq.~(\ref{eq:qubo_example_lin}).
\end{example}

Based on Proposition~\ref{prop:aux_penalty}, we have the following result.

\begin{theorem}\label{thm:linearization}
    Let $Q\in \mathbb R^{n \times n}$ be an upper-triangle matrix
    and $G$ be a partial order of variables valid with respect to minimization of a function $\phi(x) = x^\top Q x$.
    Then, the following equality holds:
    \begin{align}
        \min_{x \in B_n} x^\top Q x = \min_{x \in B_n} x^\top Q^G x.
    \end{align}
    Moreover, if $x^* \in B_n$ attains the minimum of the right hand side, then it attains the minimum of the left hand side.
\end{theorem}

\begin{proof}
    We define a non-negative function $c: E \times B_n \to \mathbb R$ as
    \begin{align}
        c((x_i,x_j),x) \coloneqq
        \begin{cases}
            Q_{i,j} & i<j {\rm \ and \ } Q_{i,j}>0 \\
            Q_{j,i} & i>j {\rm \ and \ } Q_{j,i}>0 \\
            0 & \text{otherwise.}
        \end{cases}
    \end{align}
    Then, we have
    \begin{align*}
        x^\top Q x + \sum_{e\in E} & c(e, x) (x_i - x_i x_j) \\
        = \sum_i & \Biggl( \sum_j Q_{i,j}x_i x_j + \sum_{j \in U_i^+} Q_{i,j} (x_i - x_i x_j) \\ 
         & +  \sum_{j \in U_i^-} Q_{j,i} (x_i - x_i x_j) \Biggr) \\
        = \sum_i & \sum_j Q^G_{i,j} x_i x_j = x^\top Q^G x.
    \end{align*}
    Thus, we obtain the assertion by applying Proposition~\ref{prop:aux_penalty}.
\end{proof}

\subsection{Sufficient Condition for Validity of Order}

Since the definition of validity of a partial order $G$ of variables with respect to minimization of %a function 
$\phi$ 
depends on the minimum of $\phi$, determining validity of $G$ might be as difficult as obtaining the minimum of $\phi$.
Therefore, it seems impractical to exactly determine the validity of a given partial order of variables in a reasonable time.
%compute a maximal valid order (when we define order of orders by inclusion relation between edge sets) in a reasonable time.
Instead, we give a sufficient condition for validity of $G$ when $\phi$ is a quadratic function.

\begin{theorem}\label{thm:sufficient_condition_valid_order}
    Let $G=(V,E)$ be a partial order of $n$ variables and
    $\phi(x)=x^\top Q x$ be a quadratic function with $Q\in \bR^{n\times n}$.
    For $i,j=1,2,\cdots,n$, we define
    \begin{align}
        a_{i,j} \coloneqq \begin{cases}
            Q_{i,j} + Q_{j,i} & i \ne j \\
            Q_{i,i} & i=j.
        \end{cases}
    \end{align}
    That is, $a_{i,j}$ is the coefficient of $x_i x_j$ (or $x_i$ if $i=j$) in $\phi(x)$.
    If an inequality
    %\begin{align}
    %    c^+ \coloneqq \begin{cases}
    %        c & c>0 \\
    %        0 & c\le 0
    %    \end{cases}
    %\end{align}
    \begin{align}\label{eq:detect_valid_edge}
        S_{i,j} \coloneqq \sum_{\substack{k=1 \\ k\ne i,j}}^n \max \{ 0, a_{j,k} - a_{i,k} \} + a_{j,j} - a_{i,i} \le 0
    \end{align}
    holds for every directed edge $(x_i,x_j)\in E$, then $G$ is valid with respect to minimization of $\phi$.  
\end{theorem}

Eq.~(\ref{eq:detect_valid_edge}) assures that $\phi(x)$ with $(x_i, x_j) = (0,1)$ is not larger than $\phi(x)$ with $(x_i, x_j) = (1,0)$.
Under such a situation, if there exists an optimal solution with $(x_i, x_j) = (1,0)$, then an optimal solution with $(x_i, x_j) = (0,1)$ also exists\footnote{
Note that $(x_i, x_j) = (0,1)$ indeed satisfies $x_i=1 \Rightarrow x_j=1$.
%Note also that
%the validity of the ordinal condition $x_i=1 \Rightarrow x_j=1$ 
%%does not necessarily imply that $\phi(x)$ with $(x_i, x_j) = (1,1)$ is not larger than $\phi(x)$ with $(x_i, x_j) = (1,0)$.
%is not equivalent to $\phi(x_i=1, x_j=1) \le \phi(x_i=1, x_j=0)$ as 
It might seem necessary to compare the objective values between $(x_i, x_j) = (1,0)$ and $(x_i, x_j) = (1,1)$ instead of $(x_i, x_j) = (0,1)$ to verify the validity of $x_i =1\Rightarrow x_j=1$.
However, $\phi(x_i=1, x_j=1) \le \phi(x_i=1, x_j=0)$ need not hold since
we consider the ordinal condition only for the optimum.
%See Appendix~\ref{app:sufficient_condition} for further discussion.
}.
Therefore, by checking Eq.~(\ref{eq:detect_valid_edge}) for all edges in $E$, it is concluded that $G$ is valid.
A detailed proof of Theorem~\ref{thm:sufficient_condition_valid_order} is in Appendix~\ref{app:sufficient_condition}.
For the QUBO matrix $Q$ and partial order $G$ in Example~\ref{ex:motivating}, we have $S_{1,2} = -2$ and $S_{1,3}=0$, which implies the validity of $G$. 
We provide another %simple but 
useful and instructive application of Theorem~\ref{thm:sufficient_condition_valid_order} in the following example.

\begin{example}\label{ex:symmetric_order}
    Let $d\le n$ be a positive integer.
    Assume that a function $\phi: B_n \to \bR$ is \emph{symmetric} with respect to variables $x_1,\cdots, x_d$, that is, invariant under permutation of values of $x_1,\cdots, x_d$.
    Then, we have $S_{i,j}=0$ for $i,j=1,\cdots,d$ with $i\ne j$, since $a_{i,i}=a_{j,j}$ and $a_{i,k}=a_{j,k}$ holds for any $k\ne i,j$.
    Thus, by Theorem~\ref{thm:sufficient_condition_valid_order}, a partial order defined by an edge set
    \begin{align}
        %E= \{ (x_{i+1}, x_i) \mid i=1,\cdots, d-1\}
        E= \{ (x_i, x_j) \mid 1\le i \le j \le d\}
    \end{align}
    is valid with respect to minimization of $\phi$.
    Note that $E$ defines a \emph{total order} on $V=\{x_1,\cdots, x_d\}$ and attains the largest possible edge set on $d$ nodes.
\end{example}

\section{Proposed Method}\label{sec:proposed_method}

%We explain our proposed method, which utilizes linearization of QUBO problem for Ising machines.
%We discuss algorithms to implement extraction of valid order and linearization in a computationally efficient way.

%\subsection{Method Overview}

We propose applying an Ising machine to a QUBO problem after linearizing it.
Specifically, the proposed method consists of the following steps.
First, we extract a valid partial order of variables from a given optimization problem.
Then, we linearize the QUBO matrix on the basis of the extracted order.
After that, the linearized QUBO problem is input to an Ising machine to sample a solution.
In the following, we discuss expected effects and algorithms of the proposed method.

\subsection{Effects on Application of Ising Machines}

There are two aspects regarding effects of the proposed method on application of Ising machines.
One is on the energy landscape of QUBO problems, and the other is on minor-embedding.
We expect that linearization has greater effects in both aspects with a larger number of edges $|E|$ in the extracted partial order $G$ of variables.
We explain %one by one 
the effects
below.

First, the proposed method improves solutions obtained with Ising machines by modifying the energy landscape.
%As shown in the proof of Theorem~\ref{thm:linearization}, 
Linearization is derived by adding auxiliary penalties to the objective function.
The penalties change the energy landscape of the QUBO problem as shown in Fig.~\ref{fig:modify-landscape}.
%In particular, 
They restrict a region with low energy by increasing energy on $B_n \setminus B_n^G$. %(see Fig.~\ref{fig:modify-landscape}).
Since Ising machines sample lower-energy solutions with higher probabilities, restricting low-energy regions enables them to sample near-optimal solutions. %closer to optimal one.
Furthermore, linearization eliminates quadratic terms, thus %leads to simpler energy landscape, 
possibly removing local minima, 
which seems favorable for both local and global search algorithms.
For these reasons, we expect that linearization will enable Ising machines to output lower-energy solutions.

Second, the proposed method mitigates defects of minor-embedding by making a QUBO matrix sparser.
A linearized QUBO matrix $Q^G$ has fewer off-diagonal elements than the original $Q$.
Therefore, %the proposed method 
linearization would reduce %lead to reduction of 
the number of auxiliary variables required for minor-embedding.
This would further result in mitigating performance degradation of Ising machines due to embedding with large chains as well as enabling the embedding of larger QUBO problems.

\subsection{Extraction of Valid Partial Order}

In the first step of the proposed method, we extract a valid partial order of variables from a given problem.
This might be done in a problem-specific or general way.
In the following, we explain a general algorithm to extract a valid partial order from a given QUBO problem on the basis of Theorem~\ref{thm:sufficient_condition_valid_order}.
We discuss problem-specific cases in Section~\ref{sec:application} where a valid partial order can be extracted in a more computationally efficient way.

\begin{algorithm}[t]
 \caption{Extraction of valid partial order}\label{alg:extract_order}
 \begin{algorithmic}[1]
 \Require{QUBO matrix $Q\in \bR^{n\times n}$}
 \Ensure{Edge set $E$ of valid partial order}
  \State $a \leftarrow Q + Q^\top$,
  %\State 
  $E \leftarrow \varnothing$
  \For {$i=1,2,\cdots,n$}
    \For {$j=1,2,\cdots,n$}
      \If {$i\ne j$ and $(x_j, x_i) \notin E$ and $Q_{j,j} \le Q_{i,i}$} \label{state:detect_start}
        \State $S \leftarrow Q_{j,j} - Q_{i,i}$ \label{state:calculate_score_start}
        \For {$k=1,2,\cdots,n$}
          \If {$k\ne i,j$ and $a_{j,k}>a_{i,k}$}
            \State $S \leftarrow S + a_{j,k}-a_{i,k}$
            \If {$S>0$} \label{state:pruning_start}
              \State \textbf{break}
            \EndIf \label{state:pruning_end}
          \EndIf
        \EndFor \label{state:calculate_score_end}
        \If {$S \le 0$} \label{state:append_edge_start}
          \State $E \leftarrow E \cup \{(x_i, x_j)\}$
        \EndIf \label{state:append_edge_end}
      \EndIf \label{state:detect_end}
    \EndFor
  \EndFor
  \State \Return $E$
 \end{algorithmic} 
\end{algorithm}

We describe an algorithm that takes a QUBO matrix as an input and returns an edge set of a valid partial order in Algorithm~\ref{alg:extract_order}.
The algorithm is based on %constructed on the basis of %the sufficient condition for validity of an order given in 
Theorem~\ref{thm:sufficient_condition_valid_order}.
That is, we extract directed edges satisfying Eq.~(\ref{eq:detect_valid_edge}) to form a set of edges.
An edge set $E$ is initialized as an empty set.
In lines~\ref{state:detect_start}-\ref{state:detect_end}, 
$S_{i,j}$ in Eq.~(\ref{eq:detect_valid_edge}) for each pair $(x_i, x_j)$ of variables is calculated, and 
%it is checked whether a pair $(x_i, x_j)$ of variables are satisfies Eq.~(\ref{eq:detect_valid_edge}), and 
the pair is added to $E$ if it passes the test, i.e., $S_{i,j}\le 0$.
A condition $(x_j, x_i)\notin E$ is imposed in line~\ref{state:detect_start} to ensure that the obtained directed graph $G=(V,E)$ does not have a cycle.
Note that 
$(x_i, x_j)$ should satisfy $a_{j,j}\le a_{i,i}$ for Eq.~(\ref{eq:detect_valid_edge}) to be satisfied, 
since $\max\{0, a_{j,k} - a_{i,k}\} \ge 0$. %in Eq.~(\ref{eq:detect_valid_edge}),
Therefore, we check this inequality in line~\ref{state:detect_start} to omit redundant computation.
In lines~\ref{state:calculate_score_start}-\ref{state:calculate_score_end}, $S_{i,j}$ in Eq.~(\ref{eq:detect_valid_edge}) is calculated as $S$.
We break the loop in lines~\ref{state:pruning_start}-\ref{state:pruning_end} as soon as $S$ becomes larger than $0$, since Eq.~(\ref{eq:detect_valid_edge}) never holds then.
This pruning again omits redundant computation.
If $S_{i,j}\le 0$ is checked for the pair, we find the pair represents a valid precedence and is added to $E$ in lines~\ref{state:append_edge_start}-\ref{state:append_edge_end}.
Algorithm~\ref{alg:extract_order} has computational complexity of $O(n^3)$ in the worst case.
%For most practical settings, it is expected that the algorithm run with a short time near $O(n^2)$ due to the pruning with detecting $S >0$ explained above for most of pairs.
If the pruning is effectively triggered for most of the loops over $k$, complexity approaches to $O(n^2)$.
Correctness of Algorithm~\ref{alg:extract_order} is summarized as follows.
See Appendix~\ref{app:algorithm_completeness} for a proof.

\begin{theorem}\label{thm:algorithm_completeness}
    Let $Q\in \bR^{n \times n}$ be a square real matrix and $x=(x_1, \cdots, x_n)$ be a vector of binary variables.
    Then, a graph $G=(V, E)$ with a node set $V=\{x_1, \cdots, x_n\}$ and an edge set $E$ obtained by running Algorithm~\ref{alg:extract_order} with $Q$ given as the input is a 
    %transitively closed, 
    directed acyclic graph, i.e., partial order of variables $x$.
    %In particular, $G$ is an order of variables $x$.
    Moreover, $G$ is valid with respect to minimization of %a quadratic function 
    $\phi(x) = x^\top Q x$.
\end{theorem}

We remark that Algorithm~\ref{alg:extract_order} can be more optimized for sparse matrices by bounding 
%the iterations over $n^2$ pairs of variables to only pairs with non-zero coefficients.
the iteration range of $j$ to $e(i)$ and that of $k$ to union $e(i) \cup e(j)$, where we write the set of adjacent edges of node $x_i \in V$ in a graph associated with $Q$ as $e(i)$.
%With such extension, 
Worst-case complexity of the extended version of Algorithm~\ref{alg:extract_order} is $O({\rm OD}(Q)d)$, where ${\rm OD}(Q)$ is the number of non-zero off-diagonal entries of $Q$ and $d$ is the maximum degree of a graph associated with $Q$.
The pruning in a loop over $k$ should reduce the complexity to near $O({\rm OD}(Q))$ for QUBO instances on which variables are hard to order.
Since the extension is straightforward, we omit further discussion of it.

On both dense and sparse settings, we expect that the algorithms run practically fast for most cases, since $O({\rm OD}(Q))$ ($=O(n^2)$ for dense $Q$) is typically the same complexity as preparing the QUBO matrix $Q$.
When the running time exceeds this complexity, then the obtained edge set $E$ should account for a certain percentage of all edges in a graph associated with $Q$.
Based on the assumption that linearization has larger positive effects
with larger $|E|$, the proposed method should have greater impact if the running time is long, achieving a good trade-off between time and effectiveness.

\begin{figure}[t]
\centering
\includegraphics[width=0.48\textwidth]{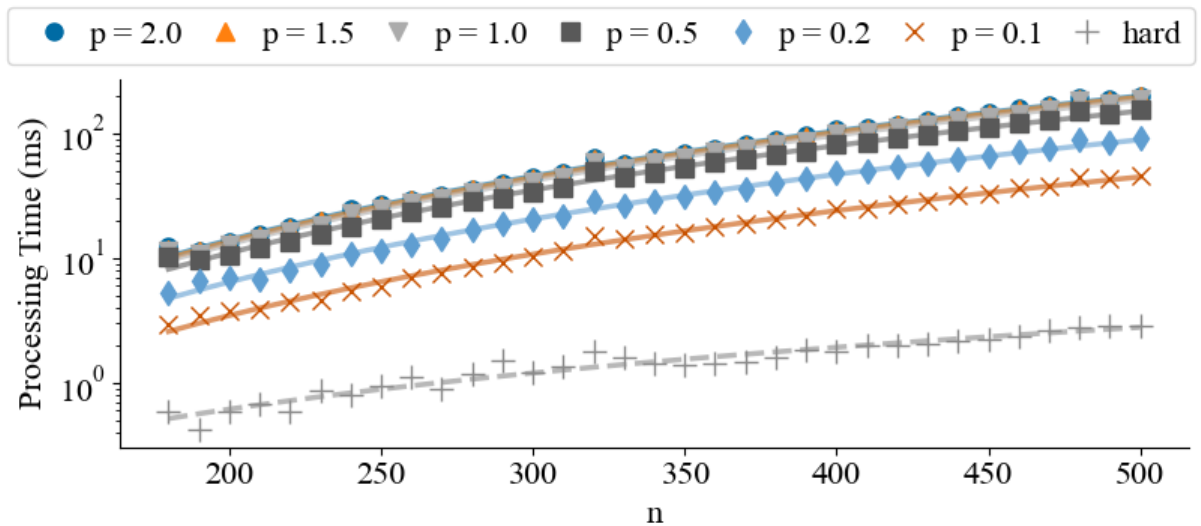}
\caption{
Running time of Algorithm~\ref{alg:extract_order} averaged over 10 randomly generated QUBO instances for each problem size $n$ with power curve fitting.
}
\label{fig:alg2_runtime}
\end{figure}

\begin{table}[t]
  \caption{Scaling Exponent of Running Time for Algorithm~\ref{alg:extract_order}.}
  \label{tab:scaling_algorithm2_runtime}
  \centering
  \begin{tabular}{@{\hspace{4pt}}c|c@{\hspace{8pt}}cccccc|c@{\hspace{4pt}}}
    \hline
    %& \multicolumn{6}{c|}{$p$} &\\
    %\cline{1-6}
    Problem Class & $p$ & 2.0 & 1.5 &1.0 & 0.5 & 0.2 & 0.1 & hard \\
    \hline %\hline
    %    & 3.03 & 3.04 & 3.03 & 2.99 & 2.94 & 1.99\\
    Exponent   && 2.90 & 2.91 & 2.90 & 2.86 & 2.86 & 2.79 & 1.65 \\
    \hline
  \end{tabular}
\end{table}

We demonstrate the effect of the pruning in %on running time of 
Algorithm~\ref{alg:extract_order}.
Running times of Algorithm~\ref{alg:extract_order} on synthetic QUBO instances of various size over several problem classes are shown in Fig.~\ref{fig:alg2_runtime}.
A parameter $p$ of problem classes controls the number $|E|$ of ordered pairs extracted via Algorithm~\ref{alg:extract_order}.
Specifically, $|E|$ is large for large $p$.
See Section~\ref{sec:experiments} for details of synthetic QUBO problems and dependence of $|E|$ on $p$. %the number of ordered pairs depending on $p$.
Fig.~\ref{fig:alg2_runtime} indicates that small $p$ leads to short processing time.
This is because the pruning in a loop over $k$ in Algorithm~\ref{alg:extract_order} is triggered almost every time for small $p$, since most pairs $(x_i, x_j)$ satisfy $S_{i,j} > 0$ for small $p$.
We also prepared another set of synthetic QUBO instances by sampling entries $Q_{i,j}$ of an upper-triangle QUBO matrix uniformly from $\{-1, 0, 1\}$.
We call them \emph{hard} instances, since we cannot expect variables of %such QUBO 
the problems to be ordered at all by Algorithm~\ref{alg:extract_order}.
On hard instances, positivity of $S_{i,j}$ for each pair $(x_i, x_j)$ is detected on early stages in the loop of $k$, and so the total processing time should be near $O(n^2)$ rather than $O(n^3)$.
We indeed observed that the running times on hard instances are a magnitude shorter than those on the other synthetic QUBO instances in Fig.~\ref{fig:alg2_runtime}.
We performed power regression on the running time, summarize the %power 
exponents in Table~\ref{tab:scaling_algorithm2_runtime}, and plot the fitted curves in Fig.~\ref{fig:alg2_runtime}.
As expected, the %computational 
running time scales about on the order of $n^3$ for large $p$ and with slightly smaller exponents for small $p$.
It scales with an exponent less than 2 on hard instances, which indicates that Algorithm~\ref{alg:extract_order} runs fast due to the pruning on QUBO instances on which variables are clearly hard to order.
This is preferable for a practical use, since typically in practical complex problems, most variables are hard to order and several variables can be ordered.
%On such problems, Algorithm~\ref{alg:extract_order} is expected to run with feasible amount of time even on large scales.

\subsection{Linearization Algorithm}

\begin{algorithm}[t]
 \caption{Linearization of QUBO}\label{alg:qubo_linearization}
 \begin{algorithmic}[1]
 \Require{%Upper-triangle 
 QUBO matrix $Q\in \bR^{n\times n}$, Valid partial order $G=(V,E)$}
 \Ensure{Linearized QUBO matrix $Q^{G}\in \bR^{n\times n}$}
 %\STATE $E \leftarrow \bar E$ (optional) \label{state:closure} 
  \For {$(x_i,x_j) \in E$} \label{state:sparsification_start}
  \If {$Q_{i,j} > 0$}
    \State $Q_{i,i} \leftarrow Q_{i,i} + Q_{i,j}$
    \State $Q_{i,j} \leftarrow 0$
  \EndIf
  \If {$Q_{j,i} > 0$}
    \State $Q_{i,i} \leftarrow Q_{i,i} + Q_{j,i}$
    \State $Q_{j,i} \leftarrow 0$
  \EndIf
  \EndFor \label{state:sparsification_end}
 \State \Return $Q$ 
 \end{algorithmic} 
\end{algorithm}

We proceed to the second step in the proposed method: linearization.
Algorithm~\ref{alg:qubo_linearization} shows a pseudo-code for linearization of a QUBO matrix.
As it simply computes $Q^G$ following the definition given in Eq.~(\ref{eq:def_linearization}),
the correctness of Algorithm~\ref{alg:qubo_linearization} is straightforward, so we omit the proof.

If one uses Algorithm~\ref{alg:extract_order} to create a partial order $G$ as an input of  Algorithm~\ref{alg:qubo_linearization}, then a combined process that takes $Q$ as an input and outputs a linearized QUBO matrix can be implemented in one algorithm in a more optimized way.
That is, instead of linearizing $Q$ after calculating whole $E$, we may linearize $Q$ each time an ordered pair $(x_i, x_j)$ is found (lines~\ref{state:append_edge_start}-\ref{state:append_edge_end} in Algorithm~\ref{alg:extract_order}).
It also allows whole use of $E$ in Algorithm~\ref{alg:extract_order} to be removed, which simplifies the algorithm.

\section{Application to Practical Problems}\label{sec:application}

We illustrate applications of the proposed method to practical problems.
We take (multi-dimensional) knapsack problems as a typical class of combinatorial optimization problems, show an example of problem-specific ways to extract a partial order of variables, and explain the applicability of the proposed method.

\subsection{Knapsack Problems}

A \emph{knapsack problem} is a combinatorial optimization problem defined as follows.
A set of $n$ items is given and each item is associated with value $v_i$ and weight $w_i$.
Capacity $C$ of a knapsack is also given.
Here, we assume all these values are positive integers.
The objective is to obtain a subset of items with the maximum total value under a constraint that the total weight does not exceed the capacity.
A knapsack problem is mathematically modeled as the following:
\begin{align}
    \max_{x\in B_n} \left\{\sum v_i x_i \middle| \ \sum w_i x_i \le C\right\}.
\end{align}
Binary variables $x=(x_1, \cdots, x_n)$ are decision variables, and $x_i=1$ is interpreted to mean that item $i$ is selected.

A knapsack problem is converted to a QUBO problem by representing an inequality constraint $\sum w_i x_i \le C$ as a penalty term~\cite{lucas2014ising}.
\hl{To this end, we introduce an integer variable $0 \le z \le C$ with a constraint $\sum w_i x_i = z$.
Representing $z$ with binary expansion, the penalty term is defined as follows:}
\begin{align}\label{eq:inequality_penalty}
    k &= \lfloor \log C \rfloor + 1, \ %\notag \\
    R = C+1 - 2^{k-1} ,\notag \\
    H_{\rm ineq} &= \left(\sum_{i=1}^n w_i x_i %- y_1 - 2 y_2 - \cdots - 2^{k-2} y_{k-1} 
    - \sum_{i=1}^{k-1} 2^{i-1}y_i - R y_k \right)^2 .
\end{align}
\hl{Here, $y_1,\cdots, y_k$ are auxiliary binary variables.
Note that every integer from $0$ to $C$ can be represented by $\sum_{i=1}^{k-1} 2^{i-1}y_i + R y_k$ for some assignment of $y_1, \cdots, y_k$, e.g.,
$y_1=\cdots=y_k=1$ corresponds to $\sum w_i x_i = C$.}
Since a knapsack problem is a maximization problem, we flip the sign of the objective function. %to get a QUBO problem.
By taking a sufficiently large penalty coefficient $\lambda$, we obtain a QUBO problem:
\begin{align}\label{eq:kp_qubo}
    \min_{(x,y)\in B_{n+k}} \left( -\sum_{i=1}^n v_i x_i + \lambda H_{\rm ineq} \right).
\end{align}
%The penalty coefficient $\lambda>0$ should be taken %sufficiently 
%large enough.

\subsection{Partial Order of Variables in Knapsack Problems}

For the QUBO problem Eq.~(\ref{eq:kp_qubo}), 
we define a partial order $G=(V,E)$ on a set $V=\{x_1,\cdots,x_n\}$ of variables as follows:
\begin{align}
    E =
    \left\{(x_i, x_j) \in V\times V \left |
    \begin{array}{l}
    v_i \le v_j, w_i \ge w_j, \\
    (v_i=v_j, w_i=w_j \Rightarrow i<j)
    \end{array}
    \right.
    \right\}.
\end{align}
The partial order $G$ is valid with respect to minimization of the objective function in Eq.~(\ref{eq:kp_qubo}). 
%as already described in the motivating example in Section~\ref{sec:theory}.
A proof sketch is shown as follows.
Fix a set of selected items except item $i$ and $j$.
When item $j$ has a higher value than item $i$, then we obtain a higher total value by selecting item $j$ rather than selecting item $i$ if the knapsack admits.
If item $j$ has less weight than item $i$, then item $j$ can be selected satisfying the constraint when item $i$ can be selected.
Therefore, if both conditions on values and weights hold between item $i$ and $j$, then we may select item $j$ prior to item $i$.
In other words, we may add an auxiliary constraint $x_i = 1 \Rightarrow x_j = 1$.
A rigorous proof is given by extending Theorem~\ref{thm:sufficient_condition_valid_order} to a constrained setting with inequalities, which is %straightforward and 
left to Appendix~\ref{app:sufficient_condition}.
%Moreover, the core of the proof is already sketched in the motivating example in the previous section.
Note that when items $i$ and $j$ have exactly the same values and weights, 
%the above intuition introduces 
there is a subtlety in considering 
two-way constraints $x_i = 1 \Rightarrow x_j = 1$ and $x_j = 1 \Rightarrow x_i = 1$.
Adding both constraints prohibits selecting only either item $i$ or $j$, which possibly changes the optimum of the knapsack problem.
To avoid this, we add a restriction that $x_i = 1 \Rightarrow x_j = 1$ only when $i<j$ if the values and weights coincide.
We remark that this corresponds to removing cycles in $G$, which illustrates why we impose acyclicity of a graph to be a partial order of variables in Definition~\ref{def:order}.
%The two-way constraints $x_i = 1 \Rightarrow x_j = 1$ and $x_j = 1 \Rightarrow x_i = 1$ would form a cycle in a graph $G=(V,E)$.
%Necessity of removing the cycle leads to necessity of acyclicity of a graph to be an order in Definition~\ref{def:order}.

Every quadratic term of form $x_i x_j$ appears in the objective function in Eq.~(\ref{eq:kp_qubo}) with a positive coefficient.
Therefore, each quadratic term corresponding to an edge $(x_i, x_j)\in E$ is reduced to a linear term through linearization.

\subsection{Generalization to Multi-Dimensional Knapsack Problems}

The partial order of variables on a knapsack problem given in the previous section is easily extended to a \emph{multi-dimensional knapsack problem} (MKP).
An MKP is a generalization of a knapsack problem where multiple types of knapsack constraints are given.
Specifically,
$m$ types of capacities $C_k \ (k=1,\cdots,m)$ and $m$ types of weights $w_{k,i} \ (k=1,\cdots,m)$ of each item $i$ are given where $m$ is a positive integer.
The objective is to maximize the total value $\sum_i v_i x_i$ under $m$ constraints $\sum_i w_{k,i} x_i \le C_k$ for $k=1,\cdots,m$ instead of only one constraint.
In a similar manner to knapsack problems, a QUBO objective function $H$ and a valid partial order $G=(V,E)$ are defined as follows:
\begin{align}\label{eq:mkp_qubo}
    %\min_{(x,y_1,\cdots,y_m)\in B_n} 
    &H = -\sum_i v_i x_i + \lambda \bigl( H^{(1)}_{\rm ineq} + \cdots + H^{(m)}_{\rm ineq} \bigr), \\ \label{eq:mkp_order}
    %&E = \notag \\
    &E=\left\{(x_i, x_j) % \in V\times V 
    \left|
    \begin{array}{l}
        v_i \le v_j, w_{k,i} \ge w_{k,j} \ \forall k, \\
        %(\mbox{Equality holds } \forall k \Rightarrow i<j)
        \left(\begin{aligned}
            v_i &= v_j,\\ w_{k,i} &= w_{k,j} \ \forall k
        \end{aligned} \Rightarrow i<j \right)
    \end{array}
    \right.
    \right\}.
\end{align}
Here, $H^{(k)}_{\rm ineq}$ denotes a penalty term corresponding to $k$-th weights defined similarly as in Eq.~(\ref{eq:inequality_penalty}) for $k=1,\cdots,m$.
The edge set $E$ is extracted by comparing $(m+1)$ scalars for each pair $(x_i, x_j)$.
Computational complexity of such an algorithm is $O(n^2 m)$.

\section{Experiments}\label{sec:experiments}

We conduct numerical experiments to evaluate the effects of the proposed method on synthetic QUBO instances and MKP instances.
We first introduce problem instances. %, then explain experimental setup.
Then, the effects on (i) degradation of Ising machine performance via minor-embedding, (ii) improvement of embeddability of large-sized problems, and (iii) performance improvement on practical problems are examined.

We use Amplify Annealing Engine (AE)~\cite{amplify} as an Ising machine with an execution time of 1 second.
For minor-embedding search, we use the minorminer algorithm~\cite{cai2014practical} with some enhancements~\mbox{\cite{zbinden2020embedding}}.
%To enhance the performance of minorminer Algorithm, we adopt a method setting a clique embedding as an initial state of the algorithm, which is simple yet effective as shown in the previous study~\cite{zbinden2020embedding}.
See Appendix~\ref{app:experiment_setup} for details on the computational setup.

\subsection{Problem Instances}

We introduce datasets of problem instances on which the proposed method is evaluated.
As a notational convention, $U(l,u)$ denotes a uniform probability distribution over integers in an interval $[l, u]$ with integers $l,u$, and $\lfloor x \rfloor$ denotes the largest integer such that $\lfloor x \rfloor \le x$ for a real number $x$.

\subsubsection{Synthetic QUBO Instances}

Let $n$ be a positive integer. % denoting the number of variables.
An upper-triangle QUBO matrix $Q\in \mathbb R^{n \times n}$ is generated as follows with a positive integer $s$ and positive real number $p$ as inputs.
Set $o \coloneqq sp$.
Every off-diagonal upper-triangle entry of $Q$ is sampled independently from $U(1+o, s+o)$.
Every diagonal entry of $Q$ is sampled independently from $U(-\lfloor p(n-1)(s+1) \rfloor -o,-1-o)$.
The generated QUBO matrix has two properties.
First, it represents a fully connected graph structure with positive edge weights.
Second, large $p$ leads to a large number of ordered pairs of variables obtained by Algorithm~\ref{alg:extract_order}.
Therefore, we can evaluate the effect of the proposed method by controlling the value of $p$.
For more details, see Appendix~\ref{app:synthetic_qubo}.

We generated 10 QUBO instances for every combination of $s=10$, $n\in \{180, 232\}$ and $p \in \{0.1, 0.2, 0.5, 1.0, 1.5, 2.0\}$ following the above process.
Note that $n=180, 232$ are the near-maximum size of complete graphs embeddable to the 16-Pegasus graph $P_{16}$ and 15-Zephyr graph $Z_{15}$, respectively.
Furthermore, we generated QUBO instances increasing $n$ from $n=190$ to $n=500$ with the same set of parameters $p$ and $s$ to examine scaling of embeddability.
The %various-sized 
instances have been %already 
used to analyze %see scaling of 
running time of Algorithm~\ref{alg:extract_order} in Section~\ref{sec:proposed_method}.

\subsubsection{MKP Instances}

We use the OR-Library dataset of MKP instances~\cite{chu1998genetic}.
It consists of 30 randomly generated problem instances for every combination of the number of items $n \in \{100, 250, 500\}$ and the number of types of weights $m \in \{5,10,30\}$.
More specifically, there are 10 instances generated for each $m, n$ and the tightness parameter $\alpha \in \{0.25, 0.5, 0.75\} $ of inequality constraints in the following process:
Weight $w_{k,i}$ is independently sampled from $U(1,1000)$ for each $i=1,\cdots,n$ and $k=1,\cdots,m$.
Capacity $C_k$ is defined as $C_k = \alpha \sum_i w_{k,i}$.
Sample a real number $q\in [0,1]$ from a uniform distribution, then define value $v_i$ as $v_i = \sum_{i=1}^m w_{k,i}/m+500q_i$.

In addition to the above instances, we create knapsack problem instances (which we also call MKPs with $m=1$ for simplicity) from the instances of $m=5$ by ignoring constraints $\sum_{i}w_{k,i} x_i \le C_k %, k=2,\cdots,m
$ except $k=1$.
We use instances with $m=1,5,10$ in our experiments.
Instances with $m=30$ are not used since we found that a valid partial order is not obtained, that is $|E|=0$, and thus linearization has no effects on those instances.

\subsection{Performance of Ising Machine with Minor-Embedding}

%Minor-embedding is known to degrade performance of Ising machines.
We evaluate the effect of the proposed method on the quality of minor-embedding and on the performance of the Ising machine combined with minor-embedding.
We conduct experiments using the synthetic QUBO matrices with $n=180$ and $232$ described in the previous section.

\subsubsection{Setting and Metrics}

%On each of generated 10 QUBO instances, 
We apply Algorithms~\ref{alg:extract_order}
and \ref{alg:qubo_linearization} to linearize the generated QUBO matrices.
We evaluate reduction rates of non-zero off-diagonal entries of QUBO matrices.
Then, we run minor-embedding search on each linearized QUBO instance setting target graphs as %16-Pegasus graph 
$P_{16}$ for $n=180$ and %15-Zephyr graph 
$Z_{15}$ for $n=232$.

We evaluate the quality of the obtained embedding in terms of the number of auxiliary variables and the maximum chain length.
All metrics are averaged over 10 QUBO instances.
We set clique embedding~\cite{Dwaveadvantage,Dwaveadvantage2} as a baseline, since the original QUBO instances have a complete graph structure.

After obtaining minor-embedding for linearized QUBO instances, we apply the Ising machine 10 times to each of the QUBO problems of three types:
the original problems, equivalent problems embedded by clique embedding and linearized problems embedded by the obtained minor-embedding.
We evaluate performance of the Ising machine in terms of energy of the best solution averaged over 10 instances.
Note that Amplify SDK and AE provide an interface that accepts a fully connected QUBO problem without minor-embedding.
Nevertheless, we intentionally input minor-embedded QUBO problems to AE,
since our aim in the experiment is to evaluate performance degradation of Ising machines due to minor-embedding with various hardware graphs in a unified way.

\begin{table}[t]
  \caption{Effect of Linearization on Minor-Embedding.}
  \label{tab:random_num_order}
  \centering
  \begin{tabular}{
  @{\hspace{2pt}}c@{\hspace{6pt}}r@{\hspace{3pt}}c@{\hspace{0pt}}r@{\hspace{4pt}}r@{\hspace{2pt}}r
  @{\hspace{7pt}} r@{\hspace{2pt}}r @{\hspace{7pt}} r@{\hspace{2pt}}r@{\hspace{2pt}}
  }
    \hline
    $H$ &&& %\multicolumn{1}{c}{$|E|$} 
    & \multicolumn{2}{c}{${\rm OD}(Q^G)$}
    & \multicolumn{2}{c}{\# of Aux. Var.} & \multicolumn{2}{c}{Chain Length} \\
    \hline \hline
    \multirow{7}{*}{$P_{16}$}
    %\multirow{6}{*}{180} &
    & \multicolumn{2}{l}{clique} &  & 16110 & & 2820 & & 17 &\\
    %\cline{2-6}
    &     & 0.1 &  & 16110.0 &  $(-0.0\%)$ & 2819.5 &  $(-0.0\%)$ & 17.0 & $(-0.0\%)$ \\
    &     & 0.2 &  & 14943.9 &  $(-7.2\%)$ & 2819.8 &  $(-0.0\%)$ & 17.0 & $(-0.0\%)$ \\
    & $p$ & 0.5 &  &  8139.0 & $(-49.5\%)$ & 2819.4 &  $(-0.0\%)$ & 17.0 & $(-0.0\%)$ \\
    &     & 1.0 &  & 4408.6 & $(-72.6\%)$ & 2818.0 &  $(-0.1\%)$ & 17.0 & $(-0.0\%)$ \\
    &     & 1.5 &  & 3086.7 & $(-80.8\%)$ & 1111.2 & $(-60.6\%)$ & 12.9 & $(-24.1\%)$ \\
    &     & 2.0 &  & 2326.0 & $(-85.6\%)$ &  790.9 & $(-72.0\%)$ &  9.7 & $(-42.9\%)$ \\
    \hline
    \multirow{7}{*}{$Z_{15}$} &
    %\multirow{6}{*}{232} &
    \multicolumn{2}{c}{clique} &  & 26796 & & 3480 & & 16 &\\
    %\cline{2-6}
    &     & 0.1 &  & 26796.0 & $(-0.0\%)$ & 3479.9 & $(-0.0\%)$ & 16.0 & $(-0.0\%)$ \\
    &     & 0.2 &  & 25051.1 & $(-6.5\%)$ & 3479.7 & $(-0.0\%)$ & 16.0 & $(-0.0\%)$ \\
    & $p$ & 0.5 &  & 13725.3 & $(-48.8\%)$ & 3480.0 & $(-0.0\%)$ & 16.0 & $(-0.0\%)$ \\
    &     & 1.0 &  &  7398.2 & $(-72.4\%)$ & 3478.7 & $(-0.04\%)$ & 16.0 & $(-0.0\%)$ \\
    &     & 1.5 &  &  5056.2 & $(-81.1\%)$ & 1510.9 & $(-56.6\%)$ & 14.1 & $(-11.9\%)$ \\
    &     & 2.0 &  &  3793.4 & $(-85.8\%)$ & 1057.3 & $(-69.6\%)$ & 10.6 & $(-33.8\%)$ \\
    \hline
  \end{tabular}
\end{table}

\begin{table}[t]
  \caption{Energy of Solutions of Synthetic QUBO Problems.}
  \label{tab:random_ising_machine}
  \centering
  \begin{tabular}{@{\hspace{1pt}}c@{\hspace{5pt}}c@{\hspace{0pt}}c@{\hspace{6pt}}rr@{\hspace{0pt}}r@{\hspace{3pt}}rr@{\hspace{3pt}}r@{\hspace{1pt}}}
    \hline
    %\multirow{2}{*}{$
    %\begin{matrix}
    %    n \\
    %    H
    %\end{matrix}
    %$} 
    &&& \multicolumn{1}{c}{Without} && \multicolumn{4}{c}{With Embedding}\\
    \cline{6-9}
    $H$ 
    & %$n$ 
    & $p$ & \multicolumn{1}{c}{Embedding} & & 
    %\multicolumn{2}{c}{Clique (Baseline)} &  
    \multicolumn{1}{c}{Baseline} & \multicolumn{1}{c}{Diff.} &
    %\multicolumn{2}{c}{Linearized} 
    \multicolumn{1}{c}{Linearized} & \multicolumn{1}{c}{Diff.} 
    \\
    \hline \hline
    \multirow{5}{*}{$
    \begin{matrix}
    %    180 \\
        P_{16}    
    \end{matrix}
    $} &
    %\multirow{5}{*}{180} 
     & 0.2 &  $ -8225.7$ &  & $  -8119.7$ & ($+106.0$) & $  -8148.8$ & ($+76.9$) \\ % -8186.9 <- clique + linear
    && 0.5 &  $-30426.4$ &  & $ -30404.3$ &  ($+22.1$) & $ -30423.3$ &  ($+3.1$) \\ % -30425.7
    && 1.0 &  $-74842.3$ &  & $ -74838.3$ &   ($+4.0$) & $ -74842.3$ &  ($+0.0$) \\ % -74842.3
    && 1.5 & $-120531.9$ &  & $-120531.2$ &   ($+0.7$) & $-120531.9$ &  ($+0.0$) \\ % -120531.9
    && 2.0 & $-165226.1$ &  & $-165224.9$ &   ($+1.2$) & $-165226.1$ &  ($+0.0$) \\ % -165226.1
    \hline
    \multirow{5}{*}{$
    \begin{matrix}
    %    232 \\
        Z_{15}
    \end{matrix}
    $} &
    %\multirow{5}{*}{232}
     & 0.2 & $ -13604.3$ &  & $ -13188.2$ &  ($+416.1$) & $ -13256.0$ & ($+348.3$) \\ % -13422.9
    && 0.5 & $ -50108.2$ &  & $ -49913.4$ &  ($+194.8$) & $ -50088.9$ &  ($+19.3$) \\ % -50101.0
    && 1.0 & $-125602.6$ &  & $-125505.1$ &   ($+97.5$) & $-125602.5$ &   ($+0.1$) \\ % -125602.6
    && 1.5 & $-196882.1$ &  & $-196872.1$ &   ($+10.0$) & $-196881.5$ &   ($+0.6$) \\ % -196882.1
    && 2.0 & $-272778.4$ &  & $-272763.6$ &   ($+14.8$) & $-272778.4$ &   ($+0.0$) \\ % -272778.4
    \hline
  \end{tabular}
\end{table}

\subsubsection{Results}

Table~\ref{tab:random_num_order} shows results on the number of non-zero off-diagonal entries of QUBO matrices and on the quality of the minor-embedding.
The column label $H$ denotes a target hardware graph. 
The ``clique'' row for each target graph shows the results of the baseline, that is, clique embedding.
%$|E|$ denotes the number of ordered pairs of variables obtained in Algorithm~\ref{alg:extract_order}.
${\rm OD}(Q^G)$ denotes the number of non-zero off-diagonal entries of linearized QUBO matrices and of total off-diagonal entries on the ``clique'' rows.
%For larger $p$, more pairs of variables are ordered, and accordingly more quadratic terms in the QUBO problems are reduced.
The values of ${\rm OD}(Q^G)$ are less for larger $p$, since more pairs of variables are ordered.
Note that the reduction rates of %quadratic terms 
non-zero off-diagonal entries %mostly 
depend on $p$ and not on $n$.
For $p=1.5$ and $p=2.0$, we observe that the numbers of auxiliary variables and maximum chain lengths of the obtained minor-embeddings are significantly reduced.
For the other cases, such drastic changes are not observed.

We show results on the performance of the Ising machine in Table~\ref{tab:random_ising_machine}.
We exclude the results of $p=0.1$, since %in this setting 
%we obtained $|E|=0$ and thus 
no ordered pairs are found as in Table~\ref{tab:random_num_order} and thus
linearization does not change the QUBO problems.
On the original QUBO problems without minor-embedding, the Ising machine produced solutions with the same energy on all 10 executions for each instance.
This indicates that the generated QUBO problems are relatively easy to solve and that the Ising machine always outputs optimal solutions on these problems.
We also applied the Ising machine to the linearized QUBO problems without minor-embedding and observed that the result (omitted from Table~\ref{tab:random_ising_machine}) is exactly the same as that of the original problems. % without minor-embedding.
In Table~\ref{tab:random_ising_machine},
the energy of solutions of QUBO problems embedded with clique embedding shown in the 
%``Clique'' 
``Baseline'' 
column %in Table~\ref{tab:random_ising_machine}) 
is larger than that of the %QUBO 
problems without minor-embedding.
This implies that minor-embedding degrades the performance of the Ising machine even on those easy QUBO problems.
The energy of solutions of linearized %QUBO 
problems shown in the ``Linearized'' column %in Table~\ref{tab:random_ising_machine}) 
is smaller than 
%``Clique'' 
the baselines for all cases. % except the case for $n=232$ and $p=0.2$, in which we obtain the same energy.
This implies that the proposed method improves performance of the Ising machine for minor-embedded QUBO problems.
This is an interesting phenomenon, considering that the quality of minor-embedding is not changed by linearization for $p\le 1$.
Specifically, the results suggest that there are reasons for the improvement of performance other than improving the quality of minor-embedding.
This is further supported by additional experiments on clique embedding with linearization in Appendix~\ref{app:clique_linear}, where similar performance improvement via linearization is observed even for a fixed embedding.
One possible reason for the improvement is that the energy landscapes of the minor-embedded QUBO problems are simplified by linearization.

\subsection{Embeddability of Large Problems}

We evaluate success rates of %finding 
minor-embedding of linearized QUBO problems increasing the problem size $n$.
We conduct experiments using synthetic QUBO matrices increasing $n$ on the Pegasus graph $P_{16}$ and Zephyr graph $Z_{15}$ of the fixed size.

\subsubsection{Setting and Metrics}

%If a clique embedding of a complete graph of size $n$ exists, then any graph of size less than or equal to $n$ can be embedded.
%Therefore, 
We start increasing the problem size $n$ from $n=180$ for $P_{16}$ and $n=232$ for $Z_{15}$ with step size 10.
Note that smaller instances can be trivially embedded via clique embedding.
For each combination of $n$ and $p\in \{0.2, 0.5, 1.0, 1.5, 2.0\}$, we apply Algorithms~\ref{alg:extract_order} and \ref{alg:qubo_linearization} to the 10 generated QUBO instances to obtain linearized QUBO instances.
For each target graph $P_{16}$ and $Z_{15}$, 
we run minor-embedding search on those linearized QUBO instances and count the number of instances for which a valid minor-embedding is found.
%Similarly to in the previous experiment, we set clique embedding of a complete graph of size $180$ and $232$ as an initial state of the minorminer algorithm for $P_{16}$ and $Z_{15}$, respectively.
We remark that since 
only graph structures of QUBO matrices matter in this experiment
and
the density of linearized synthetic QUBO problems only depends on $p$, 
the linearized synthetic QUBO instances can be viewed as random graphs with almost constant density for a fixed $p$.

\begin{figure}[t]
    \centering
    \subfloat[Pegasus graph $P_{16}$\label{fig:embeddability_pegasus}]{
         \includegraphics[width=0.47\textwidth]{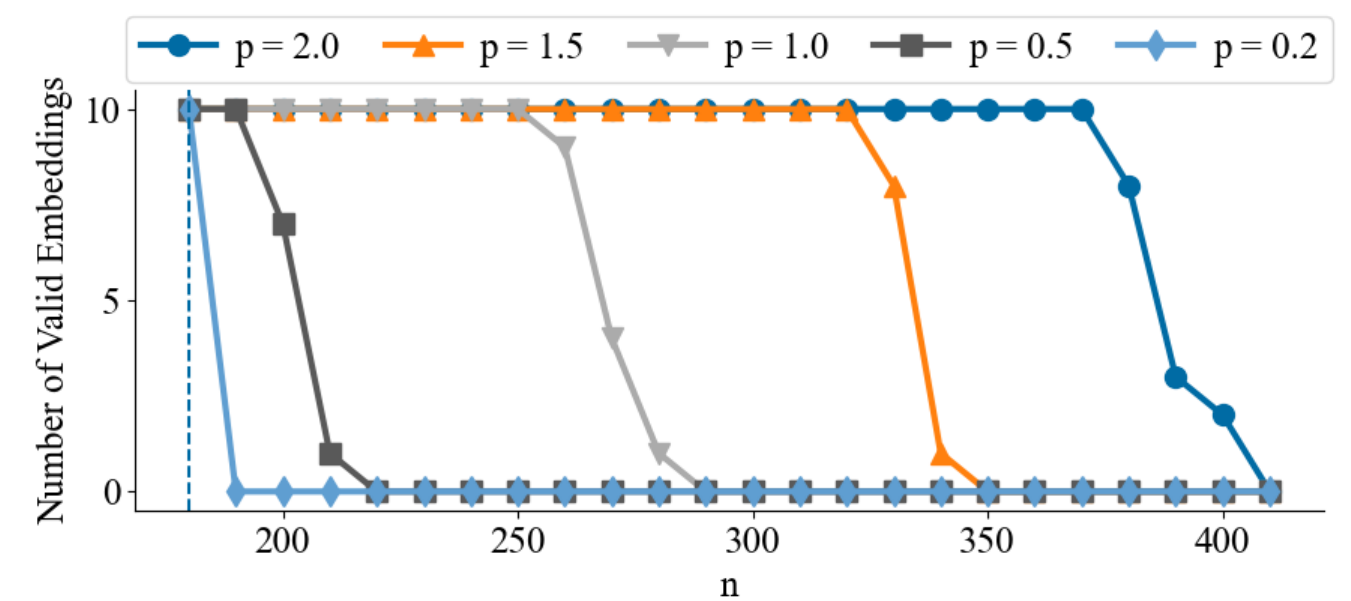}
    }
    \hfil
    \subfloat[Zephyr graph $Z_{15}$\label{fig:embeddability_zephyr}]{
         \includegraphics[width=0.47\textwidth]{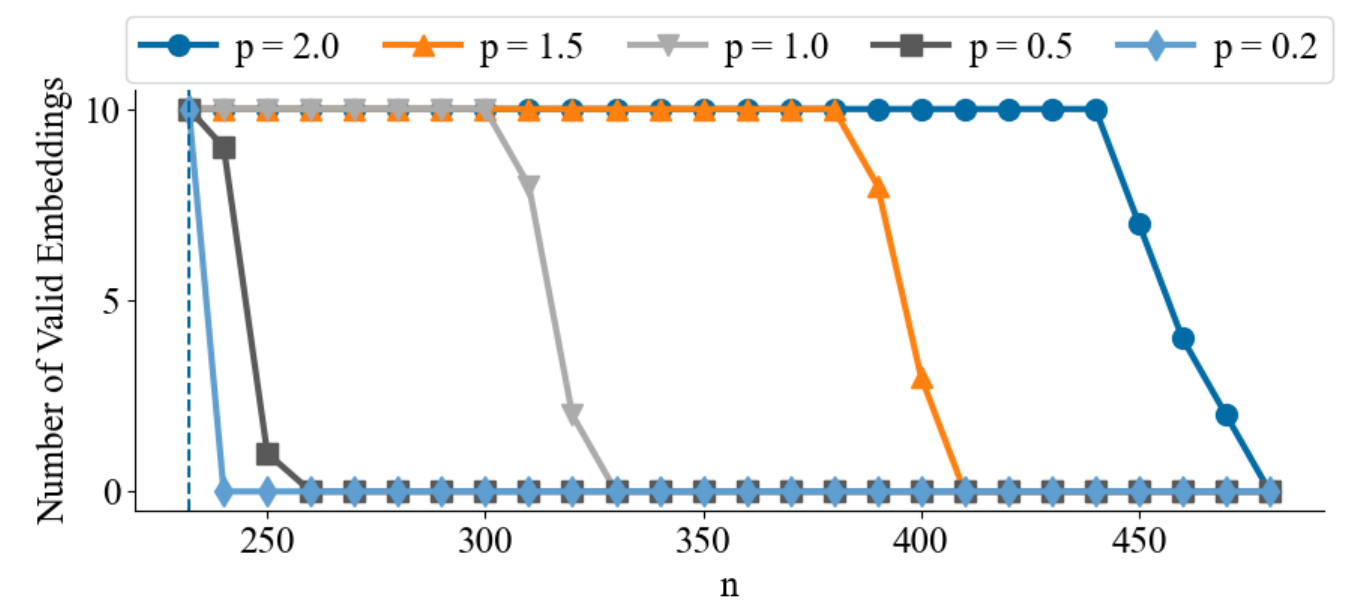}
    }
     \caption{Success counts of finding minor-embedding of linearized QUBO problems on fixed hardware graphs with increasing problem size.}
     \label{fig:embeddability_results}
\end{figure}

\subsubsection{Results}

We plot the number of instances for which valid minor-embedding is found in Fig.~\ref{fig:embeddability_results}.
%\ref{fig:embeddability_pegasus} and Fig.~\ref{fig:embeddability_zephyr}.
For $p \ge 0.5$, the QUBO instances with larger $n$ can be minor-embedded %into the target graph 
on both settings of $P_{16}$ and $Z_{15}$.
Note that no tested instances can be embedded without linearization, since the clique embedding gives a near-optimal upper bound of the size of dense QUBO problems that can be embedded.
We observe that larger $p$ results in a higher upper limit of the size of embeddable QUBO problems.
For example, all 10 instances of size $n=360$ are embedded into $P_{16}$ when $p=2.0$, while no valid embedding is found for instances of size $n=220$ when $p=0.5$. 
This is because more pairs of variables are ordered in Algorithm~\ref{alg:extract_order} for larger $p$, and thus the QUBO matrices are made sparser by linearization.

\subsection{Performance of Ising Machine on MKPs}

We evaluate the effect of the proposed method on performance of the Ising machine on MKPs.
Note that excellent algorithms to solve MKPs heuristically or exactly have been well developed through decades of research.
Our aim in this experiment is not to beat them, but to show to what extent our proposed method enables one of the state-of-the-art Ising machines to obtain near-optimal solutions under typical setting where the method can be applied.
The experiment is conducted on MKP instances described in the previous section.

\subsubsection{Setting and Metrics}

We solve MKPs by encoding them to QUBO problems as in Eq.~(\ref{eq:mkp_qubo}) with and without linearization with respect to the partial order given by Eq.~(\ref{eq:mkp_order}) and applying the Ising machine.
The whole objective function of an MKP in QUBO form with the proposed method is explicitly given as follows:
\begin{align}\label{eq:linearized_mkp}
    H_{\rm lin} &= -\sum v_i x_i %+ \lambda \sum_{k=1}^m H^{(k)}_{\rm lin,ineq}, 
    \notag \\
    & +\lambda \sum_{k=1}^m \left(H^{(k)}_{\rm ineq} + \sum_{(i,j)\in E} 2w_{k,i}w_{k,j}(x_i - x_i x_j) \right).
    %H^{(k)}_{\rm lin,ineq} &= H^{(k)}_{\rm ineq} + \sum_{(i,j)\in E} 2w_{k,i}w_{k,j}(x_i - x_i x_j).
\end{align}
A solution is feasible if $\sum_{i} w_{k,i}x_i \le C_k$ holds for all $k=1,\cdots,m$.
Note that a solution can be feasible even when $H^{(k)}_{\rm ineq}>0$, in which case the auxiliary variables for $H^{(k)}_{\rm ineq}$ are not completely optimized.
Note also that there are variations in the way to program the penalty terms with Amplify SDK~\cite{amplify} (see Appendix~\ref{app:experiment_setup} for details).

We execute the Ising machine 50 times to sample 50 solutions on every MKP instance with and without the proposed method applied.
Quality of solutions of an MKP instance is evaluated in terms of the number of feasible solutions and optimality gaps of solutions.
The average or best optimality gap of solutions are defined by
\begin{align}
    {\rm Optimality \ Gap} = \frac{S_{\rm best} - S_{\rm Ising}}{S_{\rm best}} \times 100,
\end{align}
where $S_{\rm best}$ is the best known 
score $\sum_i v_i x_i$ obtained using other solvers and $S_{\rm Ising}$ is the average or best score over feasible solutions obtained by the Ising machine.
In this experiment, $S_{\rm best}$ is obtained by running Gurobi Optimizer~\cite{gurobi} of version 9.1.2 for 10 seconds.
The obtained $S_{\rm best}$ for small-scale instances such as $n=100$ or $m=1$ is proved to be optimal by the solver.
All metrics are reported by averaging over 10 instances for each combination of $n, m, \alpha$.

On the penalty coefficient $\lambda$, we found that a sufficient number of feasible solutions is obtained with $\lambda=1$ %on both methods 
in a preliminary experiment.
Thus, we adopt $\lambda=1$ in this experiment.
Refer to Appendix~\ref{app:experiment_tuning} for results based on baselines tuned with respect to $\lambda$.

\begin{table}[t]
  \caption{Optimality Gap on MKP for Baseline and Proposed Method.}
  \label{tab:mkp_best_value}
  \centering
  \begin{tabular}{
  @{\hspace{1pt}}c@{\hspace{2pt}}c@{\hspace{3pt}}c @{\hspace{5pt}}r@{\hspace{5pt}}
  r@{\hspace{7pt}}r@{\hspace{6pt}}r@{\hspace{6pt}}
  r@{\hspace{7pt}}r@{\hspace{7pt}}r@{\hspace{8pt}} r @{\hspace{1pt}} r
  %rrrr@{\hspace{3pt}}rr@{\hspace{5pt}}c
  @{\hspace{3pt}}}
    \hline
    %& & & & & \multicolumn{6}{c}{Method 1}
    %& & \multicolumn{6}{c}{Method 2}
    %\\ 
    %\cline{6-11}
    %\cline{13-18}
    & Instance & & & \multicolumn{2}{c}{Baseline} & & \multicolumn{4}{c}{Linearized}
    %& & \multicolumn{2}{c}{Baseline} & & \multicolumn{3}{c}{Linearized}
    \\
    \cline{1-3}
    \cline{5-6}
    \cline{8-12}
    %\cline{13-14}
    %\cline{16-18}
    $m$ & $n$ & $\alpha$ 
    & & \multicolumn{1}{c}{\#FS} & Gap 
    & & \multicolumn{1}{c}{$|E|$} 
    & \multicolumn{1}{c}{\#FS} & Gap  & \multicolumn{2}{c}{\hl{RR (\%)}} %\multicolumn{1}{c}{Diff.}
    %& & \multicolumn{1}{c}{\#FS} & Gap 
    %& & \multicolumn{1}{c}{\#FS} & Gap  & %\multicolumn{1}{c}{Diff.}
    \\ 
    \hline \hline
 1 & 100 & 0.25 && 50.0 &   7.99 &&  2013.6 &50.0 &   0.01 & \hl{$-99.8 \pm $}& \hl{$ 0.6$} \\
   &     & 0.50 && 50.0 &   8.11 &&  2033.0 &49.9 &   0.03 & \hl{$-99.7 \pm $}& \hl{$ 0.6$} \\
   &     & 0.75 && 50.0 &   5.79 &&  1973.0 &50.0 &   0.19 & \hl{$-96.5 \pm $}& \hl{$ 3.2$} \\
 1 & 250 & 0.25 && 50.0 &  14.98 && 12529.9 &49.8 &   0.29 & \hl{$-98.1 \pm $}& \hl{$ 1.1$} \\
   &     & 0.50 && 50.0 &  11.95 && 12731.2 &49.4 &   0.10 & \hl{$-99.1 \pm $}& \hl{$ 1.0$} \\
   &     & 0.75 && 50.0 &   8.28 && 12552.8 &50.0 &   0.10 & \hl{$-98.8 \pm $}& \hl{$ 0.5$} \\
 1 & 500 & 0.25 && 49.8 &  18.07 && 51922.9 &48.7 &   0.31 & \hl{$-98.3 \pm $}& \hl{$ 1.1$} \\
   &     & 0.50 && 50.0 &  14.96 && 51008.0 &49.2 &   0.30 & \hl{$-98.0 \pm $}& \hl{$ 0.7$} \\
   &     & 0.75 && 50.0 &   9.91 && 50326.0 &50.0 &   0.72 & \hl{$-92.6 \pm $}& \hl{$ 4.2$} \\
 5 & 100 & 0.25 && 50.0 &  10.80 &&    23.8 &50.0 &   9.04 & \hl{$-15.7 \pm $}& \hl{$ 13.1$} \\
   &     & 0.50 && 50.0 &   9.19 &&    29.5 &50.0 &   8.04 & \hl{$-12.4 \pm $}& \hl{$ 7.0$} \\
   &     & 0.75 && 50.0 &  12.53 &&    26.7 &50.0 &  12.48 & \hl{$+  0.4 \pm$}& \hl{$ 14.7$} \\
 5 & 250 & 0.25 && 50.0 &  14.50 &&   156.7 &50.0 &  11.49 & \hl{$-20.2 \pm $}& \hl{$ 7.1$} \\
   &     & 0.50 && 50.0 &  11.26 &&   143.4 &50.0 &   8.60 & \hl{$-23.5 \pm $}& \hl{$ 3.8$} \\
   &     & 0.75 && 50.0 &  17.23 &&   148.4 &50.0 &  15.28 & \hl{$-10.2 \pm $}& \hl{$ 17.0$} \\
 5 & 500 & 0.25 && 50.0 &  17.28 &&   665.7 &49.8 &  12.06 & \hl{$-30.1 \pm $}& \hl{$ 4.0$} \\
   &     & 0.50 && 49.9 &  13.07 &&   610.0 &50.0 &   9.62 & \hl{$-26.4 \pm $}& \hl{$ 2.8$} \\
   &     & 0.75 && 50.0 &  17.71 &&   643.3 &50.0 &  14.23 & \hl{$-19.1 \pm $}& \hl{$ 9.5$} \\
10 & 100 & 0.25 && 50.0 &  16.54 &&     0.8 &49.9 &  16.00 & \hl{$ -3.1 \pm $}& \hl{$ 10.1$} \\
   &     & 0.50 && 49.9 &  13.79 &&     1.2 &50.0 &  14.44 & \hl{$+  5.1 \pm $}& \hl{$ 7.9$} \\
   &     & 0.75 && 50.0 &  22.49 &&     0.6 &50.0 &  22.89 & \hl{$+  2.5 \pm $}& \hl{$ 12.0$} \\
10 & 250 & 0.25 && 49.4 &  17.98 &&     5.2 &49.6 &  18.03 & \hl{$+  0.5 \pm $}& \hl{$ 5.4$} \\
   &     & 0.50 && 50.0 &  14.72 &&     4.0 &50.0 &  14.82 & \hl{$+  0.7 \pm $}& \hl{$ 5.0$} \\
   &     & 0.75 && 50.0 &  26.47 &&     2.9 &50.0 &  26.28 & \hl{$ -0.7 \pm $}& \hl{$ 7.0$} \\
10 & 500 & 0.25 && 49.7 &  18.58 &&    14.2 &49.2 &  18.27 & \hl{$ -1.6 \pm $}& \hl{$ 2.9$} \\
   &     & 0.50 && 50.0 &  15.67 &&    15.0 &49.9 &  15.61 & \hl{$ -0.3 \pm $}& \hl{$ 4.6$} \\
   &     & 0.75 && 50.0 &  27.10 &&    14.5 &50.0 &  26.67 & \hl{$ -0.4 \pm $}& \hl{$ 13.7$}\\
    \hline
  \end{tabular}
\end{table}

\subsubsection{Results}

Results are summarized in Table~\ref{tab:mkp_best_value}. 
\#FS stands for the number of feasible solutions.
Gap denotes the best optimality gap of solutions.
The results of the averaged optimality gap are omitted due to space limitations and are similar to those of the best optimality gap (see Appendix~\ref{app:full_mkp} for the full results).
$|E|$ denotes the number of ordered pairs of variables.
These metrics are averaged over 10 instances.
\hl{Mean reduction rate (RR) of the optimality gap, i.e., how much the gap is reduced by linearization, is reported with a standard deviation over 10 instances.}
$|E|$ is large when $n$ is large and $m$ is small.
This is a natural consequence, since 
the number of candidate pairs of variables to be ordered increases for increasing $n$ and conversely, larger $m$ leads to a tighter condition for a pair of variables to be ordered following Eq.~(\ref{eq:mkp_order}).
Meanwhile, Eq.~(\ref{eq:mkp_order}) depends only on $v_i$ and $w_{k,i}$ and not on $C_k$, and thus expectation values of $|E|$ do not depend on $\alpha$.
%On Method~1, 
The number of feasible solutions tends to be slightly decreased by applying the proposed method.
We consider that this is because the Ising machine outputs solutions violating constraints to avoid the auxiliary penalties introduced by the proposed method.
The increase in the number of infeasible solutions is negligibly small and thus is not a practical issue.
%We do not observe such increase of infeasible solutions on Method~2.
Regarding the optimality gap, the proposed method substantially decreases the gap %for both Method~1 and Method~2 
on instances with large $|E|$.
%In particular, the proposed method enables the Ising machine to reach the exact optimum on all instances of $m=1$ and $n=100$ with Method~2.
On the instances of $m=10$, % and $n=100,250$, 
we do not observe consistent improvement in performance due to small $|E|$.
We observe that the impact of the proposed method tends to be greater for smaller $\alpha$.
This suggests that the proposed method is particularly effective on problems with relatively tight inequality constraints.
Although the reason for this phenomenon is not clear, we consider it might be related to the scale of coefficients of the QUBO form Eq.~(\ref{eq:mkp_qubo}).
Since the coefficients of auxiliary variables in Eq.~(\ref{eq:inequality_penalty}) scale proportionally to knapsack capacity, large $\alpha$ leads to large coefficients of the objective function in Eq.~(\ref{eq:mkp_qubo}).
On the other hand, the proposed method also increases coefficients of linear terms through linearization.
Assuming that the distribution and scale of coefficients in the QUBO objective function influences performance of Ising machines, the difference in impact of the proposed method in $\alpha$ is expected to be explained in terms of the coefficients distribution.
Dependence of behavior of the baseline (without linearization) on $\alpha$ must be analyzed in the first place to verify the hypothesis.
Since such a study does not exist %so far
and is beyond the scope of this paper, we leave the precise analysis on dependence of performance on tightness of inequality constraints as future work.

%\section{Related Work}\label{sec:related_work}

\section{Related Work and Discussion}\label{sec:discussion}

We have proposed a method for deriving QUBO matrices suitable for Ising machines by introducing auxiliary penalties that preserve the optimum of given QUBO problems. %with appropriate coefficients.
No previous studies have taken such an approach so far.

\hl{
Numerous preprocessing techniques have been proposed for QUBO or equivalently max-cut problems~\mbox{\cite{boros2006preprocessing,lewis2017quadratic,rehfeldt2023faster}} that aim to reduce solving time on exact solvers or improve solution quality on meta-heuristic solvers.
All of these methods reduce the number of variables in a problem and can be applied to Ising machines.
%Among the methods, a technique used by Boros et al.~\cite{boros2006preprocessing} is possibly the closest to our method.
%They introduced a notion called \emph{quadratic persistency} for order relations between variables, which is similar to the notion of a valid partial order.
%Their method is to detect the quadratic persistency based on probing and second order derivatives
%for deducing an equality relation $x_i=x_j$ by expanding logical relations such as $x_i\le x_j, x_j\le x_k$ and $x_k \le x_i$.
Our approach in this work is different in that we reduce the connectivity in a problem instead of the number of variables, for utilizing Ising machines.
Note that the existing reduction methods can be combined with our method to further simplify problems.
For more details on the existing preprocessing methods, see Appendix~\mbox{\ref{app:review_preprocessing}}.
}

Several studies consider ordinal conditions in knapsack problems~\cite{you2007pegging} or in representation of integer variables via binary variables~\cite{tamura2009compiling,chancellor2019domain}.
These studies impose the ordinal conditions on variables as hard constraints given by problem definition.
Our approach is different in that ordinal conditions are derived as auxiliary penalties to improve applicability of Ising machines.
Whether linearization induces feasible solutions even on problems with hard ordinal constraints is an interesting question that will be explored in future work.

The proposed method has various directions of extension and application.
We explain some possible directions below to illustrate the potential impact of our approach utilizing auxiliary penalties.
Exploration of each %of the following directions
direction goes beyond the scope of this paper, i.e., verifying effectiveness of linearization on improving minor-embedding and Ising machine performance, and thus is left as future work.

One straightforward extension is to consider other implications such as $x_i=1 \Rightarrow x_j =0$ %rather 
than $x_i =1 \Rightarrow x_j = 1$.
Note that $x_i=1 \Rightarrow x_j =0$ is equivalent to $x_i x_j=0$, which yields a corresponding penalty term in a QUBO form.
When an auxiliary constraint $x_i =1 \Rightarrow x_j = 1$ can be imposed without changing the optimum of the QUBO problem, the corresponding quadratic term $x_i x_j$ with a \emph{negative} coefficient can be removed by adding the penalty $x_i x_j$ with a suitable coefficient.
We expect the process has similar effects on the sparsity and energy landscape of QUBO problems.
Exploring practical application of such an extension is an important direction of future studies.

The method can also be extended to Ising machines that can directly handle inequality constraints %, that is, 
without encoding them to a QUBO objective function as in Eq.~(\ref{eq:inequality_penalty}) and Eq.~(\ref{eq:kp_qubo}).
Fujitsu Digital Annealer of the 3rd generation~\cite{nakayama2021description} is one such Ising machine.
As we have seen in the case of MKPs, quadratic terms to be linearized often come from expanded polynomials of penalties of inequality constraints.
In such a case, the proposed method might not be applied as-is for Ising machines that directly handle inequality constraints.
Meanwhile, the method can be extended to the situation by considering an auxiliary constraint $x_i=1 \Rightarrow x_j = 1$ as an inequality constraint $x_i \le x_j$ and encoding it to the Ising machines, instead of converting it to a penalty term.
It is an interesting direction of investigation to evaluate the performance of Ising machines with the extended method.

A possible application of the proposed method is efficient processing of QUBO matrices.
When dealing with 
%a number of symmetric 
numerous variables in a QUBO problem, dense interaction of them results in large computational overhead for construction of the QUBO matrix both in time and space.
Symmetric variables admit a total order as in Example~\ref{ex:symmetric_order} and this nice property might help efficient construction of the linearized QUBO matrix, possibly skipping calculation of dense interaction.
Upon emergence of large-scale Ising machines, construction of QUBO matrices is considered to become a bottleneck of the whole process of using Ising machines.
Thus, this direction of application of the proposed method seems a promising approach to tackle this bottleneck.

\hl{
Lastly, we discuss the limitations of the proposed method.
We consider that the applicability of the proposed method depends on the problem class and the parameters in the problem.
%Expanding the applicability is an important direction of future research which we are working on, as mentioned above.
For further experimental exploration of the applicability, refer to Appendix~\mbox{\ref{app:experiments_mqlib}}, where we evaluate the proposed method on various domains such as max-cut and graph coloring problems.
We also discuss the extension of Theorem~\mbox{\ref{thm:sufficient_condition_valid_order}} to integer linear programming in Appendix~\mbox{\ref{app:sufficient_condition}}.
Moreover, as we have seen in the experiments on the MKP (Section~\mbox{\ref{sec:experiments}}), whether or not a (non-empty) valid partial order can be obtained heavily depends on the number of constraints.
To deal with problems with many constraints, we should first revisit the baseline, i.e., the naive application of Ising machines.
Current Ising machines typically suffer from being unable to solve problems with many constraints because of the difficulty to optimize multiple penalty terms, which can indeed be seen in our MKP experiments; the optimality gap obtained by the Ising machine becomes worse for large $m$ (the number of constraints).
Therefore, we consider that the limitation of the proposed method on such problems could be approached after further developments addressing this fundamental performance issue.
}

\section{Conclusion}\label{sec:conclusion}

We proposed linearization of quadratic unconstrained binary optimization (QUBO) problems using variable posets to improve applicability and performance of Ising machines.
Linearization eliminates quadratic terms in the objective function by introducing an auxiliary penalty of ordinal conditions with suitable coefficients, thereby simplifying the energy landscape of the QUBO problem and enhancing minor-embedding.
We developed general and practical algorithms to extract a valid partial order and demonstrated their computational complexity.
Through experiments on synthetic QUBO problems and MKPs, we validated the effects of the proposed method.
The results show that linearization mitigates performance degradation of Ising machines due to minor-embedding (possibly not by improving quality of minor-embedding but by simplifying energy landscape), enables minor-embedding of larger QUBO instances, and substantially reduces optimality gaps on practical problems.

\bibliographystyle{IEEEtran}
\bibliography{ref}

\begin{IEEEbiography}[{\includegraphics[width=1in,height=1.25in,clip,keepaspectratio]{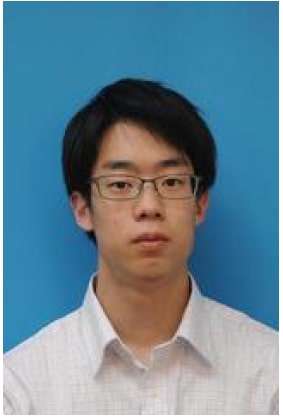}}]{Kentaro Ohno}
received the B.~S. and M.~S. degrees in mathematics from the University of Tokyo in 2017 and 2019, respectively. He joined NTT in 2019. He is currently studying combinatorial optimization using Ising machines.
\end{IEEEbiography}

\begin{IEEEbiography}[{\includegraphics[width=1in,height=1.25in,clip,keepaspectratio]{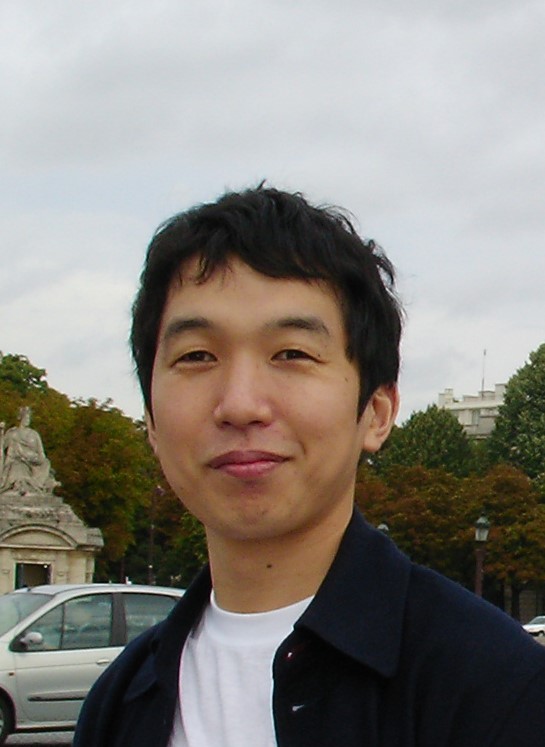}}]{Nozomu Togawa}
received the B.~Eng., M.~Eng., and Dr.~Eng. degrees from Waseda University in 1992, 1994, and 1997, respectively, all in electrical engineering. He is presently a professor in the Department of Computer Science and Communications Engineering, Waseda University. His research interests are quantum computation and integrated system design. He is a member of ACM, IEICE, and IPSJ.
\end{IEEEbiography}

\appendices

\section{\hl{
Review on Existing Preprocessing Techniques
}}\label{app:review_preprocessing}

\hl{
To clarify the novelty of the proposed method, 
we briefly review existing preprocessing on QUBO problems.
}

\hl{
As noted in Section~\mbox{\ref{sec:discussion}} in the main text, 
existing preprocessing techniques for QUBO (or max-cut)~\mbox{\cite{boros2006preprocessing,lewis2017quadratic,rehfeldt2023faster}} simplify problems by reducing the number of variables and are applicable to Ising machines.
A naive approach for reduction is variable fixation based on the \emph{first order derivative}~\mbox{\cite{boros2006preprocessing}} of the objective function, which is widely used in the existing methods~\mbox{\cite{lewis2017quadratic,rehfeldt2023faster}}.
A stronger preprocessing technique is based on \emph{roof-duality}~\mbox{\cite{hammer1984roof}}, which provides a lower (or upper) bound of the objective function called the roof dual.
For theoretical and technical details on the first order derivative and roof-duality, we refer to the review paper~\mbox{\cite{boros2002pseudo}} and references therein.
An important fact is that the computation of the roof dual results in a representation (precisely, a posiform) of the objective function with a nice property: if a linear term has a positive coefficient, the corresponding variable 
%variables with positive coefficients in that representation 
must take 0 at every minimum.
%As noted in the previous study~\cite{hammer1984roof}, 
The roof dual can be computed efficiently involving complexity of $O(n^3)$ for the problem size $n$.
Therefore, the roof-duality has been exploited to efficiently reduce the problem size by fixing the value of variables\footnote{Public implementation can be found in D-Wave Ocean SDK: https://github.com/dwavesystems/dwave-ocean-sdk}.
Note that the existing reduction methods can naturally be combined with our method by successively applying the preprocessings.
We also note that the roof-duality could not reduce the number of variables on all problem instances used in the main text.
This fact highlights the difference in applicability between the existing reduction methods and our proposed method.
}

\hl{
Boros et al.~\mbox{\cite{boros2006preprocessing}} integrated several preprocessing techniques including the use of the roof-duality and first/second order derivatives.
Their method is based on ideas of probing and consensus, resulting in an efficient and extensive reduction framework.
Lewis and Glover~\mbox{\cite{lewis2017quadratic}} proposed a variable fixation method based on basic four criteria including the first order derivative.
They showed that the reduction improves both solution quality and time to solution when applied with meta--heuristic solvers.
Rehfeldt et al.~\mbox{\cite{rehfeldt2023faster}} employed a combination of five criteria for variable reduction to enhance the performance of a branch-and-cut solver.
They verified that the combined use of various criteria is effective on instances of certain domains.
While all of these methods aim to reduce the number of variables, our approach is to reduce the connectivity of the problem to facilitate Ising machines.
}

%Their method is to detect the quadratic persistency based on probing and second order derivatives
%for deducing an equality relation $x_i=x_j$ by expanding logical relations such as $x_i\le x_j, x_j\le x_k$ and $x_k \le x_i$.
%Our approach in this work is different in that we reduce the connectivity in a problem instead of the number of variables.

\section{Proofs of Theoretical Results}

\subsection{Proof of Proposition~\ref{prop:aux_penalty}}\label{app:aux_penalty}
\begin{proposition*}[= Proposition~\ref{prop:aux_penalty}]
    Let $G$ be a partial order of variables valid with respect to minimization of a function $\phi: B_n \to \mathbb R$.
    Then, for any non-negative function $c: E\times B_n \to \mathbb R$, we have
    \begin{align}
        \min_{x\in B_n} \phi(x) = \min_{x\in B_n} \left( \phi(x) + \sum_{e\in E} c(e, x) (x_i - x_i x_j) \right).
    \end{align}
    Moreover, if $x^* \in B_n$ attains the minimum of the right hand side, then it attains the minimum of the left hand side.
\end{proposition*}

\begin{proof}
    We define $\psi(x) \coloneqq \phi(x) + \sum_{e\in E} c(e, x) (x_i - x_i x_j)$.
    Then $\psi(x) \ge \phi(x)$ always holds.
    Since $G$ is valid with respect to minimization of $\phi$, there exists $x^* \in B_n^G$ such that $\phi(x^*) = \min_{x\in B_n} \phi(x)$.
    We have an equality $\phi(x^*) = \psi(x^*)$ since $\phi(x) = \psi(x)$ for every $x \in B_n^G$.
    For any $x\in B_n$, we have $\psi(x) \ge \phi(x) \ge \phi(x^*) = \psi(x^*)$, so $x^*$ minimizes $\psi$.
    Therefore, Eq.~(\ref{eq:aux_penalty}) is proved.
    Conversely, if $x^{*}$ minimizes $\psi$, then
    $\psi(x^{*}) \ge \phi(x^*)$ and $\psi(x^{*}) = 
    \min_{x\in B_n} \psi(x) = \min_{x\in B_n} \phi(x)$ holds.
    Thus, we have $\phi(x^*) \le \min_{x\in B_n} \phi(x)$, and so $x^{*}$ minimizes $\phi$.    
\end{proof}

\subsection{Proof of Theorem~\ref{thm:sufficient_condition_valid_order}}\label{app:sufficient_condition}

\begin{theorem*}[= Theorem~\ref{thm:sufficient_condition_valid_order}]
    Let $G=(V,E)$ be a partial order of $n$ variables and
    $\phi(x)=x^\top Q x$ be a quadratic function with $Q\in \bR^{n\times n}$.
    For $i,j=1,2,\cdots,n$, we define
    \begin{align}
        a_{i,j} \coloneqq \begin{cases}
            Q_{i,j} + Q_{j,i} & i \ne j \\
            Q_{i,i} & i=j.
        \end{cases}
    \end{align}
    That is, $a_{i,j}$ is the coefficient of $x_i x_j$ (or $x_i$ if $i=j$) in $\phi(x)$.
    If an inequality
    \begin{align}
        S_{i,j} \coloneqq \sum_{\substack{k=1 \\ k\ne i,j}}^n \max \{ 0, a_{j,k} - a_{i,k} \} + a_{j,j} - a_{i,i} \le 0
    \end{align}
    holds for every directed edge $(x_i,x_j)\in E$, then $G$ is valid with respect to minimization of $\phi$.  
\end{theorem*}

\begin{proof}
    Assume $S_{i,j}\le 0$ for any directed edge $(x_i, x_j) \in E$.
    It suffices to show that for any $x\in B_n$, there exists $\tilde x\in B_n^G$ such that $\phi(\tilde x)\le \phi(x)$.
    Take $x \in B_n$.
    If $x \in B_n^G$, then there is nothing to show.
    Otherwise, there exists a directed edge $(x_i, x_j) \in E$ such that $x_i=1, x_j=0$.
    Let $F$ be a set of all directed edges $(x_i, x_j) \in E$ such that $x_i=1, x_j=0$.
    Take some $(x_i, x_j) \in F$ and define $x'\in B_n$ by
    \begin{align}
        x_k' = \begin{cases}
            0 & k=i \\
            1 & k=j \\
            x_k & k \ne i,j.
        \end{cases}
    \end{align}
    Then, we have 
    \begin{align}
        \phi(x') - \phi(x) &= \sum_{\substack{k=1 \\ k \ne i,j}}^n (a_{j,k} - a_{i,k})x_k + a_{j,j} - a_{i,i} \\
        &\le S_{i,j} \le 0.
    \end{align}
    Therefore, we may replace $x$ by $x'$.
    We claim that repetition of this replacement terminates in a finite number of iterations.
    To show this, we observe that a subgraph $G_F=(V,F)$ of $G=(V,E)$ is acyclic.
    By each replacement of $x$, the number of paths on $G_F$ is decreased by more than or equal to 1 since $G_F$ is acyclic.
    Thus, the number of repetitions must not exceed the number of paths on $G_F$, which is finite.
    After the termination of the repetition, we obtain $G_F$ with no edges.
    This implies that $F$ is empty, so we conclude that $x \in B_n^G$.
\end{proof}

\begin{remark}
\hl{
    In the above proof, we consider a case when $\phi(x_i=1,x_j=0) \ge \phi(x_i=0,x_j=1)$ holds uniformly, i.e., for any assignment of the other variables $x_k, k\ne i,j$.
    Similarly, we can also deduce that an order $x_i=1 \Rightarrow x_j=1$ is valid when $\phi(x_i=1,x_j=0) \ge \phi(x_i=1,x_j=1)$ or $\phi(x_i=1,x_j=0) \ge \phi(x_i=0,x_j=0)$ holds uniformly.
    However, in these cases, we can actually derive a stronger result: the value of $x_j$ or $x_i$ can be safely fixed.
    Indeed, let us assume $Q_{ij}>0$ and $\phi(x_i=1,x_j=0) \ge \phi(x_i=1,x_j=1)$ holds uniformly.
    Then, we have}
    \begin{align}
        0 & \le \phi(x_i=1,x_j=0) - \phi(x_i=1,x_j=1)\\
        &= -\sum_{k:k\ne i,j} Q_{jk}x_k - Q_{jj} - Q_{ij}.
    \end{align}
    \hl{Since $Q_{ij}>0$, we also have}
    \begin{align}
        %\mathcolorbox{yellow}{
        \phi(x_i=0,x_j=0) - \phi(x_i=0,x_j=1) \\
        = -\sum_{k:k\ne i,j} Q_{jk}x_k - Q_{jj} > 0.
        %}
    \end{align}
\hl{The inequalities imply that $\phi(x_j=0)\ge \phi(x_j=1)$ holds uniformly.
    Therefore, $x_j$ can be fixed to $1$ while preserving the optimum.
    Similarly, if $Q_{ij}>0$ and $\phi(x_i=1,x_j=0) \ge \phi(x_i=0,x_j=0)$ hold uniformly, we can fix $x_i$ to $0$.
    This type of fixation criterion is essentially included in most existing variable reduction methods~\mbox{\cite{boros2006preprocessing,lewis2017quadratic,rehfeldt2023faster}}, see also the discussion in Appendix~\mbox{\ref{app:review_preprocessing}} above.
    Thus, our approach does not claim anything beyond the previous work on these cases, so they are not dealt with in Theorem~\mbox{\ref{thm:sufficient_condition_valid_order}}.
}
\end{remark}

The above theorem and proof can be straightforwardly generalized to the following constrained setting,
which shows the validity of the partial order of variables given in Eq.~(\ref{eq:mkp_order}) in the main text as a direct corollary.

\begin{theorem}\label{thm:sufficient_condition_valid_order_constrained}
    Let $G=(V,E)$ be a partial order of $n$ variables and
    $\phi(x)=x^\top Q x$ be a quadratic function with $Q\in \bR^{n\times n}$.
    Let 
    $0 \le \sum_i w_{k,i} x_i \le C_k$ ($k=1,2,\cdots,m$) be linear inequality constraints with positive integers $w_{k,i}$ and $C_k$.
    Let $H_{\rm ineq}^{(k)}$ be the corresponding penalties given as in Eq.~(\ref{eq:inequality_penalty}) for each $k$.
    Let $\tilde \phi(x) \coloneqq \phi(x) + \sum_{k=1}^m \lambda_k H_{\rm ineq}^{(k)}$ be the QUBO objective function with sufficiently large $\lambda_k$ ($k=1,2,\cdots,m$) so that an optimal solution satisfies all constraints.
    For $i,j=1,2,\cdots,n$, we define
    \begin{align}
        a_{i,j} \coloneqq \begin{cases}
            Q_{i,j} + Q_{j,i} & i \ne j \\
            Q_{i,i} & i=j.
        \end{cases}
    \end{align}
    %That is, $a_{i,j}$ is the coefficient of $x_i x_j$ (or $x_i$ if $i=j$) in $\phi(x)$.
    We also define
    \begin{align}
        S_{i,j} \coloneqq \sum_{\substack{l=1 \\ l\ne i,j}}^n \max \{ 0, a_{j,l} - a_{i,l} \} + a_{j,j} - a_{i,i}.
    \end{align}
    If inequalities
    $ %\begin{align}
        S_{i,j} \le 0$ and $w_{k,i} \ge w_{k,j}
    $ %\end{align}
    for all $k=1,2,\cdots,m$ hold for every directed edge $(x_i,x_j)\in E$, then $G$ is valid with respect to minimization of $\tilde \phi$.  
\end{theorem}

\begin{proof}
    We set $C \coloneqq \{ x\in B_n \mid 0 \le \sum_i w_{k,i} x_i \le C_k \ \forall k=1,2,\cdots,m \} \subset B_n$ and $C^G \coloneqq C \cap B_n^G$.
    Assume $S_{i,j}\le 0$ and $w_{k,i} \ge w_{k,j}$
    for all $k=1,2,\cdots,m$ for any directed edge $(x_i, x_j) \in E$.
    It suffices to show that for any $x\in C$, there exists $\tilde x\in C^G$ such that $\phi(\tilde x)\le \phi(x)$.
    Take $x \in C$.
    If $x \in C^G$, then there is nothing to show.
    Otherwise, there exists a directed edge $(x_i, x_j) \in E$ such that $x_i=1, x_j=0$.
    Let $F$ be a set of all directed edges $(x_i, x_j) \in E$ such that $x_i=1, x_j=0$.
    Take some $(x_i, x_j) \in F$ and define $x'\in B_n$ by
    \begin{align}
        x_k' = \begin{cases}
            0 & k=i \\
            1 & k=j \\
            x_k & k \ne i,j.
        \end{cases}
    \end{align}
    Then, $x'$ satisfies all constraints by inequalities $w_{k,i} \ge w_{k,j}$.
    Thus, we have $x' \in C$. 
    Furthermore, we have 
    \begin{align}
        \phi(x') - \phi(x) &= \sum_{\substack{k=1 \\ k \ne i,j}}^n (a_{j,k} - a_{i,k})x_k + a_{j,j} - a_{i,i} \\
        &\le S_{i,j} \le 0.
    \end{align}
    Therefore, we may replace $x$ by $x'$.
    We claim that repetition of this replacement terminates in a finite number of iterations.
    To show this, we observe that a subgraph $G_F=(V,F)$ of $G=(V,E)$ is acyclic.
    By each replacement of $x$, the number of paths on $G_F$ is decreased by more than or equal to 1 since $G_F$ is acyclic.
    Thus, the number of repetitions must not exceed the number of paths on $G_F$, which is finite.
    After the termination of repetition, we obtain $G_F$ with no edges.
    This implies that $F$ is empty, so we conclude that $x \in C^G$.
\end{proof}

\begin{remark}
\hl{
    In Theorem~\mbox{\ref{thm:sufficient_condition_valid_order_constrained}} above, the positivity of weights $w_{k,i}$ and the form of inequality $0 \le \sum_i w_{k,i} x_i \le C_k$ are not critical assumptions.
    Indeed, we may consider more general inequality constraints $L_k \le \sum_i w_{k,i} x_i \le U_k$ with arbitrary (possibly negative) integers $w_{k,i}, L_k, U_k$ to obtain the same conclusion in Theorem~\mbox{\ref{thm:sufficient_condition_valid_order_constrained}} with no change in the proof.
    In other words, Theorem~\mbox{\ref{thm:sufficient_condition_valid_order_constrained}} can be generalized to \emph{integer linear programming (ILP) problems}.
    Regarding the applicability of the proposed method,
    the difference between the MKP case and general ILP case is that on the MKP, a quadratic term $x_i x_j$ can always be linearized if $x_i \Rightarrow x_j$ is valid, while it does not hold on ILP since the quadratic term might have a negative coefficient in the objective function (when $w_{k,i}$ or $w_{k,j}$ is negative for some $k$).
}
\end{remark}

\subsection{Proof of Theorem~\ref{thm:algorithm_completeness}}\label{app:algorithm_completeness}

We provide a proof of Theorem~\ref{thm:algorithm_completeness}.
For ease of notation, we define $a_i \coloneqq a_{i,i}$ and $c^+ \coloneqq \max\{0, c\}$ for a real number $c \in \mathbb R$.
We first show a lemma used in the proof.

\begin{lemma}\label{lem:cycle_is_symmetric}
    Let $l$ be a positive integer and $(a_{i,j})_{i,j=1,\cdots,l} \in \mathbb R^{l \times l}$ be a symmetric matrix.
    Assume inequality $a_{i,k} \ge a_{i+1,k}$ holds for any $i, k=1,2,\cdots,l$ with $k\ne i,i+1$.
    Then, for any $i$, we have 
    a chain of inequality
    \begin{align}
        a_{1,i} \ge \cdots \ge a_{i-1,i}
                \ge a_{i+1,i} \ge \cdots
                \ge a_{l,i}.
    \end{align}
\end{lemma}

\begin{proof}
    Take any $i\in \{1,\cdots,l\}$.
    From the assumption, we have 
    \begin{align}
        a_{1,i} \ge \cdots \ge a_{i-1,i}        
    \end{align}
    and
    \begin{align}
        a_{i+1,i} \ge \cdots \ge a_{l,i}.
    \end{align}
    Moreover, by symmetry of $(a_{i,j})$ and the assumption, we have
    \begin{align}
        a_{i-1,i} = a_{i,i-1} 
        \ge a_{i+1,i-1}
        = a_{i-1, i+1}
        \ge a_{i,i+1}
        = a_{i+1, i}.
    \end{align}
    Thus, we have a chain of inequality as desired.
    
\end{proof}

\begin{theorem*}[= Theorem~\ref{thm:algorithm_completeness}]
    Let $Q\in \bR^{n \times n}$ be a square real matrix and $x=(x_1, \cdots, x_n)$ be a vector of binary variables.
    Then, a graph $G=(V, E)$ with a node set $V=\{x_1, \cdots, x_n\}$ and an edge set $E$ obtained by running Algorithm~\ref{alg:extract_order} with $Q$ given as the input is a 
    %transitively closed, 
    directed acyclic graph.
    %In particular, $G$ is an order of variables $x$.
    Moreover, $G$ is valid with respect to minimization of a quadratic function $\phi(x) = x^\top Q x$.
\end{theorem*}

\begin{proof}
    To show that $G$ is acyclic, we first prove that a cycle contained in $G$ must consist of symmetric variables. 
    Assume that edges $(x_{i_1}, x_{i_2}), \cdots, (x_{i_{l-1}}, x_{i_l}), (x_{i_l}, x_{i_1}) \in E$ form a cycle.
    Since each edge satisfies Eq.~(\ref{eq:detect_valid_edge}), we have $a_{i_1} \ge a_{i_2} \ge \cdots \ge a_{i_l} \ge a_{i_1}$.
    Therefore, equality $a_{i_1} = a_{i_2} = \cdots = a_{i_l}$ holds.
    Furthermore, again by Eq.~(\ref{eq:detect_valid_edge}), we have $(a_{i_{j+1},k} - a_{i_j,k})^+=0$ for each $j=1,\cdots, l$ and any $k=1,\cdots,n$ with $ k\ne i,j$.
    Here, we set $i_{l+1}\coloneqq i_1$.
    Therefore, for $k \in \{1,\cdots,\} \setminus \{i_1, \cdots, i_l\}$, we have $a_{i_1, k} \ge a_{i_2, k} \ge \cdots a_{i_l, k} \ge a_{i_1, k}$, which then must be equality $a_{i_1, k} = a_{i_2, k} = \cdots a_{i_l, k}$.
    For $k \in \{i_1, \cdots, i_l\}$, we take $t$ such that $i_t =k$.
    Applying Lemma~\ref{lem:cycle_is_symmetric} with suitable reindexing, we get
    \begin{align}
    a_{i_1, k} \ge \cdots \ge a_{i_{t-1}, k} \ge  a_{i_{t+1}, k} \ge \cdots \ge a_{i_l, k} \ge a_{i_1, k},
    \end{align}
    which must be equality.
    Thus, $\phi(x)$ is symmetric with respect to $x_{i_1}, \cdots, x_{i_l}$.
    In Algorithm~\ref{alg:extract_order}, if $\phi(x)$ is symmetric with respect to $x_i$ and $x_j$ and $i<j$, then $(x_i, x_j) \in E$ and $(x_j, x_i) \notin E$ are ensured by line~\ref{state:detect_start}.
    Thus, $G$ is acyclic.
    Note that this proof explains that Algorithm~\ref{alg:extract_order} outputs a total order over symmetric variables $x_{i_1}, \cdots, x_{i_l}$ described in Example~\ref{ex:symmetric_order}.

    Validity of $G$ directly follows from Theorem~\ref{thm:sufficient_condition_valid_order}, since the edge set $E$ consists of edges satisfying Eq.~(\ref{eq:detect_valid_edge}).
\end{proof}

\section{Experiment Details}

\subsection{Computational Setup}\label{app:experiment_setup}

Source code for the experiments is written with Python~3.9.16 using D-Wave Ocean SDK\footnote{https://github.com/dwavesystems/dwave-ocean-sdk}~6.3.0 and Fixstars Amplify SDK~\cite{amplify}~0.11.1 libraries, except for Algorithm~\ref{alg:extract_order}.
To measure processing time,
Algorithm~\ref{alg:extract_order} is implemented with Cython~0.29.35 and compiled as code of C\texttt{++}.
Programs are run on MacBook Pro with Apple M2 chip and 8GB memory.
We use Amplify Annealing Engine (AE)~\cite{amplify} of version v0.7.3-A100 as an Ising machine with an execution time of 1 second.

\hl{
For minor-embedding search, we use the minorminer algorithm~\mbox{\cite{cai2014practical}} implemented on D-Wave Ocean SDK with a timeout of 1000 seconds.
%Performance of the search algorithm for minor-embedding might have significant effect on experimental results.
To enhance the performance of the minorminer algorithm, we adopt a method setting clique embedding as an initial state of the algorithm, which is simple yet effective as shown in a previous study~\mbox{\cite{zbinden2020embedding}}.
In preliminary experiments, we found an interesting fact that better results are occasionally yielded setting clique embedding of a complete graph of a size slightly less than $n$ as an initial state.
Thus, for QUBO instances of size $n=180$ and $n=232$, we run the minorminer algorithm with clique embedding of a complete graph with size each of $\{160, 180\}$ and $\{200, 232\}$
as the initial state, respectively, and then, we adopt the result with fewer auxiliary variables to report.
As a result, we observed that clique embedding of the smaller size indeed gives better results for synthetic instances of $p=1.5, 2.0$ on the first experiment.
On the second experiment testing the embeddability, %Similarly to in the previous experiment, 
we set clique embedding of a complete graph of size $180$ and $232$ as the initial state %of the minorminer algorithm 
for $P_{16}$ and $Z_{15}$, respectively.
When solving a QUBO problem after minor-embedding,
chain strength for minor-embedding is calculated with
\texttt{dwave.embedding.chain\_strength.} 
\texttt{uniform\_torque\_compensation()} function
%called by default in \texttt{dwave.embedding.embed\_qubo()} function 
in Ocean SDK.
Note that energy for a minor-embedded QUBO problem can be defined in two ways: raw energy of solutions of the embedded problem and \emph{unembedded} energy of solutions on the original problem decoded by fixing broken chains in some way.
We observed that %raw energy and unembedded energy coincide for every
no broken chains appear in 
all solutions in the experiments in the main text, and thus they are not distinguished.
}

\hl{
For the experiments on the MKP, we program the penalty term Eq.~(\mbox{\ref{eq:inequality_penalty}}) of the inequality constraints using the Amplify SDK.
We adopt two slightly different ways of implementing each penalty term:
}
\begin{enumerate}[leftmargin=1.75cm,label=\hl{Method~{{\arabic*}}}.]
    \item \hl{Penalty terms are programmed as-is by explicitly defining auxiliary variables.}
    \item \hl{Penalty terms are defined using \texttt{less\_equal()} function in Amplify SDK~\mbox{\cite{amplify}} %which is implemented to program inequality constraints 
    without explicitly defining auxiliary variables.}
\end{enumerate}
\hl{
We call the above methods Method~1 and Method~2, respectively.
In Method~1, the whole objective function of the MKP in QUBO form with the proposed method is given as follows:}
\begin{align}\label{eq:linearized_mkp_app}
    H_{\rm lin,1} &= -\sum v_i x_i %+ \lambda \sum_{k=1}^m H^{(k)}_{\rm lin,ineq}, 
    \notag \\
    + \lambda &\sum_{k=1}^m \left(H^{(k)}_{\rm ineq} + \sum_{(i,j)\in E} 2w_{k,i}w_{k,j}(x_i - x_i x_j) \right).
    %H^{(k)}_{\rm lin,ineq} &= H^{(k)}_{\rm ineq} + \sum_{(i,j)\in E} 2w_{k,i}w_{k,j}(x_i - x_i x_j).
\end{align}
\hl{
The results for Method~1 are reported in the main text.
In Method~2, 
\texttt{less\_equal()} function in Amplify SDK serves functionality to conceal auxiliary variables and enable users to avoid tedious tasks to program penalties of constraints.
The implementation of the function is undisclosed and might be specially designed to enhance performance of the Ising machine.
While an object returned by the function can be added to usual polynomial objects or multiplied by a scalar, users cannot access and change the penalty term.
Therefore, we cannot linearize each penalty as in Eq.~(\mbox{\ref{eq:linearized_mkp_app}}).
In the experiment, we apply the proposed method in Method~2 by replacing $H^{(k)}_{\rm ineq}$ by \texttt{less\_equal()} function in Eq.~(\mbox{\ref{eq:linearized_mkp_app}}) and reformulating as follows:
}
\begin{align}
    H_{\rm lin,2} =& -\sum v_i x_i \notag %\\ &
    + \lambda \sum_{k=1}^m \texttt{less\_equal}\bigl(\sum_{i} w_{k,i}x_i, C_k \bigr) \notag \\
    & +\lambda \sum_{k=1}^m \sum_{(i,j)\in E} 2w_{k,i}w_{k,j}(x_i - x_i x_j),
\end{align}
\hl{
where $\texttt{less\_equal}\bigl(\sum_{i} w_{k,i}x_i, C_k \bigr)$ represents the inequality constraint $\sum_{i} w_{k,i}x_i \le C_k$.
We report the results for Method~2 in Appendix~\mbox{\ref{app:full_mkp}}.
}

\begin{comment}

\section{Combined Algorithm}

\begin{algorithm}[t]
 \caption{Extraction and Linearization}
 \begin{algorithmic}[1]
 \renewcommand{\algorithmicrequire}{\textbf{Input:}}
 \renewcommand{\algorithmicensure}{\textbf{Output:}}
 \REQUIRE{QUBO matrix $Q\in \bR^{n\times n}$}
 \ENSURE{Linearized QUBO matrix $Q^G$ with respect to order $G=(V,E)$ given by Algorithm~\ref{alg:extract_order}}
  \STATE $a \leftarrow Q + Q^\top$
  %\FOR {$i=1,2,\cdots,n$}
  %  \STATE $a_{i,i} = Q_{i,i}$
  %\ENDFOR
  %\STATE $E \leftarrow \varnothing$
  \FOR {$i=1,2,\cdots,n$}
    \FOR {$j=1,2,\cdots,n$}
      \IF {$i\ne j$ and $a_{i,j}>0$ and $Q_{j,j} \le Q_{i,i}$} %\label{state:detect_start}
        \STATE $S \leftarrow Q_{j,j} - Q_{i,i}$ %\label{state:calculate_score_start}
        \FOR {$k=1,2,\cdots,n$}
          \IF {$k\ne i,j$ and $a_{j,k}>a_{i,k}$}
            \STATE $S \leftarrow S + a_{j,k}-a_{i,k}$
            \IF {$S>0$} %\label{state:pruning_start}
              \STATE \textbf{break}
            \ENDIF %\label{state:pruning_end}
          \ENDIF
        \ENDFOR %\label{state:calculate_score_end}
        \IF {$S \le 0$} %\label{state:append_edge_start}
          %\STATE $E \leftarrow E \cup \{(x_i, x_j)\}$
          \STATE $Q_{i,i} \leftarrow Q_{i,i} + a_{i,j}$
          \STATE $Q_{i,j} \leftarrow 0,\ Q_{j,i} \leftarrow 0, \ a_{i,j} \leftarrow 0$
        \ENDIF %\label{state:append_edge_end}
      \ENDIF %\label{state:detect_end}
    \ENDFOR
  \ENDFOR
  \RETURN $Q$
 \end{algorithmic} 
\end{algorithm}

\end{comment}

\begin{table*}[!ht]
\begin{minipage}[t]{\textwidth}
  \caption{Optimality Gap on MKPs for Baseline and Proposed Method Based on Method 1.}
  \label{tab:mkp_best_value_case1_full}
  \centering
  \begin{tabular}{c@{\hspace{3pt}}c@{\hspace{4pt}}crrrrrrrr@{\hspace{5pt}}r@{\hspace{2pt}}rr@{\hspace{5pt}}r@{\hspace{2pt}}r}
    \hline
    %& & & & \multicolumn{10}{c}{Method 1}
    %\\ 
    %\cline{5-16}
    & Instance & & & \multicolumn{3}{c}{Baseline} & & \multicolumn{8}{c}{Linearized}
    \\
    \cline{1-3}
    \cline{5-7}
    \cline{9-16}
    $m$ & $n$ & $\alpha$ 
    & & \multicolumn{1}{c}{\#FS} & Avg. & Best 
    & & \multicolumn{1}{c}{$|E|$} & \multicolumn{1}{c}{\#FS} & Avg. & \multicolumn{2}{c}{\hl{RR (\%)}} & Best  &  \multicolumn{2}{c}{\hl{RR (\%)}}
    \\ 
    \hline \hline
 1 & 100 & 0.25 && 50.0 & 14.940&  7.995 &&  2013.6 &50.0 &  0.501 & \hl{$-96.7 \pm $}&\hl{$ 0.9$} & 0.014 & \hl{$-99.8 \pm $}&\hl{$ 0.6$} \\
   &     & 0.50 && 50.0 & 14.249&  8.110 &&  2033.0 &49.9 &  1.782 & \hl{$-87.0 \pm $}&\hl{$ 13.5$} & 0.026 & \hl{$-99.7 \pm $}&\hl{$ 0.6$} \\
   &     & 0.75 && 50.0 & 13.851&  5.792 &&  1973.0 &50.0 &  6.447 & \hl{$-51.9 \pm $}&\hl{$ 32.8$} & 0.195 & \hl{$-96.5 \pm $}&\hl{$ 3.2$} \\
 1 & 250 & 0.25 && 50.0 & 20.973& 14.983 && 12529.9 &49.8 &  2.137 & \hl{$-89.8 \pm $}&\hl{$ 3.8$} & 0.285 & \hl{$-98.1 \pm $}&\hl{$ 1.1$} \\
   &     & 0.50 && 50.0 & 17.378& 11.948 && 12731.2 &49.4 &  1.682 & \hl{$-90.5 \pm $}&\hl{$ 5.4$} & 0.097 & \hl{$-99.1 \pm $}&\hl{$ 1.0$} \\
   &     & 0.75 && 50.0 & 13.697&  8.278 && 12552.8 &50.0 &  3.066 & \hl{$-77.2 \pm $}&\hl{$ 7.2$} & 0.098 & \hl{$-98.8 \pm $}&\hl{$ 0.5$} \\
 1 & 500 & 0.25 && 49.8 & 21.422& 18.074 && 51922.9 &48.7 &  1.560 & \hl{$-92.7 \pm $}&\hl{$ 3.4$} & 0.307 & \hl{$-98.3 \pm $}&\hl{$ 1.1$} \\
   &     & 0.50 && 50.0 & 17.030& 14.964 && 51008.0 &49.2 &  1.365 & \hl{$-92.1 \pm $}&\hl{$ 4.4$} & 0.298 & \hl{$-98.0 \pm $}&\hl{$ 0.7$} \\
   &     & 0.75 && 50.0 & 13.538&  9.908 && 50326.0 &50.0 &  3.437 & \hl{$-74.7 \pm $}&\hl{$ 5.5$} & 0.719 & \hl{$-92.6 \pm $}&\hl{$ 4.2$} \\
 5 & 100 & 0.25 && 50.0 & 15.386& 10.798 &&    23.8 &50.0 & 13.792 & \hl{$-10.3 \pm $}&\hl{$ 4.2$} & 9.044 & \hl{$-15.7 \pm $}&\hl{$ 13.1$} \\
   &     & 0.50 && 50.0 & 13.048&  9.191 &&    29.5 &50.0 & 11.502 & \hl{$-11.5 \pm $}&\hl{$ 6.1$} & 8.039 & \hl{$-12.4 \pm $}&\hl{$ 7.0$} \\
   &     & 0.75 && 50.0 & 22.991& 12.533 &&    26.7 &50.0 & 22.996 & \hl{$+  0.3 \pm $}&\hl{$ 8.7$} &12.481 & \hl{$+  0.4 \pm $}&\hl{$ 14.7$} \\
 5 & 250 & 0.25 && 50.0 & 17.190& 14.497 &&   156.7 &50.0 & 13.897 & \hl{$-19.1 \pm $}&\hl{$ 3.0$} &11.492 & \hl{$-20.2 \pm $}&\hl{$ 7.1$} \\
   &     & 0.50 && 50.0 & 13.401& 11.257 &&   143.4 &50.0 & 10.666 & \hl{$-20.2 \pm $}&\hl{$ 5.4$} & 8.598 & \hl{$-23.5 \pm $}&\hl{$ 3.8$} \\
   &     & 0.75 && 50.0 & 24.211& 17.234 &&   148.4 &50.0 & 22.012 & \hl{$ -9.0 \pm $}&\hl{$ 2.7$} &15.283 & \hl{$-10.2 \pm $}&\hl{$ 17.0$} \\
 5 & 500 & 0.25 && 50.0 & 18.885& 17.282 &&   665.7 &49.8 & 13.930 & \hl{$-26.2 \pm $}&\hl{$ 2.6$} &12.064 & \hl{$-30.1 \pm $}&\hl{$ 4.0$} \\
   &     & 0.50 && 49.9 & 17.941& 13.074 &&   610.0 &50.0 & 14.092 & \hl{$-21.1 \pm $}&\hl{$ 8.9$} & 9.619 & \hl{$-26.4 \pm $}&\hl{$ 2.8$} \\
   &     & 0.75 && 50.0 & 25.162& 17.710 &&   643.3 &50.0 & 21.479 & \hl{$-14.6 \pm $}&\hl{$ 5.0$} &14.228 & \hl{$-19.1 \pm $}&\hl{$ 9.5$} \\
10 & 100 & 0.25 && 50.0 & 24.967& 16.539 &&     0.8 &49.9 & 25.327 & \hl{$+  1.7 \pm $}&\hl{$ 5.4$} &16.004 & \hl{$ -3.1 \pm $}&\hl{$ 10.1$} \\
   &     & 0.50 && 49.9 & 20.688& 13.788 &&     1.2 &50.0 & 21.432 & \hl{$+  3.7 \pm $}&\hl{$ 4.5$} &14.443 & \hl{$+  5.1 \pm $}&\hl{$ 7.9$} \\
   &     & 0.75 && 50.0 & 30.995& 22.485 &&     0.6 &50.0 & 30.726 & \hl{$ -0.8 \pm $}&\hl{$ 3.1$} &22.885 & \hl{$+  2.5 \pm $}&\hl{$ 12.0$} \\
10 & 250 & 0.25 && 49.4 & 21.252& 17.981 &&     5.2 &49.6 & 20.997 & \hl{$ -1.1 \pm $}&\hl{$ 3.7$} &18.026 & \hl{$+  0.5 \pm $}&\hl{$ 5.4$} \\
   &     & 0.50 && 50.0 & 18.601& 14.723 &&     4.0 &50.0 & 18.480 & \hl{$ -0.5 \pm $}&\hl{$ 4.9$} &14.819 & \hl{$+  0.7 \pm $}&\hl{$ 5.0$} \\
   &     & 0.75 && 50.0 & 34.022& 26.468 &&     2.9 &50.0 & 34.030 & \hl{$+  0.0 \pm $}&\hl{$ 1.0$} &26.280 & \hl{$ -0.7 \pm $}&\hl{$ 7.0$} \\
10 & 500 & 0.25 && 49.7 & 20.623& 18.580 &&    14.2 &49.2 & 20.524 & \hl{$ -0.5 \pm $}&\hl{$ 1.4$} &18.266 & \hl{$ -1.6 \pm $}&\hl{$ 2.9$} \\
   &     & 0.50 && 50.0 & 23.015& 15.670 &&    15.0 &49.9 & 21.769 & \hl{$ -5.3 \pm $}&\hl{$ 5.6$} &15.609 & \hl{$ -0.3 \pm $}&\hl{$ 4.6$} \\
   &     & 0.75 && 50.0 & 33.994& 27.096 &&    14.5 &50.0 & 33.799 & \hl{$ -0.6 \pm $}&\hl{$ 3.5$} &26.668 & \hl{$ -0.4 \pm $}&\hl{$ 13.7$} \\
    \hline
  \end{tabular}
%\end{table*}
\end{minipage}
\bigskip
\begin{minipage}[t]{\textwidth}
%\begin{table*}[t]
  \caption{Optimality Gap on MKPs for Baseline and Proposed Method Based on Method 2.}
  \label{tab:mkp_best_value_case2_full}
  \centering
  \begin{tabular}{c@{\hspace{3pt}}c@{\hspace{4pt}}crrrrrrrr@{\hspace{5pt}}r@{\hspace{2pt}}rr@{\hspace{5pt}}r@{\hspace{2pt}}r}
    \hline
    %& & & & \multicolumn{10}{c}{Method 2}
    %\\ 
    %\cline{5-16}
    & Instance & & & \multicolumn{3}{c}{Baseline} & & \multicolumn{8}{c}{Linearized}
    \\
    \cline{1-3}
    \cline{5-7}
    \cline{9-16}
    $m$ & $n$ & $\alpha$ 
    & & \multicolumn{1}{c}{\#FS} & Avg. & Best 
    & & \multicolumn{1}{c}{$|E|$} & \multicolumn{1}{c}{\#FS} & Avg. & \multicolumn{2}{c}{\hl{RR (\%)}} & Best  &  \multicolumn{2}{c}{\hl{RR (\%)}}
    \\ 
    \hline \hline
 1 & 100 & 0.25 && 50.0 &  6.616&  3.287 &&  2013.6 &50.0 &  0.004 & \hl{$-99.9 \pm $}&\hl{$ 0.1$} & 0.000 & \hl{$-100.0 \pm $}&\hl{$ 0.0$} \\
   &     & 0.50 && 50.0 &  6.447&  4.149 &&  2033.0 &50.0 &  0.046 & \hl{$-99.3 \pm $}&\hl{$ 0.9$} & 0.000 & \hl{$-100.0 \pm $}&\hl{$ 0.0$} \\
   &     & 0.75 && 50.0 &  5.194&  3.189 &&  1973.0 &50.0 &  0.062 & \hl{$-98.8 \pm $}&\hl{$ 1.3$} & 0.000 & \hl{$-100.0 \pm $}&\hl{$ 0.0$} \\
 1 & 250 & 0.25 && 50.0 & 14.230& 10.524 && 12529.9 &50.0 &  0.693 & \hl{$-95.2 \pm $}&\hl{$ 2.0$} & 0.216 & \hl{$-98.0 \pm $}&\hl{$ 1.4$} \\
   &     & 0.50 && 50.0 & 11.780&  9.643 && 12731.2 &50.0 &  0.346 & \hl{$-97.1 \pm $}&\hl{$ 0.8$} & 0.117 & \hl{$-98.8 \pm $}&\hl{$ 0.6$} \\
   &     & 0.75 && 50.0 &  9.513&  5.512 && 12552.8 &50.0 &  1.829 & \hl{$-80.8 \pm $}&\hl{$ 3.0$} & 0.213 & \hl{$-95.8 \pm $}&\hl{$ 2.4$} \\
 1 & 500 & 0.25 && 50.0 & 18.361& 14.414 && 51922.9 &50.0 &  1.711 & \hl{$-90.7 \pm $}&\hl{$ 2.7$} & 0.721 & \hl{$-95.0 \pm $}&\hl{$ 1.8$} \\
   &     & 0.50 && 50.0 & 13.650& 11.995 && 51008.0 &50.0 &  1.067 & \hl{$-92.2 \pm $}&\hl{$ 0.9$} & 0.434 & \hl{$-96.4 \pm $}&\hl{$ 1.0$} \\
   &     & 0.75 && 50.0 & 11.126&  7.268 && 50326.0 &50.0 &  3.746 & \hl{$-66.3 \pm $}&\hl{$ 2.0$} & 1.324 & \hl{$-80.3 \pm $}&\hl{$ 8.2$} \\
 5 & 100 & 0.25 && 50.0 & 13.631&  9.610 &&    23.8 &50.0 & 12.198 & \hl{$-10.5 \pm $}&\hl{$ 2.9$} & 8.527 & \hl{$-10.4 \pm $}&\hl{$ 10.9$} \\
   &     & 0.50 && 50.0 & 11.287&  7.821 &&    29.5 &50.0 &  9.753 & \hl{$-13.5 \pm $}&\hl{$ 2.6$} & 7.241 & \hl{$ -6.8 \pm $}&\hl{$ 10.3$} \\
   &     & 0.75 && 50.0 & 17.959& 11.157 &&    26.7 &50.0 & 16.865 & \hl{$ -6.0 \pm $}&\hl{$ 3.5$} &11.440 & \hl{$+  3.2 \pm $}&\hl{$ 15.1$} \\
 5 & 250 & 0.25 && 50.0 & 16.201& 13.624 &&   156.7 &50.0 & 13.217 & \hl{$-18.4 \pm $}&\hl{$ 2.6$} &11.147 & \hl{$-18.1 \pm $}&\hl{$ 4.0$} \\
   &     & 0.50 && 50.0 & 12.585& 10.572 &&   143.4 &50.0 & 10.072 & \hl{$-19.9 \pm $}&\hl{$ 3.7$} & 8.071 & \hl{$-23.5 \pm $}&\hl{$ 6.9$} \\
   &     & 0.75 && 50.0 & 21.923& 15.029 &&   148.4 &50.0 & 19.428 & \hl{$-11.3 \pm $}&\hl{$ 1.6$} &12.665 & \hl{$-14.9 \pm $}&\hl{$ 13.0$} \\
 5 & 500 & 0.25 && 50.0 & 17.841& 16.020 &&   665.7 &50.0 & 12.905 & \hl{$-27.6 \pm $}&\hl{$ 2.3$} &11.582 & \hl{$-27.7 \pm $}&\hl{$ 2.2$} \\
   &     & 0.50 && 50.0 & 14.363& 12.695 &&   610.0 &50.0 & 10.925 & \hl{$-23.9 \pm $}&\hl{$ 1.3$} & 9.107 & \hl{$-28.3 \pm $}&\hl{$ 3.8$} \\
   &     & 0.75 && 50.0 & 21.767& 15.699 &&   643.3 &50.0 & 18.709 & \hl{$-14.0 \pm $}&\hl{$ 2.5$} &12.368 & \hl{$-20.1 \pm $}&\hl{$ 12.5$} \\
10 & 100 & 0.25 && 50.0 & 19.014& 14.779 &&     0.8 &50.0 & 18.853 & \hl{$ -0.8 \pm $}&\hl{$ 1.2$} &15.288 & \hl{$+  4.8 \pm $}&\hl{$ 12.2$} \\
   &     & 0.50 && 50.0 & 17.700& 13.802 &&     1.2 &50.0 & 17.489 & \hl{$ -1.2 \pm $}&\hl{$ 1.7$} &13.148 & \hl{$ -4.6 \pm $}&\hl{$ 5.4$} \\
   &     & 0.75 && 50.0 & 25.902& 19.740 &&     0.6 &50.0 & 25.973 & \hl{$+  0.3 \pm $}&\hl{$ 1.9$} &20.226 & \hl{$+  3.8 \pm $}&\hl{$ 14.1$} \\
10 & 250 & 0.25 && 50.0 & 19.629& 17.475 &&     5.2 &50.0 & 19.429 & \hl{$ -1.0 \pm $}&\hl{$ 0.8$} &17.041 & \hl{$ -2.4 \pm $}&\hl{$ 4.7$} \\
   &     & 0.50 && 50.0 & 17.495& 14.297 &&     4.0 &50.0 & 17.311 & \hl{$ -1.0 \pm $}&\hl{$ 1.5$} &14.439 & \hl{$+  1.3 \pm $}&\hl{$ 5.5$} \\
   &     & 0.75 && 50.0 & 31.308& 24.759 &&     2.9 &50.0 & 31.344 & \hl{$+  0.1 \pm $}&\hl{$ 2.0$} &24.898 & \hl{$+  0.9 \pm $}&\hl{$ 8.5$} \\
10 & 500 & 0.25 && 50.0 & 19.887& 18.186 &&    14.2 &50.0 & 19.542 & \hl{$ -1.7 \pm $}&\hl{$ 0.9$} &17.702 & \hl{$ -2.7 \pm $}&\hl{$ 3.8$} \\
   &     & 0.50 && 50.0 & 18.908& 15.287 &&    15.0 &50.0 & 18.550 & \hl{$ -1.9 \pm $}&\hl{$ 1.9$} &15.124 & \hl{$ -0.9 \pm $}&\hl{$ 5.1$} \\
   &     & 0.75 && 50.0 & 32.709& 27.397 &&    14.5 &50.0 & 32.453 & \hl{$ -0.8 \pm $}&\hl{$ 0.8$} &26.839 & \hl{$ -2.0 \pm $}&\hl{$ 5.6$} \\
    \hline
  \end{tabular}
\end{minipage}
\end{table*}

\subsection{Properties of Synthetic QUBO Instances}\label{app:synthetic_qubo}

We explain properties of synthetic QUBO instances described in the main text.

\hl{
The generated QUBO matrix has the following properties:
}
{
\begin{enumerate}
    \item \hl{All off-diagonal upper-triangle entries have positive values. In particular, a graph associated with the matrix is complete.}
    \item \hl{For large $p$, the variance of diagonal entries is large.
    In particular, difference $|Q_{i,i}-Q_{j,j}|$ between two diagonal entries tends to be large.}
\end{enumerate}
}
\hl{
From Property~1, a generated QUBO instance requires a number of additional variables for minor-embedding to a sparse hardware graph.
From Property~2, the number of pairs $(x_i, x_j)$ with $S_{i,j} \le 0$ is expected to be increased by increasing $p$.
On the basis of Theorem~\mbox{\ref{thm:sufficient_condition_valid_order}}, this indicates that large $p$ leads to a large number of ordered pairs of variables obtained by Algorithm~\mbox{\ref{alg:extract_order}}.
Moreover, from Property~1, all off-diagonal entries corresponding to ordered pairs are reduced to diagonal entries through linearization.
In summary, by setting $p$ larger, the proposed method is expected to have greater effects.

We argue the other parameters for the generation.
First, we consider the offset parameter $o=sp$ pushing the entries away from the origin.
Note that the expectation value of sum $\sum_{j=1}^{i-1} Q_{j,i} + \sum_{j=i+1}^n Q_{i,j}$ of interactions between a variable $x_i$ and the others is $0.5(n-1)(s+1)+(n-1)o$.
Suppose we take %the offset $o$ of each entry as 
$o=0$ to generate QUBO instances.
If $p$ is larger than 0.5, the scale $p(n-1)(s+1)$ of coefficients of linear terms is larger than the above expected total interaction.
Then, for large $p$, values of several variables can be determined as 1 without solving, since a diagonal entry $Q_{i,i}$ is negatively too much larger than off-diagonal entries $Q_{i,j}, Q_{j,i}$.
To avoid such triviality of the problems, we set a positive offset $o=sp$.
Indeed, the roof-duality method (see Appendix~\mbox{\ref{app:review_preprocessing}}) could not reduce the size of generated QUBO instances.
}
Note that $o$ has no effect on the ordered pairs obtained by Algorithm~\mbox{\ref{alg:extract_order}}.
Next, we consider the other parameter $s$.
The parameter $s$ determines the whole scale of QUBO matrices and is considered to have no impact on effects of the proposed methods or on solvability of QUBO matrices (at least for $s$ that is not too small).
Thus, $s$ can be taken arbitrarily in the experiments.

\subsection{Full Results on MKPs}\label{app:full_mkp}

Tables~\ref{tab:mkp_best_value_case1_full} and \ref{tab:mkp_best_value_case2_full} present the full results including the averaged optimality gap (shown in ``Avg.'' column) on MKPs for Methods~1 and 2, respectively.
\hl{
The results of the number of feasible solutions and best optimality gap for Method~1 correspond to the results in the main text. %are re-printed for convenience.
}
We observe a similar trend of improving the averaged gap compared with the best gap.

\hl{
As noted in the main text,
on Method~1, 
the number of feasible solutions tends to slightly decrease by applying the proposed method.
We do not observe such an increase in infeasible solutions for Method~2.
The optimality gap for Method~2 is substantially reduced as well as that for Method~1 for large $|E|$.
In particular, the proposed method enables the Ising machine to reach the exact optimum on all instances of $m=1$ and $n=100$ with Method~2.
}

\section{Additional Experiments}

\begin{table*}[t]
  \caption{Effect of Linearization on Minor-Embedding.}
  \label{tab:king_random_num_order}
  \centering
  \begin{tabular}{cccrrr@{\hspace{3pt}}r r@{\hspace{3pt}}r r@{\hspace{3pt}}r}
    \hline
    $H$ & $n$ &&& $|E|$ & \multicolumn{2}{c}{${\rm OD}(Q^G)$}
    & \multicolumn{2}{c}{\# of Aux. Var.} & \multicolumn{2}{c}{Chain Length} \\
    \hline \hline
    \multirow{7}{*}{$KG_{256,256}$} &
    \multirow{7}{*}{257} &
    \multicolumn{2}{c}{clique} & - & 32896 & & 65279 & & 256 &\\
    %\cline{2-6}
    &&     & 0.1 &     0.0 & 32896.0 & $(-0.0\%)$ & 65279.0 & $(-0.0\%)$ & 256.0 & $(-0.0\%)$ \\
    &&     & 0.2 &  2519.0 & 30376.9 & $(-7.7\%)$ & 65279.0 & $(-0.0\%)$ & 256.0 & $(-0.0\%)$ \\
    && $p$ & 0.5 & 16347.3 & 16548.7 & $(-49.7\%)$ & 65279.0 & $(-0.0\%)$ & 256.0 & $(-0.0\%)$ \\
    &&     & 1.0 & 23881.7 &  9014.3 & $(-72.6\%)$ & 65279.0 & $(-0.0\%)$ & 256.0 & $(-0.0\%)$ \\
    &&     & 1.5 & 26653.9 &  6242.1 & $(-72.6\%)$ & 65279.0 & $(-0.0\%)$ & 256.0 & $(-0.0\%)$ \\
    &&     & 2.0 & 28160.8 &  4735.2 & $(-85.6\%)$ & 65278.8 & $(-0.0\%)$ & 256.0 & $(-0.0\%)$ \\
    \hline
  \end{tabular}
\end{table*}

\begin{table*}[t]
  \caption{Energy of Solutions of Synthetic QUBO Problems.}
  \label{tab:king_random_ising_machine}
  \centering
  \begin{tabular}{cccrrr@{\hspace{3pt}}rr@{\hspace{3pt}}r}
    \hline
    &&& Without && \multicolumn{4}{c}{With Embedding}\\
    \cline{6-9}
    $H$ & $n$ & $p$ & Embedding & & 
    %\multicolumn{2}{c}{Clique} &  \multicolumn{2}{c}{Linearized} 
    \multicolumn{1}{c}{Baseline} & \multicolumn{1}{c}{Diff.} &
    \multicolumn{1}{c}{Linearized} & \multicolumn{1}{c}{Diff.} \\
    \hline \hline
    \multirow{5}{*}{$KG_{256,256}$} &
    \multirow{5}{*}{257}
     & 0.2 & $ -16991.0$ &  & $  -3636.6$ &  ($+13354.4$) & $  -9260.7$ & ($+7730.3$) \\ % -9524.9
    && 0.5 & $ -63406.1$ &  & $ -52692.8$ &  ($+10713.3$) & $ -60251.9$ & ($+3154.2$) \\ % -59903.5
    && 1.0 & $-151207.4$ &  & $-143336.2$ &   ($+7871.2$) & $-145167.8$ & ($+6039.6$) \\ % -145042.4
    && 1.5 & $-243713.6$ &  & $-236933.5$ &   ($+6780.1$) & $-240049.8$ & ($+3663.8$) \\ % -239986.3
    && 2.0 & $-336410.5$ &  & $-329740.5$ &   ($+6670.0$) & $-333827.0$ & ($+2583.5$) \\ % -333823.5
    \hline
  \end{tabular}
\end{table*}

\subsection{Minor-Embedding to King's Graph}\label{app:experiment_on_king_graph}

We conduct the first and second experiments in the main text setting the target hardware graph as a King's graph, which is adopted for the current commercial version of Hitachi CMOS Annealer~\cite{yamaoka201520k,sugie2021minor}.
Vertices in a King's graph are aligned in a rectangular shape and connected to other vertices that are in horizontally, vertically, or diagonally adjacent positions.
For a square King's graph $KG_{L,L}$ of size $L \times L$, a clique embedding of a complete graph of size $L+1$ is known~\cite{sugie2021minor}.
Since the Ising machine (AE) in our experiments accepts QUBO problems of a size up to 65536, we set the size of the King's graph to $256 \times 256 = 65536$ and use $n=257$ to generate synthetic QUBO instances.
The experimental setting and metrics are the same as in the main text except for the target graph and $n$.

Table~\ref{tab:king_random_num_order} shows the evaluation results of effects of the proposed method on minor-embedding.
In contrast to the cases of $P_{16}$ and $Z_{15}$ in the main text, the number of auxiliary variables and maximum chain length are not reduced even on $p=1.5$ and $p=2.0$.
This is probably because the minorminer algorithm is not suitable for searching for a good minor-embedding on the King's graph.
As we mentioned in Section~\ref{sec:experiments} in the main text, performance of minor-embedding search can have a significant impact on this experiment.
We hypothesize that improvement in the quality of minor-embedding is observed for $p=1.5$ and $p=2.0$ when we use a search algorithm suitable for the King's graph such as probabilistic-swap-shift-annealing (PSSA)~\cite{sugie2021minor}.
However, since additional experiments in Appendix~\ref{app:clique_linear} suggest that the performance improvement of Ising machine in our experiment might not be relevant to the quality of minor-embedding, we do not delve into further experiments using other search algorithms.

Table~\ref{tab:king_random_ising_machine} shows the performance of the Ising machine on the original and minor-embedded synthetic QUBO problems.
%In this experiment, w
We observed that several chains in solutions of embedded problems are broken, so we adopt the majority decision heuristic for fixing chains to compute energy for the embedded problems.
We observed similar trends as in the main text:
the energies for baselines are much higher than those for the original QUBO problems, and the linearization substantially mitigates the degradation.
Note that since the minor-embedding for the linearized problems is almost the same as that for the baselines (from Table~\ref{tab:king_random_num_order}), the mitigation apparently comes from factors other than the quality of minor-embedding.

Fig.~\ref{fig:embeddability_king} shows the results of the second experiment for the King's graph $KG_{256,256}$.
In contrast to the cases of $P_{16}$ and $Z_{15}$, we do not observe significant improvement in embeddability for $KG_{256,256}$, as no instances of size 260 other than only two with $p=2$ can be embedded to $KG_{256,256}$ even with linearization.
We again hypothesize that this is due to low performance of minor-embedding search and that using other algorithms such as PSSA leads to improving embeddability of linearized problems.

\begin{figure}[t]
     \centering
         \includegraphics[width=0.48\textwidth]{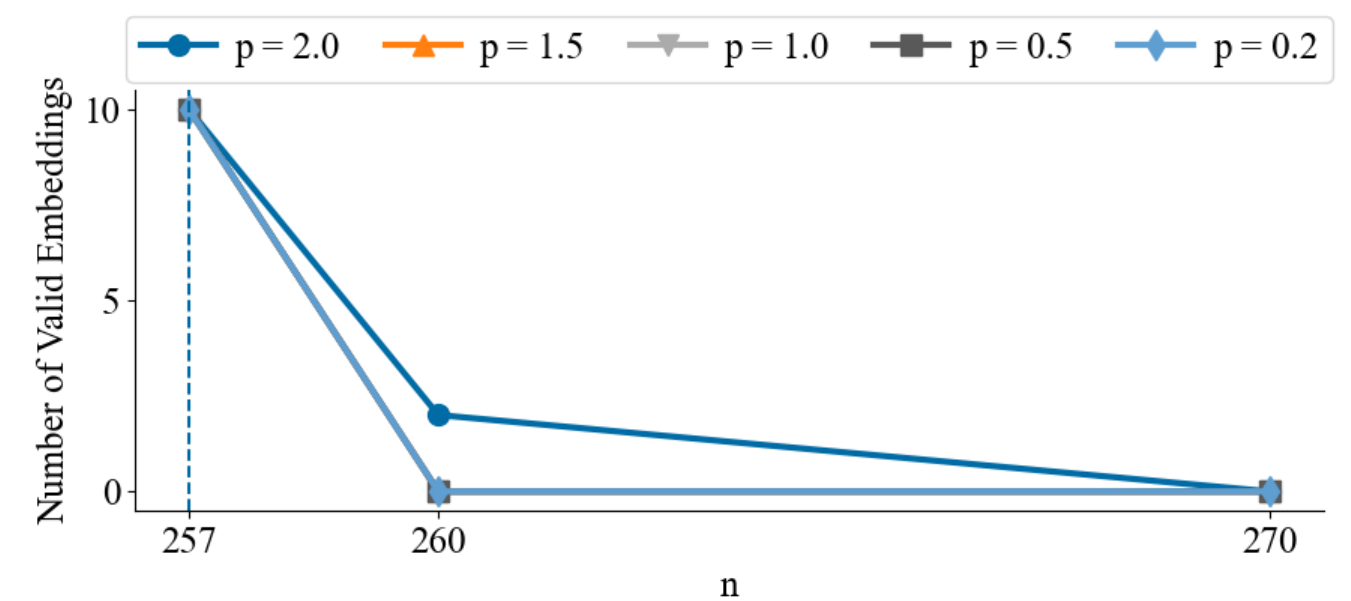}
         \caption{Success counts of finding minor-embedding of linearized QUBO problems to $K_{256,256}.$}
         \label{fig:embeddability_king}
\end{figure}

\begin{table*}[t]
  \caption{Energy of Solutions of Synthetic QUBO Problems for Fixed Embedding.}
  \label{tab:clique_linear_energy}
  \centering
  \begin{tabular}{cccrr@{\hspace{0pt}}r@{\hspace{3pt}}rr@{\hspace{3pt}}rcrr}
    \hline
    %\multirow{2}{*}{$
    %\begin{matrix}
    %    n \\
    %    H
    %\end{matrix}
    %$} 
    &&& \multicolumn{1}{c}{Without} && \multicolumn{4}{c}{With Embedding (Re-print)} %& \multicolumn{2}{c}{}
    \\
    &&& \multicolumn{1}{c}{Embedding} && \multicolumn{2}{c}{Clique Embedding} & \multicolumn{2}{c}{Minorminer} 
    && \multicolumn{2}{c}{Clique Embedding} \\
    \cline{6-9}
    \cline{11-12}
    $H$ 
    & $n$ 
    & $p$ & \multicolumn{1}{c}{(Re-print)} & & 
    %\multicolumn{2}{c}{Clique (Baseline)} &  
    \multicolumn{1}{c}{Baseline} & \multicolumn{1}{c}{Diff.} &
    %\multicolumn{2}{c}{Linearized} 
    \multicolumn{1}{c}{Linearized} & \multicolumn{1}{c}{Diff.} 
    && \multicolumn{1}{c}{Linearized} & \multicolumn{1}{c}{Diff.} 
    \\
    \hline \hline
    \multirow{2}{*}{$
    \begin{matrix}
    %    180 \\
        P_{16}    
    \end{matrix}
    $} &
    \multirow{2}{*}{180} 
    % & 0.2 &  $ -8225.7$ &  & $  -8119.7$ & ($+106.0$) & $  -8148.8$ & ($+76.9$) \\ % -8186.9 <- clique + linear
    %&& 0.5 &  $-30426.4$ &  & $ -30404.3$ &  ($+22.1$) & $ -30423.3$ &  ($+3.1$) \\ % -30425.7
    %&& 1.0 &  $-74842.3$ &  & $ -74838.3$ &   ($+4.0$) & $ -74842.3$ &  ($+0.0$) \\ % -74842.3
    & 1.5 & $-120531.9$ &  & $-120531.2$ &   ($+0.7$) & $-120531.9$ &  ($+0.0$)  && $-120531.9$ & ($+0.0$) \\
    && 2.0 & $-165226.1$ &  & $-165224.9$ &   ($+1.2$) & $-165226.1$ &  ($+0.0$) && $-165226.1$ & ($+0.0$) \\
    \hline
    \multirow{2}{*}{$
    \begin{matrix}
    %    232 \\
        Z_{15}
    \end{matrix}
    $} &
    \multirow{2}{*}{232}
    % & 0.2 & $ -13604.3$ &  & $ -13188.2$ &  ($+416.1$) & $ -13256.0$ & ($+348.3$) \\ % -13422.9
    %&& 0.5 & $ -50108.2$ &  & $ -49913.4$ &  ($+194.8$) & $ -50088.9$ &  ($+19.3$) \\ % -50101.0
    %&& 1.0 & $-125602.6$ &  & $-125505.1$ &   ($+97.5$) & $-125602.5$ &   ($+0.1$) \\ % -125602.6
    & 1.5 & $-196882.1$ &  & $-196872.1$ &   ($+10.0$) & $-196881.5$ &   ($+0.6$)  && $-196882.1$ & ($+0.0$) \\
    && 2.0 & $-272778.4$ &  & $-272763.6$ &   ($+14.8$) & $-272778.4$ &   ($+0.0$) && $-272778.4$ & ($+0.0$) \\
    \hline
    %\multirow{4}{*}{$KG_{256,256}$} &
    %\multirow{4}{*}{256}
    % & 0.2 &  -16991.0 &  &   -3636.6 &  (+13354.4) &   -9260.7 & (+7730.3) \\
    %&& 0.5 &  -63406.1 &  &  -52692.8 &  (+10713.3) &  -60251.9 & (+3154.2) \\
    %&& 1.0 & -151207.4 &  & -143336.2 &   (+7871.2) & -145167.8 & (+6039.6) \\
    %&& 1.5 & -243713.6 &  & -236933.5 &   (+6780.1) & -240049.8 & (+3663.8) \\ 
    %&& 2.0 & -336410.5 &  & -329740.5 &   (+6670.0) & -333827.0 & (+2583.5) \\ 
    %\hline
  \end{tabular}
\end{table*}

\subsection{Energy for Linearized Problems with Fixed Embedding}\label{app:clique_linear}

The experimental results (Tables~\ref{tab:random_num_order} and \ref{tab:random_ising_machine}) in the main text suggest that %there are other reasons for the improvement of Ising machine performance for the minor-embedded linearized problems than improving quality of minor-embedding.
the performance improvement of Ising machine on the minor-embedded linearized problems might not be relevant to the quality of minor-embedding.
To confirm this, we evaluate Ising machine performance on linearized problems embedded with a \emph{fixed} embedding.
If we see similar performance improvement against the baselines in this setting, it serves as evidence for the irrelevance of the minor-embedding quality to solving performance for the synthetic QUBO instances.
Similarly as in the main text, we sample 10 solutions of the linearized QUBO problems embedded with clique embedding for each combination of $p\in \{1.5, 2.0\}$ (for which the quality improvement of minor-embedding was observed) and $n\in \{180, 232\}$ and report the averaged energies.

Table~\ref{tab:clique_linear_energy} summarizes the results.
We re-print the results in the main text in the ``Without Embedding'' and ``With Embedding'' columns for comparison.
The energies for linearized problems with clique embedding achieve the same values as those of the original problems without embedding.
The performance improvement is similar to that for linearized problems with minor-embedding obtained by the minorminer algorithm.
Therefore, in our experimental settings, it seems that the performance improvement of the Ising machine does not come from the difference in mapping of minor-embedding.

\begin{figure*}[t]
     \centering
         \includegraphics[width=0.92\textwidth]{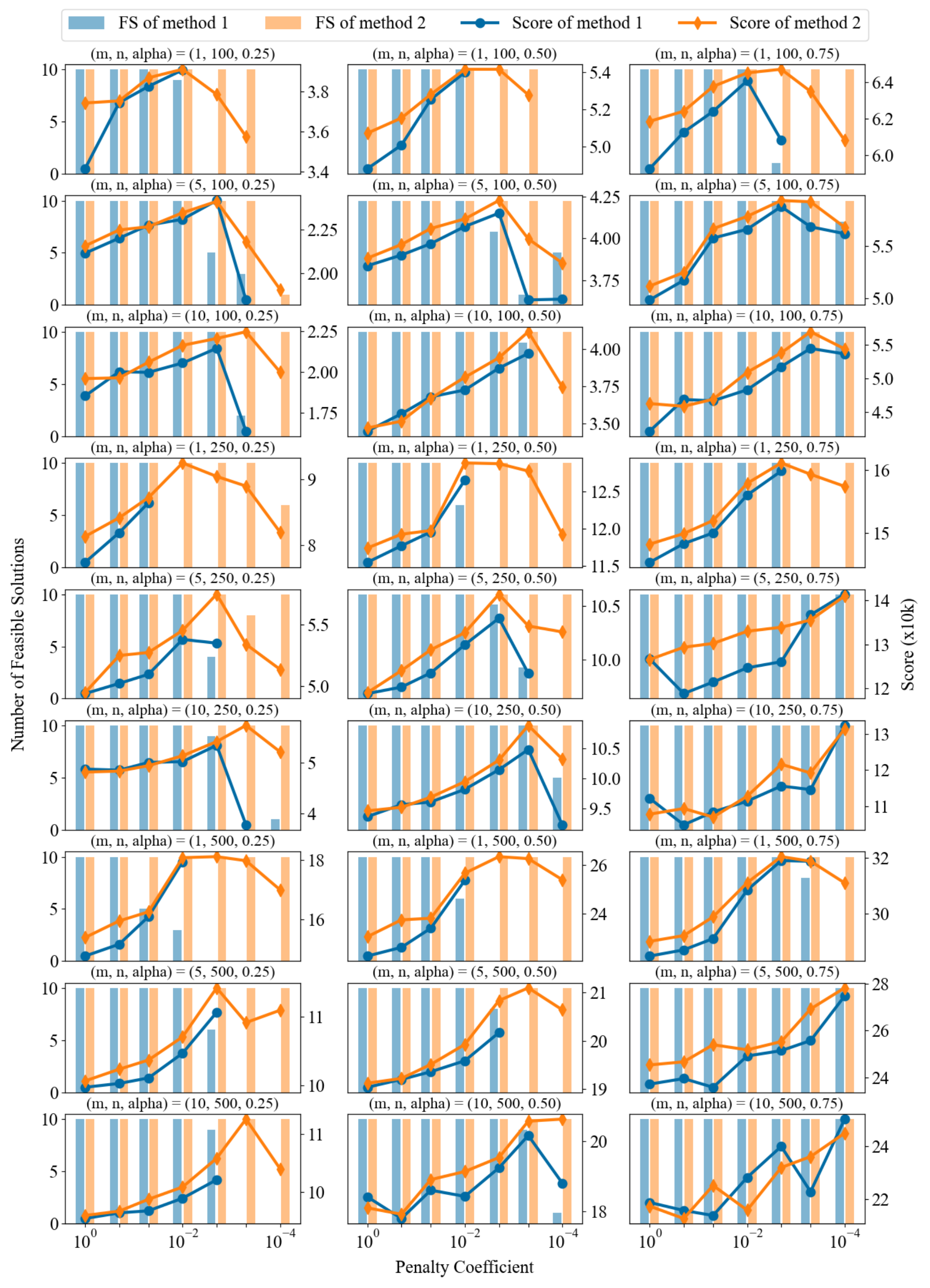}
         \caption{Results on tuning penalty coefficient $\lambda$ on MKPs. Bar charts represent numbers of feasible solutions and lines represent scores. As $\lambda$ becomes smaller, numbers of feasible solutions tend to decrease and scores tend to increase (as long as sufficient number of feasible solutions are obtained).}
         \label{fig:tuning_on_mkp}
\end{figure*}

\begin{table*}[t]
\begin{minipage}[t]{\textwidth}
    \caption{Optimality Gap on MKPs for Baseline and Proposed Method Based on Method 1 with Tuning on $\lambda$.}
  \label{tab:mkp_tuned_best_value_case1}
  \centering
  \begin{tabular}{c@{\hspace{3pt}}c@{\hspace{4pt}}clrrrrrrrr@{\hspace{2pt}}rrr@{\hspace{2pt}}r}
    \hline
    %& & & & \multicolumn{12}{c}{Method 1}
    %\\ 
    %\cline{5-16}
    & Instance & & & \multicolumn{3}{c}{Baseline} & & \multicolumn{8}{c}{Linearized}
    \\
    \cline{1-3}
    \cline{5-7}
    \cline{9-16}
    $m$ & $n$ & $\alpha$ 
    & \multicolumn{1}{c}{$\lambda$} & \multicolumn{1}{c}{\#FS} & Avg. & Best 
    & & \multicolumn{1}{c}{$|E|$} & \multicolumn{1}{c}{\#FS} & Avg. & \multicolumn{2}{c}{\hl{RR (\%)}} & Best  & \multicolumn{2}{c}{\hl{RR (\%)}} 
    \\ 
    \hline \hline
 1 & 100 & 0.25 & 0.01   & 42.2 &  2.037&  0.031 &&  2013.6 &45.9 &  0.982 & \hl{$-50.8 \pm $}&\hl{$ 39.9$ } & 0.013 & \hl{$ -100.0 \pm $}&\hl{$ 0.0 $} \\
   &     & 0.50 & 0.01   & 45.5 &  1.393&  0.117 &&  2033.0 &42.9 &  4.580 & \hl{$+275.1 \pm $}&\hl{$ 453.1$ } & 0.045 & \hl{$ -67.1 \pm $}&\hl{$ 50.1 $} \\
   &     & 0.75 & 0.01   & 48.6 &  3.900&  0.410 &&  1973.0 &45.4 &  7.620 & \hl{$+ 85.7 \pm $}&\hl{$ 112.4$ } & 0.047 & \hl{$ -87.6 \pm $}&\hl{$ 15.6 $} \\
 1 & 250 & 0.25 & 0.05   & 47.3 & 10.739&  7.370 && 12529.9 &48.5 &  1.310 & \hl{$-87.6 \pm $}&\hl{$ 6.4$ } & 0.015 & \hl{$ -99.8 \pm $}&\hl{$ 0.3 $} \\
   &     & 0.50 & 0.01   & 31.9 &  4.474&  1.697 && 12731.2 &19.4 &  3.848 & \hl{$ -6.4 \pm $}&\hl{$ 135.4$ } & 0.060 & \hl{$ -96.0 \pm $}&\hl{$ 5.3 $} \\
   &     & 0.75 & 0.002  & 47.3 &  2.154&  0.833 && 12552.8 &36.3 &  1.564 & \hl{$-27.0 \pm $}&\hl{$ 7.2$ } & 0.255 & \hl{$ -59.3 \pm $}&\hl{$ 38.1 $} \\
 1 & 500 & 0.25 & 0.01   & 7.1 &  4.824&  2.663 && 51922.9 &3.7 &  0.880 & \hl{$-81.2 \pm $}&\hl{$ 11.2$ } & 0.321 & \hl{$ -87.6 \pm $}&\hl{$ 6.6 $} \\
   &     & 0.50 & 0.01   & 30.2 &  5.970&  3.740 && 51008.0 &28.4 &  0.784 & \hl{$-86.5 \pm $}&\hl{$ 9.7$ } & 0.064 & \hl{$ -98.3 \pm $}&\hl{$ 0.5 $} \\
   &     & 0.75 & 0.002  & 49.8 &  2.112&  0.599 && 50326.0 &49.9 &  1.393 & \hl{$-33.8 \pm $}&\hl{$ 7.8$ } & 0.159 & \hl{$ -73.5 \pm $}&\hl{$ 11.4 $} \\
 5 & 100 & 0.25 & 0.002  & 29.1 & 10.722&  1.570 &&    23.8 &30.2 & 11.502 & \hl{$+  7.4 \pm $}&\hl{$ 15.8$ } & 1.436 & \hl{$ +  8.0 \pm $}&\hl{$ 74.6 $} \\
   &     & 0.50 & 0.002  & 39.4 & 12.351&  1.697 &&    29.5 &37.9 & 11.784 & \hl{$ -2.8 \pm $}&\hl{$ 14.8$ } & 1.728 & \hl{$ + 17.0 \pm $}&\hl{$ 71.8 $} \\
   &     & 0.75 & 0.002  & 50.0 &  7.378&  1.377 &&    26.7 &50.0 &  7.838 & \hl{$+  7.0 \pm $}&\hl{$ 29.0$ } & 1.180 & \hl{$  -9.2 \pm $}&\hl{$ 29.7 $} \\
 5 & 250 & 0.25 & 0.01   & 49.7 & 11.871&  8.954 &&   156.7 &49.8 &  9.531 & \hl{$-19.7 \pm $}&\hl{$ 3.7$ } & 7.291 & \hl{$ -18.4 \pm $}&\hl{$ 6.5 $} \\
   &     & 0.50 & 0.002  & 47.0 & 13.450&  3.730 &&   143.4 &45.4 & 10.531 & \hl{$-21.1 \pm $}&\hl{$ 14.7$ } & 2.987 & \hl{$ -19.4 \pm $}&\hl{$ 11.6 $} \\
   &     & 0.75 & 0.0001 & 43.9 &  7.616&  4.496 &&   148.4 &44.6 &  7.812 & \hl{$+  3.2 \pm $}&\hl{$ 10.3$ } & 4.419 & \hl{$  -0.2 \pm $}&\hl{$ 18.0 $} \\
 5 & 500 & 0.25 & 0.002  & 9.7 & 13.070&  7.315 &&   665.7 &11.5 & 11.017 & \hl{$-14.8 \pm $}&\hl{$ 16.0$ } & 5.594 & \hl{$ -22.0 \pm $}&\hl{$ 12.3 $} \\
   &     & 0.50 & 0.002  & 45.5 & 11.146&  6.500 &&   610.0 &48.8 & 10.370 & \hl{$ -6.6 \pm $}&\hl{$ 16.3$ } & 4.916 & \hl{$ -24.1 \pm $}&\hl{$ 6.2 $} \\
   &     & 0.75 & 0.0001 & 50.0 & 10.645&  5.129 &&   643.3 &50.0 & 10.458 & \hl{$ -1.7 \pm $}&\hl{$ 3.6$ } & 5.114 & \hl{$ +  1.8 \pm $}&\hl{$ 29.6 $} \\
10 & 100 & 0.25 & 0.002  & 48.2 & 19.109&  5.798 &&     0.8 &47.8 & 19.975 & \hl{$+  5.0 \pm $}&\hl{$ 8.4$ } & 5.377 & \hl{$  -6.2 \pm $}&\hl{$ 17.5 $} \\
   &     & 0.50 & 0.0005 & 39.1 & 23.905&  9.667 &&     1.2 &39.7 & 23.010 & \hl{$ -2.9 \pm $}&\hl{$ 11.6$ } & 5.644 & \hl{$ -39.0 \pm $}&\hl{$ 16.7 $} \\
   &     & 0.75 & 0.0005 & 50.0 & 14.350&  2.402 &&     0.6 &50.0 & 15.067 & \hl{$+  5.4 \pm $}&\hl{$ 8.3$ } & 2.833 & \hl{$ + 47.1 \pm $}&\hl{$ 83.3 $} \\
10 & 250 & 0.25 & 0.002  & 43.4 & 14.436&  9.411 &&     5.2 &44.2 & 14.667 & \hl{$+  1.8 \pm $}&\hl{$ 6.8$ } & 9.640 & \hl{$ +  2.5 \pm $}&\hl{$ 5.5 $} \\
   &     & 0.50 & 0.0005 & 45.7 & 19.591&  5.936 &&     4.0 &46.2 & 18.312 & \hl{$ -6.6 \pm $}&\hl{$ 8.1$ } & 4.296 & \hl{$ -20.9 \pm $}&\hl{$ 37.4 $} \\
   &     & 0.75 & 0.0001 & 50.0 & 18.940& 11.970 &&     2.9 &50.0 & 19.008 & \hl{$+  0.4 \pm $}&\hl{$ 4.2$ } & 9.114 & \hl{$ -23.2 \pm $}&\hl{$ 20.1 $} \\
10 & 500 & 0.25 & 0.002  & 44.7 & 16.460& 12.200 &&    13.1 &44.9 & 16.324 & \hl{$ -0.7 \pm $}&\hl{$ 3.4$ } &12.593 & \hl{$ +  3.3 \pm $}&\hl{$ 5.4 $} \\
   &     & 0.50 & 0.0005 & 47.6 & 12.531&  6.261 &&    15.0 &47.9 & 13.194 & \hl{$+  7.2 \pm $}&\hl{$ 22.7$ } & 5.983 & \hl{$  -3.2 \pm $}&\hl{$ 15.8 $} \\
   &     & 0.75 & 0.0001 & 50.0 & 22.006& 15.918 &&    14.5 &50.0 & 22.235 & \hl{$+  1.1 \pm $}&\hl{$ 2.9$ } &14.974 & \hl{$  -5.8 \pm $}&\hl{$ 10.1 $} \\
    \hline
  \end{tabular}
%\end{table*}
\end{minipage}
\bigskip
\begin{minipage}[t]{\textwidth}
%\begin{table*}[t]

  \caption{Optimality Gap on MKPs for Baseline and Proposed Method Based on Method 2 with Tuning on $\lambda$.}
  \label{tab:mkp_tuned_best_value_case2}
  \centering
  \begin{tabular}{c@{\hspace{3pt}}c@{\hspace{4pt}}clrrrrrrrr@{\hspace{2pt}}rrr@{\hspace{2pt}}r}
    \hline
    %& & & & \multicolumn{12}{c}{Method 2}
    %\\ 
    %\cline{5-16}
    & Instance & & & \multicolumn{3}{c}{Baseline} & & \multicolumn{8}{c}{Linearized}
    \\
    \cline{1-3}
    \cline{5-7}
    \cline{9-16}
    $m$ & $n$ & $\alpha$ 
    & \multicolumn{1}{c}{$\lambda$} & \multicolumn{1}{c}{\#FS} & Avg. & Best 
    & & \multicolumn{1}{c}{$|E|$} & \multicolumn{1}{c}{\#FS} & Avg. & \multicolumn{2}{c}{\hl{RR (\%)}} & Best  & \multicolumn{2}{c}{\hl{RR (\%)}}
    \\ 
    \hline \hline
 1 & 100 & 0.25 & 0.01   & 50.0 &  0.054&  0.000 &&  2013.6 &50.0 &  0.000 & -& & 0.000 & -& \\
   &     & 0.50 & 0.01   & 50.0 &  0.279&  0.006 &&  2033.0 &50.0 &  0.000 & \hl{$-100.0 \pm $}&\hl{$ 0.0$ } & 0.000 & \hl{$ -100.0 \pm $}&\hl{$ 0.0 $} \\
   &     & 0.75 & 0.002  & 50.0 &  0.231&  0.012 &&  1973.0 &50.0 &  0.009 & \hl{$-89.9 \pm $}&\hl{$ 19.6$ } & 0.000 & \hl{$ -100.0 \pm $}&\hl{$ 0.0 $} \\
 1 & 250 & 0.25 & 0.01   & 50.0 &  0.872&  0.246 && 12529.9 &50.0 &  0.015 & \hl{$-98.1 \pm $}&\hl{$ 1.0$ } & 0.000 & \hl{$ -100.0 \pm $}&\hl{$ 0.0 $} \\
   &     & 0.50 & 0.01   & 50.0 &  1.920&  1.167 && 12731.2 &50.0 &  0.031 & \hl{$-98.4 \pm $}&\hl{$ 0.6$ } & 0.003 & \hl{$ -99.7 \pm $}&\hl{$ 0.4 $} \\
   &     & 0.75 & 0.002  & 50.0 &  0.846&  0.104 && 12552.8 &50.0 &  0.289 & \hl{$-64.9 \pm $}&\hl{$ 21.2$ } & 0.002 & \hl{$ -97.4 \pm $}&\hl{$ 6.4 $} \\
 1 & 500 & 0.25 & 0.002  & 50.0 &  1.526&  1.418 && 51922.9 &48.7 &  2.501 & \hl{$+ 64.0 \pm $}&\hl{$ 9.3$ } & 1.045 & \hl{$ -26.3 \pm $}&\hl{$ 15.2 $} \\
   &     & 0.50 & 0.002  & 50.0 &  1.122&  0.936 && 51008.0 &50.0 &  0.735 & \hl{$-34.4 \pm $}&\hl{$ 8.7$ } & 0.080 & \hl{$ -91.4 \pm $}&\hl{$ 5.1 $} \\
   &     & 0.75 & 0.002  & 50.0 &  1.441&  0.223 && 50326.0 &50.0 &  2.392 & \hl{$+ 70.0 \pm $}&\hl{$ 52.0$ } & 0.135 & \hl{$ -18.3 \pm $}&\hl{$ 65.9 $} \\
 5 & 100 & 0.25 & 0.002  & 50.0 &  2.377&  0.595 &&    23.8 &50.0 &  2.273 & \hl{$ -4.3 \pm $}&\hl{$ 4.9$ } & 0.539 & \hl{$  -3.1 \pm $}&\hl{$ 40.9 $} \\
   &     & 0.50 & 0.002  & 50.0 &  1.861&  0.885 &&    29.5 &50.0 &  1.719 & \hl{$ -7.6 \pm $}&\hl{$ 7.2$ } & 0.662 & \hl{$ -22.0 \pm $}&\hl{$ 30.8 $} \\
   &     & 0.75 & 0.002  & 50.0 &  1.817&  0.659 &&    26.7 &50.0 &  1.369 & \hl{$-23.8 \pm $}&\hl{$ 9.4$ } & 0.454 & \hl{$ -28.1 \pm $}&\hl{$ 28.0 $} \\
 5 & 250 & 0.25 & 0.002  & 50.0 &  4.305&  2.859 &&   156.7 &50.0 &  3.868 & \hl{$-10.0 \pm $}&\hl{$ 4.2$ } & 2.401 & \hl{$ -15.6 \pm $}&\hl{$ 12.7 $} \\
   &     & 0.50 & 0.002  & 50.0 &  3.534&  2.603 &&   143.4 &50.0 &  2.879 & \hl{$-18.4 \pm $}&\hl{$ 4.2$ } & 2.013 & \hl{$ -22.5 \pm $}&\hl{$ 5.6 $} \\
   &     & 0.75 & 0.0001 & 50.0 &  6.042&  3.586 &&   148.4 &50.0 &  5.890 & \hl{$ -2.3 \pm $}&\hl{$ 5.4$ } & 3.586 & \hl{$ +  0.7 \pm $}&\hl{$ 11.9 $} \\
 5 & 500 & 0.25 & 0.002  & 50.0 &  6.162&  4.802 &&   665.7 &50.0 &  4.929 & \hl{$-20.0 \pm $}&\hl{$ 3.1$ } & 3.938 & \hl{$ -17.9 \pm $}&\hl{$ 5.8 $} \\
   &     & 0.50 & 0.0005 & 50.0 &  4.289&  2.411 &&   610.0 &50.0 &  3.780 & \hl{$-11.8 \pm $}&\hl{$ 4.7$ } & 2.328 & \hl{$  -3.0 \pm $}&\hl{$ 16.3 $} \\
   &     & 0.75 & 0.0001 & 50.0 &  8.935&  4.039 &&   643.3 &50.0 & 11.838 & \hl{$+ 33.0 \pm $}&\hl{$ 10.7$ } & 5.633 & \hl{$ + 45.7 \pm $}&\hl{$ 41.1 $} \\
10 & 100 & 0.25 & 0.0005 & 50.0 &  8.438&  2.418 &&     0.8 &50.0 &  8.198 & \hl{$ -2.9 \pm $}&\hl{$ 4.3$ } & 1.795 & \hl{$ + 26.0 \pm $}&\hl{$ 140.2 $} \\
   &     & 0.50 & 0.0005 & 50.0 &  3.955&  1.176 &&     1.2 &50.0 &  3.846 & \hl{$ -2.8 \pm $}&\hl{$ 6.8$ } & 1.314 & \hl{$ + 34.4 \pm $}&\hl{$ 80.7 $} \\
   &     & 0.75 & 0.0005 & 50.0 &  2.254&  0.643 &&     0.6 &50.0 &  2.241 & \hl{$ -1.0 \pm $}&\hl{$ 9.6$ } & 0.651 & \hl{$ +  3.0 \pm $}&\hl{$ 17.9 $} \\
10 & 250 & 0.25 & 0.0005 & 50.0 &  7.805&  2.794 &&     5.2 &50.0 &  7.813 & \hl{$+  0.3 \pm $}&\hl{$ 6.3$ } & 3.100 & \hl{$ + 21.9 \pm $}&\hl{$ 51.8 $} \\
   &     & 0.50 & 0.0005 & 50.0 &  3.792&  1.992 &&     4.0 &50.0 &  3.827 & \hl{$+  1.1 \pm $}&\hl{$ 4.6$ } & 1.925 & \hl{$  -2.9 \pm $}&\hl{$ 11.3 $} \\
   &     & 0.75 & 0.0001 & 50.0 & 16.110&  6.666 &&     2.9 &50.0 & 15.663 & \hl{$ -2.8 \pm $}&\hl{$ 3.1$ } & 4.835 & \hl{$  -5.8 \pm $}&\hl{$ 74.9 $} \\
10 & 500 & 0.25 & 0.0005 & 50.0 &  7.504&  3.628 &&    14.2 &50.0 &  7.104 & \hl{$ -5.1 \pm $}&\hl{$ 4.4$ } & 3.783 & \hl{$ +  5.6 \pm $}&\hl{$ 18.5 $} \\
   &     & 0.50 & 0.0001 & 50.0 &  6.453&  4.293 &&    15.0 &50.0 &  5.632 & \hl{$-12.7 \pm $}&\hl{$ 3.7$ } & 4.053 & \hl{$  -4.5 \pm $}&\hl{$ 14.8 $} \\
   &     & 0.75 & 0.0001 & 50.0 & 21.332& 14.636 &&    14.5 &50.0 & 21.727 & \hl{$+  1.9 \pm $}&\hl{$ 1.7$ } &15.785 & \hl{$ +  9.7 \pm $}&\hl{$ 18.0 $} \\
    \hline
  \end{tabular}
\end{minipage}
\end{table*}

\subsection{Experiments on MKPs with Tuning Penalty Coefficient}\label{app:experiment_tuning}

%For fairer comparison, 
We evaluate the proposed method on MKPs against stronger baselines on the basis of tuning penalty coefficients. 

When applying an Ising machine to constrained optimization problems, feasible solutions are not necessarily obtained.
It is favorable to obtaining feasible solutions with sufficiently high probability while achieving high optimization scores.
Performance of the Ising machines in this aspect mostly depends on the value of penalty coefficients for the constraints~\cite{verma2022penalty,yarkoni2022quantum}.
On MKPs, $\lambda$ in Eq.~(\ref{eq:mkp_qubo}) should be taken sufficiently large for constraints on solutions to be satisfied.
On the other hand, large $\lambda$ tends to degrade the objective score.
To tune baselines with respect to this aspect, we take one instance from each combination of $(m,n, \alpha)$ and apply the Ising machine to it 10 times for each $\lambda \in \{1, 0.2, 0.05, 0.01, 0.002, 0.0005, 0.0001\}$.
We report the number of feasibility solutions and best optimization scores $\sum_i v_i x_i$ over feasible solutions.

The tuning results are shown in Fig.~\ref{fig:tuning_on_mkp}.
As expected, small $\lambda$ leads to a decrease in the number of feasible solutions and higher optimization scores. % (as long as sufficient number of feasible solutions are obtained).
When $\lambda$ is too small to obtain a sufficient number of feasible solutions, variance of scores increases, so the best scores tend to decrease.
For Method~2, there are several cases where the scores decrease while the numbers of solutions do not.
This is probably due to the function specification of \texttt{less\_than()}. %, though it is undisclosed.

We conduct the third experiment in the main text again with setting $\lambda$ to the best parameter achieving the highest best score in the tuning experiment.
Note that the proposed method is evaluated with the same value of $\lambda$ as that of the tuned baselines.
The results are shown in Tables~\ref{tab:mkp_tuned_best_value_case1} and \ref{tab:mkp_tuned_best_value_case2}.
We observe that the proposed method improves the average and best gap in most cases where $|E|$ is larger than a hundred, even on the strong tuned baselines.
One exception is the case of $m=5$, $n=500$, and $\alpha=0.75$ for Method~2, where the solutions with linearization have a more than 1.3x average or best gap than the baseline.
The cause of the phenomenon is difficult to analyze as the specifications of \texttt{less\_than()} function are undisclosed.
Overall, the proposed method generally improves performance of the Ising machine on MKPs even on the well-tuned setting.

\section{\hl{Experiments on Max-Cut/QUBO dataset}}\label{app:experiments_mqlib}

\begin{table*}[t]
      \caption{\hl{Summary on MQLib~\mbox{\cite{dunning2018works}} Instances.}}
  \label{tab:mqlib}
  \centering
  \begin{tabular}{crrrrr}
    \hline
    \hlrow Dataset   & \# of instances & \# of nodes & Density (\%) & Valid order found & Before preprocessing \\
    \hline \hline
    \hlrow Gset      & 17   & 5k-20k & 0.0-0.2 & 0 & (3)\\
    \hlrow Beasley   & 60   & 0k-3k  & 9.9-14.7 & 0 & (0)\\
    \hlrow Culberson & 108  & 1k-5k  & 0.1-98.7 & 0 & (0)\\
    \hlrow Imgseg    & 100  & 1k-28k & 0.0-0.2 & 0 & (0) \\
    \hlrow GKA       & 45   & 21-501 &  8.0-99.0 & 0 & (0) \\
    \hlrow p3-7      & 21   & 3k-7k & 49.8-99.5 & 0 & (0) \\
%    Others    & 3111 & 0k-38k & 1-6965 & 24 & (106)\\
    \hlrow Others    & 3269 & 0k-38k & 0.0-100 & 24 & (106)\\
    \hline
  \end{tabular}
\end{table*}

\begin{table*}[t]
      \caption{\hl{Results of Solutions of Max-Cut on MQLib Instances.}}
  \label{tab:mqlib_energy_without_embedding}
  \centering
  \begin{tabular}{crrrrrrr}
    \hline
 \hlrow Instance & Domain & \# of nodes & \# of edges & $|E|$ & Cut (Baseline)  & Cut \\
\hline
    \hlrow g000064 & PCST & 3624 & 28964 & 499 &  144608.22 $\pm$ 185.46 & 144519.57 $\pm$ 217.71 \\
    \hlrow g000222 & Planer & 5347 & 7634 & 7 &  121187.60 $\pm$ 69.73 & 121178.10 $\pm$ 63.31 \\
    \hlrow g000304 & Isom. & 34 & 153 & 52 &  98.00 $\pm$ 0.00 & 98.00 $\pm$ 0.00 \\
    \hlrow g000496 & Coloring & 138 & 3925 & 55 &  2540.00 $\pm$ 0.00 & 2540.00 $\pm$ 0.00 \\
    \hlrow g000577 & Coloring & 363 & 8688 & 76 &  7215.00 $\pm$ 0.00 & 7215.00 $\pm$ 0.00 \\
    \hlrow g000599 & Planer & 3 & 3 & 3 &  372.00 $\pm$ 0.00 & 372.00 $\pm$ 0.00 \\
    \hlrow g000655 & Isom. & 34 & 153 & 52 &  98.00 $\pm$ 0.00 & 98.00 $\pm$ 0.00 \\
    \hlrow g000742 & Coloring & 173 & 3885 & 52 &  3006.00 $\pm$ 0.00 & 3006.00 $\pm$ 0.00 \\
    \hlrow g000812 & Coloring & 558 & 13979 & 142 &  12082.00 $\pm$ 0.00 & 12082.00 $\pm$ 0.00 \\
    \hlrow g001244 & PCST & 5226 & 93394 & 3724 &  212334.94 $\pm$ 76.14 & 212409.04 $\pm$ 107.44 \\
    \hlrow g001328 & Coloring & 363 & 8691 & 77 &  7215.00 $\pm$ 0.00 & 7215.00 $\pm$ 0.00 \\
    \hlrow g001371 & Isom. & 34 & 153 & 52 &  98.00 $\pm$ 0.00 & 98.00 $\pm$ 0.00 \\
    \hlrow g001634 & Isom. & 34 & 153 & 52 &  98.00 $\pm$ 0.00 & 98.00 $\pm$ 0.00 \\
    \hlrow g001860 & Isom. & 77 & 231 & 71 &  154.00 $\pm$ 0.00 & 154.00 $\pm$ 0.00 \\
    \hlrow g001933 & Coloring & 559 & 13969 & 135 &  12108.00 $\pm$ 0.00 & 12108.00 $\pm$ 0.00 \\
    \hlrow g002031 & Isom. & 34 & 153 & 50 &  97.00 $\pm$ 0.00 & 97.00 $\pm$ 0.00 \\
    \hlrow g002294 & Coloring & 126 & 4100 & 91 &  2779.00 $\pm$ 0.00 & 2779.00 $\pm$ 0.00 \\
    \hlrow g002510 & Coloring & 269 & 11654 & 142 &  8689.00 $\pm$ 0.00 & 8689.00 $\pm$ 0.00 \\
    \hlrow g002739 & Coloring & 519 & 18707 & 287 &  15316.00 $\pm$ 0.00 & 15316.00 $\pm$ 0.00 \\
    \hlrow g002756 & Isom. & 77 & 231 & 70 &  154.00 $\pm$ 0.00 & 154.00 $\pm$ 0.00 \\
    \hlrow g002783 & Isom. & 34 & 153 & 50 &  97.00 $\pm$ 0.00 & 97.00 $\pm$ 0.00 \\
    \hlrow g003074 & Coloring & 175 & 3946 & 50 &  3060.00 $\pm$ 0.00 & 3060.00 $\pm$ 0.00 \\
    \hlrow g003084 & PCST & 5226 & 93394 & 3724 &  234235.59 $\pm$ 101.66 & 234246.89 $\pm$ 113.92 \\
    \hlrow g003107 & PCST & 5226 & 93394 & 3724 &  321020.44 $\pm$ 124.71 & 321048.44 $\pm$ 153.99 \\
    \hline
  \end{tabular}
\end{table*}

\begin{table*}[t]
      \caption{\hl{Results of Minor-Embedding on MQLib.}}
  \label{tab:embedding_mqlib}
  \centering
  \begin{tabular}{@{\hspace{2pt}}c@{\hspace{6pt}} r r@{\hspace{5pt}}r@{\hspace{7pt}}r@{\hspace{5pt}} r r@{\hspace{5pt}}r@{\hspace{7pt}}r@{\hspace{5pt}} r r@{\hspace{5pt}}r@{\hspace{7pt}}r@{\hspace{5pt}} r r@{\hspace{5pt}}r@{\hspace{7pt}}r@{\hspace{2pt}}}
    \hline
    \hlrow && \multicolumn{7}{c}{$Z_{15}$} && \multicolumn{7}{c}{$P_{16}$}  \\
    \cline{3-9}
    \cline{11-17}    
    \hlrow && \multicolumn{3}{c}{Baseline} && \multicolumn{3}{c}{Linearized} && \multicolumn{3}{c}{Baseline} && \multicolumn{3}{c}{Linearized} \\
    \cline{3-5}
    \cline{7-9}
    \cline{11-13}
    \cline{15-17}
 \hlrow &&  Aux. & Chain &&& Aux. & Chain &&& Aux. & Chain &&& Aux. & Chain &\\
 \hlrow Instance &&  Var. & Length & Cut && Var. & Length & Cut && Var. & Length & Cut && Var. & Length & Cut \\
\hline
   \hlrow g000304 && 25 & 2 & 98.0 $\pm$ 0.0 && 9 & 2 & 98.0 $\pm$ 0.0  && 29 & 3 & 98.0 $\pm$ 0.0 && 15 & 3 & 98.0 $\pm$ 0.0 \\
   \hlrow g000496 && 889 & 18 & 2460.0 $\pm$ 18.9 && 868 & 17 & 2461.6 $\pm$ 20.9  && 987 & 22 & 2476.8 $\pm$ 23.9 && 963 & 17 & 2467.5 $\pm$ 24.1 \\
   \hlrow g000577 && 2794 & 49 & 6344.1 $\pm$ 180.2 && 2548 & 43 & 6723.9 $\pm$ 138.5  && 2958 & 60 & 6777.5 $\pm$ 49.4 && 3003 & 60 & 6794.8 $\pm$ 83.4 \\
   \hlrow g000655 && 24 & 2 & 98.0 $\pm$ 0.0 && 10 & 2 & 98.0 $\pm$ 0.0  && 27 & 2 & 98.0 $\pm$ 0.0 && 27 & 2 & 98.0 $\pm$ 0.0 \\
   \hlrow g000742 && 965 & 20 & 2965.7 $\pm$ 16.0 && 911 & 19 & 2948.4 $\pm$ 23.8  && 1113 & 22 & 3000.7 $\pm$ 5.2 && 1055 & 20 & 3001.6 $\pm$ 4.3 \\
   \hlrow g000812 && 4320 & 83 & 8622.9 $\pm$ 426.9 && 4120 & 77 & 8548.4 $\pm$ 659.6 && - & - & - && - & - & -\\
   \hlrow g001328 && 2786 & 57 & 5528.8 $\pm$ 693.1 && 2647 & 55 & 6240.5 $\pm$ 335.0  && 3305 & 67 & 6494.6 $\pm$ 142.7 && 3305 & 67 & 6525.1 $\pm$ 112.6 \\
   \hlrow g001371 && 23 & 2 & 98.0 $\pm$ 0.0 && 10 & 2 & 98.0 $\pm$ 0.0  && 28 & 3 & 98.0 $\pm$ 0.0 && 13 & 3 & 98.0 $\pm$ 0.0 \\
   \hlrow g001634 && 23 & 2 & 98.0 $\pm$ 0.0 && 8 & 2 & 98.0 $\pm$ 0.0  && 27 & 3 & 98.0 $\pm$ 0.0 && 15 & 3 & 98.0 $\pm$ 0.0 \\
   \hlrow g001860 && 32 & 2 & 153.6 $\pm$ 0.8 && 3 & 2 & 153.6 $\pm$ 0.8  && 28 & 3 & 154.0 $\pm$ 0.0 && 7 & 2 & 154.0 $\pm$ 0.0 \\
   \hlrow g001933 && 5122 & 105 & 7888.8 $\pm$ 269.0 && 5424 & 105 & 7771.3 $\pm$ 266.0 && - & - & - && - & - & -\\
   \hlrow g002031 && 26 & 2 & 97.0 $\pm$ 0.0 && 11 & 2 & 97.0 $\pm$ 0.0  && 28 & 3 & 97.0 $\pm$ 0.0 && 15 & 3 & 97.0 $\pm$ 0.0 \\
   \hlrow g002294 && 933 & 16 & 2695.5 $\pm$ 42.4 && 933 & 16 & 2717.8 $\pm$ 24.5  && 983 & 16 & 2702.3 $\pm$ 36.5 && 983 & 16 & 2692.0 $\pm$ 51.4 \\
   \hlrow g002510 && 3509 & 41 & 6701.8 $\pm$ 419.1 && 2831 & 34 & 7711.6 $\pm$ 314.0  && 3499 & 38 & 7954.0 $\pm$ 107.4 && 3499 & 38 & 8099.3 $\pm$ 81.7 \\
   \hlrow g002756 && 31 & 2 & 154.0 $\pm$ 0.0 && 7 & 2 & 154.0 $\pm$ 0.0  && 27 & 2 & 153.8 $\pm$ 0.6 && 4 & 2 & 153.8 $\pm$ 0.6 \\
   \hlrow g002783 && 22 & 2 & 97.0 $\pm$ 0.0 && 10 & 2 & 97.0 $\pm$ 0.0  && 29 & 3 & 97.0 $\pm$ 0.0 && 22 & 3 & 97.0 $\pm$ 0.0 \\
   \hlrow g003074 && 979 & 24 & 2989.0 $\pm$ 32.9 && 979 & 24 & 2993.8 $\pm$ 24.6  && 1429 & 29 & 3022.6 $\pm$ 16.1 && 1271 & 27 & 3037.4 $\pm$ 14.4 \\
   \hline
  \end{tabular}
\end{table*}

\begin{comment}
\begin{table*}[t]
  \caption{Energy of Solutions of MQLib Instances}
  \label{tab:mqlib_embed}
  \centering
  \begin{tabular}{cccrrrr@{\hspace{3pt}}rr@{\hspace{3pt}}r}
    \hline
    &&& \multicolumn{2}{c}{Without Embedding} && \multicolumn{4}{c}{With Embedding}\\
    \cline{4-5}
    \cline{7-10}
    Instance & $H$ & & Baseline & Linearized & 
    %\multicolumn{2}{c}{Clique} &  \multicolumn{2}{c}{Linearized} 
    \multicolumn{1}{c}{Baseline} & \multicolumn{1}{c}{Diff.} &
    \multicolumn{1}{c}{Linearized} & \multicolumn{1}{c}{Diff.} \\
    \hline \hline
    
    \hline
  \end{tabular}
\end{table*}
\end{comment}

\hl{
In addition to the synthetic QUBO instances tested in the main text, we conduct experiments on MQLib~\mbox{\cite{dunning2018works}}, a large collection of the max-cut/QUBO instances.
The MQLib dataset consists of several groups of max-cut/QUBO instances from various domains, including the famous Gset\footnote{https://web.stanford.edu/{\raise.17ex\hbox{$\scriptstyle\sim$}}yyye/yyye/Gset/} graphs.
In the following, each instance is formulated as a max-cut problem on a weighted graph admitting negative weights.
For a graph $G=(V, E)$ with weights $w=(w_{uv})_{(u,v) \in E} \in \mathbb R^E$ on edges, the max-cut problem asks to find a cut $S\subset V$ maximizing the cut value
}
\begin{align}
    C(S) \coloneqq \sum_{u\in S} \sum_{v \in V\setminus S} w_{uv}, 
\end{align}
\hl{
where we define $w_{uv} \coloneqq 0$ for $(u,v)\notin E$ and $w_{uv} \coloneqq w_{vu}$ for $(v,u) \in E$.
By assigning a binary variable $\sigma = (\sigma_v)_{v\in V} \in \{\pm 1\}^V$ on vertices, the max-cut problem is formulated as
}
\begin{align}
    \max_{\sigma \in \{\pm 1\}^V} \sum_{(u,v)\in E} \frac{1-\sigma_u \sigma_v}{2} w_{uv}.
\end{align}
\hl{
The max-cut problem can be reduced to a QUBO problem by substituting $\sigma_v = 2x_v - 1$ with binary variables $x_v\in \{0,1\}, v\in V$.
Therefore, we apply the proposed method after the conversion, and evaluate solutions in terms of the cut values.
}

\hl{
The procedure of the experiment is as follows.
First, for every problem instance, we apply the roof-duality method (see Appendix~\mbox{\ref{app:review_preprocessing}}) as a preprocessing.
The aim of this is to test the applicability of the proposed method on more realistic settings where standard preprocessing is applied.
Then we extract valid partial orders by executing Algorithm~\mbox{\ref{alg:extract_order}} on all instances.
We collect ``linearizable'' instances on which a non-empty edge set of valid partial order is found.
We sample 10 solutions for each of those instances using the Ising machine with and without linearization and compare the cut values of the solutions.
Next, we run minor-embedding search with the minorminer algorithm on the collected instances with and without linearization.
For each obtained embedding, we again sample 10 solutions for each instance using the Ising machine and evaluate their cut values.
}

\hl{
We summarize the basic statistics of problem instances and the number of linearizable instances in Table~\mbox{\ref{tab:mqlib}}.
We label the dataset categories by the file name of each instance.
The ``Others'' category includes instances from various domains from synthetic random graphs to real problems such as graph coloring\footnote{We refer to the MQLib repository for the detailed labeling: https://github.com/MQLib/MQLib}.
After the roof-duality preprocessing, 24 linearizable instances are found in the ``Others'' category.
For comparison, we also run Algorithm~\mbox{\ref{alg:extract_order}} without the preprocessing and report the number of the linearizable instances.
Without the preprocessing, a valid partial order is found on three instances in Gset.
In ``Others'' category, the preprocessing reduces the number of linearizable instances from 106 to 24.
These results indicate that several instances can be linearized but not reduced, while there is some overlap between linearizable instances and instances to be reduced by the roof-duality method.
}

\hl{
We show the results of the Ising machine performance on linearizable instances in Table~\mbox{\ref{tab:mqlib_energy_without_embedding}}.
The 24 linearizable instances distribute on the domains of Prize-collecting Steiner Tree (PCST), Planer random graph (Planer), Isomorphism testing (Isom.), and Graph Coloring (Coloring).
The cut values averaged over 10 solutions are reported with the standard deviation.
On all 24 instances, we do not observe significant differences in the cut values.
We hypothesize that this is because the Ising machine already reaches near-optimal solutions without linearization and there is little room for improvement.
}

\hl{
The results of minor-embedding to the Pegasus graph $P_{16}$ and Zephyr graph $Z_{15}$ are shown in Table~\mbox{\ref{tab:embedding_mqlib}}.
We omit instances that cannot be embedded to either of them.
We also omit instance g000599 since the number of nodes after the preprocessing is too small.
For instances g000812 and g001933, only embedding to $Z_{15}$ is found.
Regarding the quality of minor-embedding, the linearization generally reduces the number of auxiliary variables and the maximum chain length.
In some cases, the number of auxiliary variables increases after linearization even though a linearized connectivity graph is a subgraph of the original graph.
This is because the minorminer algorithm involves randomness and incidentally cannot find better embedding for the linearized problem than the baseline.
For the $Z_{15}$ graph, the cut value is significantly increased on instances g000577, g001328, and g002510 by linearization.
For the $P_{16}$ graph, similar improvement is observed on instance g002510.
%The cut values are not much changed on the other cases.
Overall, we conclude that the proposed method mitigates the performance degradation due to minor-embedding on several instances in MQLib.
}

\end{document}